\def\blue{\textcolor{black}}

\newcommand*{\Scale}[2][4]{\scalebox{#1}{$#2$}}
\newcommand{\Erdos}{Erd\H{o}s-R\'enyi }

\newcommand{\quarter}{{\frac{1}{4}}}

\newcommand{\eps}{{\varepsilon}}

\newcommand{\ttau}{\varsigma}

\newcommand{\limstar}{\lim\star}

\newcommand{\Gmb}{{\mathbb{G}}}

\newcommand{\Mmb}{{\mathbb{M}}}
\newcommand{\Nmb}{{\mathbb{N}}}

\newcommand{\Rmb}{{\mathbb{R}}}

\newcommand{\Umb}{{\mathbb{U}}}

\newcommand{\Smc}{{\mathcal{S}}}

\newcommand{\zero}{{\boldsymbol{0}}}
\newcommand{\one}{{\boldsymbol{1}}}

\newcommand{\qbd}{{\boldsymbol{q}}}

\newcommand{\Xbd}{{\boldsymbol{X}}}

\newcommand{\zbd}{{\boldsymbol{z}}}

\newcommand{\tbar}{{\bar{t}}}
\newcommand{\taubar}{{\bar{\tau}}}

\newcommand{\Btil}{{\tilde{B}}}
\newcommand{\Ctil}{{\tilde{C}}}

\newcommand{\ntil}{{\tilde{n}}}

\newcommand{\varphitil}{{\tilde{\varphi}}}

\newcommand{\ztil}{{\tilde{z}}}
\newcommand{\zetatil}{{\tilde{\zeta}}}

\newcommand{\newbd}{\zbd}
\newcommand{\new}{z}

\documentclass[10pt]{article}
\usepackage[portrait,margin=2.54cm]{geometry}

\usepackage{amsfonts}

\usepackage{amssymb}
\usepackage{mathrsfs}
\usepackage{soul}
\usepackage{hyperref}
\usepackage{amssymb,amsthm,amsmath,amsfonts,amsbsy,latexsym}
\usepackage{graphicx}
\usepackage[numeric,initials,nobysame]{amsrefs}
\usepackage{upref,setspace}
\usepackage{enumerate}
\usepackage{color}

\numberwithin{equation}{section}
\numberwithin{figure}{section}
\numberwithin{table}{section}
\newtheorem{Lemma}{Lemma}[section]
\newtheorem{Proposition}[Lemma]{Proposition}

\newtheorem{Theorem}[Lemma]{Theorem}
\newtheorem{Assumption}[Lemma]{Assumption}
\newtheorem{Conjecture}[Lemma]{Conjecture}

\newtheorem{Remark}[Lemma]{Remark}

\newtheorem{Construction}{Construction}
\newtheorem{Corollary}[Lemma]{Corollary}

\numberwithin{equation}{section}

\definecolor{gold}{RGB}{178,34,34}
\definecolor{bl}{RGB}{34,178,178}
\newcommand{\cgg}{\textcolor{black}}
\newcommand{\cg}{\textcolor{black}}
\begin{document}

\title{Rare event asymptotics for exploration processes for random graphs}
\author{Shankar Bhamidi, Amarjit Budhiraja, Paul Dupuis, Ruoyu Wu}
\maketitle

\begin{abstract}
\noindent Large deviations for random graph models has been a topic of significant recent research activity. \cg{Much work} in this area is focused on the class of \emph{dense} random graph models (number of edges in the graph scale as $n^{2}$, where $n$ is the number of vertices) where the theory of graphons has emerged as a principal tool in the study of large deviation properties. 
\cg{These tools do not give a good approach to
large deviation problems for random graph models in the sparse regime.} The aim of this paper is to \cg{study} an approach for  large deviation problems in this regime by establishing Large Deviation Principles (LDP) on suitable path spaces for certain exploration processes of the associated random graph sequence. Exploration processes are an important tool in the study of sparse random graph models and have been used to understand detailed asymptotics of many functionals of sparse random graphs, such as component sizes, surplus, deviations from trees, etc. In the context of rare event asymptotics of interest here, the point of view of exploration process transforms a large deviation analysis of a static random combinatorial structure to the study of a small noise LDP for certain stochastic dynamical systems with jumps. 
 
Our work focuses on one particular class of random graph models, namely the configuration model; however the general approach of using exploration processes for studying large deviation properties of sparse random graph models has broader applicability. The goal is to study asymptotics of probabilities of non-typical behavior in the large network limit. The first key step for this is to establish a LDP for an exploration process associated with the configuration model. A suitable exploration process here turns out to be an infinite dimensional Markov process with transition probability rates that diminish to zero in certain parts of the state space. Large deviation properties of such Markovian models is challenging due to poor regularity behavior of the associated local rate functions. Our proof of the LDP relies on a representation of the exploration process in terms of a system of stochastic differential equations driven by Poisson random measures and variational formulas for moments of nonnegative functionals of Poisson random measures. Uniqueness results for certain controlled systems of deterministic equations play a key role in the analysis. Next, using the rate function in the LDP for the exploration process we formulate a calculus of variations problem associated with the asymptotics of component degree distributions. The second key ingredient in our study is a careful analysis of the infinite dimensional Euler-Lagrange equations associated with this calculus of variations problem. Exact solutions of these systems of nonlinear differential equations are identified which then provide explicit formulas for decay rates of probabilities of non-typical component degree distributions and related quantities. 
\newline

\noindent


\end{abstract}

%

\section{Introduction}
\label{sec:intro}

 Large deviations for random graph
models has been a topic of significant recent research activity (see, e.g., \cite{chatterjee2011large,chatterjee2016introduction,bordenave2015large,puhalskii2005stochastic,o1998some,choi2013large, DhaSen}). 
\cg{Much work} in this area is focused on the class of \emph{dense}
random graph models (number of edges in the graph scale like $n^{2}$, where $%
n $ is the number of vertices). In this regime, the theory of graphons
obtained under dense graph limits \cite{borgs2008convergent,borgs2012convergent,lovasz2012large, DhaSen} has emerged as a
key tool in the study of large deviation asymptotics. 
In contrast to the
above papers, the focus in the current work is on a \emph{sparse} random
graph setting where the average degree of a typical vertex is $O(1)$ so that
the number of edges in the graph are $O(n)$ as $n\rightarrow \infty $.  
\cg{In this regime  tools based on the theory of graphons do not give a good approach to
the study of large deviation problems.}
The goal of this work is to \cg{study an approach for  large deviation problems} in the sparse regime by establishing
large deviation principles for a class of stochastic dynamical systems, known as the  \emph{exploration processes}, that play a central role in the study of sparse random graphs. \cg{The idea of using stochastic processes to study large deviation problems for static combinatorial objects has been used previously in several works, e.g.\ in \cite{dupnuzwhi} for studying urn models, in
\cite{puhalskii2005stochastic} for studying \Erdos random graphs, in \cite{choset} in the study of preferential attachment model, 
and in  \cite{puhalskii2013number} for another type of attachment model.}
Our work focuses on one particular class of random graph models, namely the  configuration model; however similar techniques  are expected to be useful for other sparse random graph models as well where tractable dynamic constructions via exploration processes are available.
%
%
%

 The \emph{configuration model}  refers to a sequence of random
graphs with number of vertices approaching infinity and the degree
distribution converging to a pre-specified probability distribution $\boldsymbol{p}= \{p_k\}_{k\in \Nmb}$ on the
set of non-negative integers. \cg{This random graph
model is a basic object in probabilistic combinatorics; see \cite{molloy1995critical} where sufficient conditions for the existence of a large connected component in a configuration model were given, which then lead to these types of random graphs being used as models for various real world systems, see e.g.\ \cite{newman2001random} and
\cite{Hofstad2016} and references therein for a comprehensive survey of rigorous results on this  model
(see also  \cite{bender1978asymptotic,bollobas1980probabilistic} where constructions similar to the configuration model were first used to count graphs with a prescribed degree sequence). This model has become one of the standard
workhorses in the study of networks in areas such as epidemiology (see e.g.\ \cite{newman2002spread} where epidemics on graphs with prescribed degree distribution are considered) and community detection (where the configuration model forms the basis of one of the most well known techniques called modularity optimization \cite{newman2006modularity}, \cite[Section 6]{fortunato2010community}). In such applications, after observing a real world system,  the configuration model with the same degree distribution is used as a ``baseline'' model to compare against the real world system to judge the existence of atypical events. Thus an important question in such random graph models is to estimate probabilities of atypical structural behaviors, particularly when the system size is large.
 }

In this paper, we are interested in probabilities of events $E^{n,\eps}(\boldsymbol{q})$ associated with the configuration model random graph $G_n$ on $n$ vertices, described as
\begin{align}
	E^{n,\eps}(\boldsymbol{q}) &= \{\mbox{there exists a component in } G_n \mbox{ with } m_k \mbox{ degree $k$ vertices, where } \nonumber\\
	&\quad \quad\quad\quad\quad \quad m_k \in [n(q_k-\eps), n(q_k+\eps)],\; k \in \Nmb\},\label{mainevent}
\end{align}
and where $\boldsymbol{q} = (q_k)_{k\in \Nmb}$ is such that $0 \le q_k \le p_k$ for every $k$.
One of our main results (see Theorem \ref{thm:ldg_degree_distribution}) shows that, under conditions, for large $n$ and small $\eps$
\begin{equation}\label{eq:mainasymp}
	P\left\{ E^{n,\varepsilon}({\boldsymbol{q}})\right\} \approx \exp\left\{-n\left[H(\boldsymbol{q}) + H(\boldsymbol{p}-\boldsymbol{q}) - H(\boldsymbol{p})\right]\right\},
	\end{equation}
where for a nonnegative sequence $\boldsymbol{r} = (r_k)_{k\in \Nmb}$, 
\begin{equation}\label{eq:hdefn}
H(\boldsymbol{r})\doteq \sum_{k=1}^{\infty }{r}_{k}\log {%
r}_{k}-\left( \frac{1}{2}\sum_{k=1}^{\infty }k{r}_{k}\right)
\log \left( \frac{1}{2}\sum_{k=1}^{\infty }k{r}_{k}\right).
\end{equation}%
This result in particular gives asymptotics for probabilities of observing a component of a given size (see Remark \ref{rem:conjsize}) and explicit formulas for rates of decay of probabilities of observing a $D$-regular component of a given size in $G_n$ (see Corollaries \ref{cordreg} and \ref{cordregsubgraph}); see also Conjectures \ref{conj:Dreg-multi} and \ref{conj:Dreg-max} on large deviation 
asymptotics for the size of the largest component in a $D$-regular graph.


\cg{In order to prove Theorem \ref{thm:ldg_degree_distribution} we first study a more general and abstract problem of large deviations for a certain class of stochastic dynamical systems in Theorem \ref{thm:main-ldp}.}
The starting point  is a dynamical construction of the
configuration model given through a discrete time infinite dimensional
Markov chain referred to as the {\em exploration process} (cf. \cite{MolloyReed1998size,Janson2009new}). As the name suggests,
the exploration process is constructed by first appropriately selecting a
vertex in the graph and then exploring the neighborhood of the chosen vertex
until the component of that vertex is exhausted. After this one moves on to
another `unexplored' vertex resulting in successive exploration of
components of the random graph until the entire graph has been explored. The
stochastic process corresponding to one particular coordinate of this
infinite dimensional Markov chain encodes the number of edges in any given
component through the length of its excursions away from zero. The remaining
coordinates of this Markov chain can be used to read off the number of
vertices of a given degree in any given component of the random graph. See Section \ref{sec:eea}
for a precise description of the state space of this Markov chain.
The
exploration process can be viewed as a small noise stochastic dynamical
system in which the transition steps are of size $O(1/n)$ with $n$ denoting the
number of vertices in the random graph. 
A key ingredient in the proof of Theorem \ref{thm:ldg_degree_distribution}, is a Large Deviation Principle (LDP) for an infinite dimensional jump-Markov process 
that can be viewed as a continuous time analogue of the exploration process.   This result, given in Theorem \ref{thm:main-ldp}, is our second main result. As other applications of this theorem, we recover a well known result on the asymptotics of the largest component in the configuration model due to Molloy and Reed \cite{MolloyReed1998size} and Janson and Luczak \cite{Janson2009new}, and also present  a result (whose proof is omitted) on asymptotics of scaled number of components in a configuration model (see Remark \ref{rem:remnumcopm}).
The rate function in the LDP given in Theorem \ref{thm:main-ldp} can be used to formulate a calculus of variations problem associated with the event  $E^{n,\eps}(\boldsymbol{q})$ described in \eqref{mainevent}.
This problem is at the heart of our analysis and by studying the corresponding infinite dimensional system of coupled Euler-Lagrange equations we construct an explicit minimizer in this optimization problem (see Lemma \ref{lem:minimizer-verify-general}). The cost associated with the minimizer is the exponent on the right side of \eqref{eq:mainasymp} and provides the exact expression for the decay rate for the probability of interest.

\subsection{Proof techniques and overview of contributions}
In addition to the study of the asymptotics of the configuration model, one of the main motivations for working on these sets of problems was the development of new techniques for handling large deviations for processes with ``degeneracies.'' We will give an overview of these contributions in this section. 

 The exploration process
associated with the $n$-th random graph (with $n$ vertices) in the
configuration model is described as an $\mathbb{R}^{\infty }$-valued `small
noise' Markov chain $\{\boldsymbol{X}^{n}(j)\}_{j\in \mathbb{N}_{0}}$. Under our
assumptions, there exists a $N\in \mathbb{N}$ such that for all $%
j\geq nN$, $\boldsymbol{X}^{n}(j)=\boldsymbol{0}$ for all $n\in \mathbb{N}$. 
In order
to study large deviations for such a sequence, one usually considers a
sequence of continuous times processes, or equivalently $\mathbb{C}([0,N]:%
\mathbb{R}^{\infty })$-valued random variables, obtained by a linear
interpolation of $\{\boldsymbol{X}^{n}(j)\}_{j\in \mathbb{N}_{0}}$ over intervals of
length $1/n$. A large deviations analysis of such a sequence in the current
setting is challenging due to `diminishing rates' feature of the transition
kernel (see \eqref{eq:aj-dynamics}) which in turn leads to poor regularity
of the associated local rate function. By diminishing rates we mean the
property that probabilities of certain transitions, although non-zero, can
get arbitrarily close to $0$ as the system becomes large. In the model we
consider, the system will go through phases where some state transitions have
very low probabilities, that are separated by phases of `regular behavior,'
many times. In terms of the underlying random graphs the first type of
phase corresponds to time periods in the dynamic construction that are close
to the completion of exploration of one component and beginning of
exploration of a new component. The poor regularity of the local rate
function makes standard approximations of the near optimal trajectory that
are used in proofs of large deviation principles for such small noise
systems hard to implement. In order to overcome these difficulties we
instead consider a different continuous time process associated with the
exploration of the configuration model. This continuous time process is
obtained by introducing i.i.d.\ exponential random times before each step in the
edge exploration Markov chain. A precise description of this process is
given in terms of stochastic differential equations (SDE) driven by a
countable collection of Poisson random measures (PRM), where different PRMs
are used to describe the different types of transitions (see Section \ref%
{sec:model}). Although the coefficients in this SDE are discontinuous
functions, their dependence on the state variable is much more tractable
than the state dependence in the transition kernel of the discrete time
model.

Large deviations for small noise SDE driven by Brownian motions have been
studied extensively both in finite and infinite dimensions. An approach
based on certain variational representations for moments of nonnegative
functionals of Brownian motions and weak convergence methods \cite%
{BoueDupuis1998variational, BudhirajaDupuis2000variational} has been quite
effective in studying a broad range of such systems (cf. references in \cite{BudhirajaDupuisMaroulas2011variational}). A similar variational
representation for functionals of a Poisson random measure has been obtained
in \cite{BudhirajaDupuisMaroulas2011variational}. There have been several
recent papers that have used this representation for studying large
deviation problems (see, e.g., \cite{BudhirajaChenDupuis2013large,BudhirajaDupuisGanguly2015moderate, BudhirajaWu2017moderate}). This representation is the starting point of the analysis in  the current work as well, however the application of the representation to the setting considered here leads to new challenges. One
key challenge that arises in the proof of the large deviations lower bound
can be described as follows. The proof of the lower bound based on
variational representations and weak convergence methods, for systems driven
by Brownian motions, requires establishing unique solvability of controlled
deterministic equations of the form 
\begin{equation}
dx(t)=b(x(t))dt+\sigma (x(t))u(t)dt,\;x(0)=x_{0},  \label{eq:condifdet}
\end{equation}%
where $u\in L^{2}([0,T]:\mathbb{R}^{d})$ (space of square integrable
functions from $[0,T]$ to $\mathbb{R}^{d}$) is a given control. It turns out
that the conditions that are typically introduced for the well-posedness of
the original small noise stochastic dynamical system of interest (e.g.\
Lipschitz properties of the coefficients $b$ and $\sigma $) are enough to
give the wellposedness of \eqref{eq:condifdet}. For example when the
coefficients are Lipschitz, one can use a standard argument based on
Gronwall's lemma and an application of the Cauchy-Schwarz inequality to
establish the desired uniqueness property. In contrast, when studying
systems driven by a PRM one instead needs to establish wellposedness of
controlled equations of the form 
\begin{equation}
x(t)=x(0)+\int_{[0,t]\times S}1_{[0,g(x(s))]}(y)\varphi
(s,y)ds\,m(dy),\;0\leq t\leq T,  \label{eq:contdetpoi}
\end{equation}%
where $S$ is a locally compact metric space, $m$ a locally finite measure on 
$S$, $g \colon \mathbb{R}\rightarrow \mathbb{R}_+$ is a measurable map and the
control $\varphi $ is a nonnegative measurable map on $[0,T]\times S$ which
satisfies the integrability property 
\begin{equation*}
\int_{\lbrack 0,T]\times S}\ell (\varphi (s,y))ds\,m(dy)<\infty ,
\end{equation*}%
where $\ell (x)=x\log x-x+1$. If $\varphi $ were uniformly bounded and $g$
sufficiently regular (e.g., Lipschitz) uniqueness follows once more by a
standard Gronwall argument. However, in general if $g$ is not Lipschitz or $%
\varphi $ is not bounded (both situations arise in the problem considered
here, see e.g.\ \eqref{eq:psi}-\eqref{eq:phi_k}) the problem of uniqueness
becomes a challenging obstacle. One of the novel contributions of this work
is to obtain uniqueness results for equations of the form %
\eqref{eq:contdetpoi} when certain structural properties are satisfied. The
setting we need to consider is more complex than the one described above in
that there is an infinite collection of coupled equations (one of which
corresponds to the Skorokhod problem for one dimensional reflected
trajectories) that describe the controlled system. However the basic
difficulties can already be seen for the simpler setting in %
\eqref{eq:contdetpoi}. Although for a general $\varphi $ the unique
solvability of equations of the form \eqref{eq:contdetpoi} may indeed be
intractable, the main idea in our approach is to argue that one can perturb
the original $\varphi $ slightly so that the solution $x(\cdot )$ stays the same and moreover this $x(\cdot)$ is the unique solution
of the corresponding equation with the perturbed $\varphi $. Furthermore the
cost difference between the original and perturbed $\varphi $ is
appropriately small. The uniqueness result given in Lemma \ref%
{lem:uniqueness} is a key ingredient in the proof of the lower bound given
in Section \ref{sec:lower}. The proof of the upper bound, via the weak
convergence based approach to large deviations relies on establishing suitable tightness and limit characterization
results for certain controlled versions of the original small noise system.
This proof is given in Section \ref{sec:upper}.

The rate function in the LDP for the exploration process in Theorem \ref{thm:main-ldp}  is given as a variational formula on an infinite dimensional path space (see \eqref{eq:rate_function}).
Getting useful information from such an abstract formula in general seems hopeless, however, as we show in this work, for the event  considered in \eqref{mainevent}, the variational formula 
can be used to extract much more explicit information. We begin by observing (see \eqref{eq:enverps}) that the event $E^{n,\eps}(\boldsymbol{q})$ of interest can be written explicitly 
in terms of the exploration process. Using this and the LDP in Theorem \ref{thm:main-ldp} one can provide an upper bound for the probability of the event in terms of a quantity
$I^2_{0,\tau}((0,\boldsymbol{p}), (0,\boldsymbol{p} - \boldsymbol{q}))$ which can be interpreted (see Section \ref{sec:cal} for a precise definition) as the minimal cost 
for certain controlled analogues of the exploration process to move from the state $(0,\boldsymbol{p})$ to $(0,\boldsymbol{p} - \boldsymbol{q})$ in $\tau$ units of time, where
$\tau = \frac{1}{2} \sum_{k=1}^\infty kq_k$ (see Lemmas \ref{lem:puhalskii-upper-bound} and \ref{lem:puhalskii-upper-bound-improvement}). 
We then show that this deterministic control problem, which can be reformulated as a calculus of variations problem, admits an explicit solution. This solution is given in Construction \ref{cons:cont}
and its optimality is studied in Lemma \ref{lem:minimizer-verify-general}. Using this optimality property, the complementary lower bound for the probability of interest is given in 
Lemma \ref{lem:puhalskii-lower-bound}.  Lemmas \ref{lem:minimizer-verify-general} and \ref{lem:puhalskii-lower-bound} form the technical heart of the proof of Theorem \ref{thm:ldg_degree_distribution} and rely on a detailed and careful analysis of the infinite dimensional Euler-Lagrange equations associated with the calculus of variations problem.

\subsection{Organization of the paper}

The paper is organized as follows. In Section \ref{sec:assuandres} we
introduce the configuration model, our main assumptions, and our first main result, Theorem \ref{thm:ldg_degree_distribution}, on asymptotics of probabilities of $E^{n,\eps}(\boldsymbol{q})$.
We record some consequences of these results for $D$-regular graphs and subgraphs in Corollaries \ref{cordreg} and \ref{cordregsubgraph}. Remark \ref{rem:conjsize}
discusses another application of this result to the study of asymptotics of probabilities of components of a given size.  In Section \ref{sec:eea} we review the edge-exploration algorithm (EEA)
from \cite{MolloyReed1998size,Janson2009new} that gives a dynamical construction of the configuration model. For reasons discussed previously, the large deviation analysis 
of the discrete time EEA presents several technical obstacles and thus in Section \ref{sec:model} we introduce a closely related continuous time jump-Markov process 
$(\boldsymbol{X}^n, Y^n)$ with values in $(\Rmb\times \Rmb_+^{\infty})\times \Rmb$ which is mathematically more tractable. Sections \ref{sec:rate_def} and \ref{sec:main-result}
present our second main result, Theorem \ref{thm:main-ldp}, that gives a large deviation principle for the sequence $(\boldsymbol{X}^n, Y^n)_{n \in \Nmb}$ in a suitable infinite dimensional
path space.  In Section \ref{sec:main-result} we also note two side consequences of Theorem \ref{thm:main-ldp}. The first, given in Section \ref{sec:LLN}
is a law of large numbers (LLN) result that recovers well known results of Janson and Luczak (2009) on the asymptotics of the largest component in the configuration model. The second, discussed in Remark \ref{rem:remnumcopm}, gives a LDP for the scaled number of components in $G_n$ as $n\to \infty$. 

Section \ref{sec:repnWCCP} presents the variational representation from \cite%
{BudhirajaDupuisMaroulas2011variational} for functionals of PRM that is the
starting point of our proofs. Some tightness and characterization results
that are used both in the upper and lower bound proofs are also given in
this section. Next, Section \ref{sec:upper} gives the proof of the large
deviation upper bound whereas the proof of the lower bound is given in
Section \ref{sec:lower}. Finally, Section \ref{sec:rate_function}
establishes the compactness of level sets of the function $I_{T}$ defined in
Section \ref{sec:main-result}, thus proving that $I_{T}$ is a rate function.
Together, results of Sections \ref{sec:upper}, \ref{sec:lower} and \ref{sec:rate_function} complete the proof of Theorem \ref{thm:main-ldp}.

We next turn to the proof of Theorem \ref{thm:ldg_degree_distribution} which is given in Sections \ref{sec:cal}-\ref{sec:pfsect4}. First in Section \ref{sec:cal}
we introduce a calculus of variations problem that is central to the proof of Theorem \ref{thm:ldg_degree_distribution}. We also introduce (see Construction \ref{cons:cont})
a candidate minimizer in this optimization problem and present several technical results (Lemmas \ref{lem:I1I2L}--\ref{lem:minimizer-verify-general}) that are needed for the proof 
of the optimality property of the candidate minimizer. Using results of Section \ref{sec:cal} the proof of Theorem \ref{thm:ldg_degree_distribution} is completed in Section \ref{sec:pf_LDP_degree}.
Finally, Section \ref{sec:pfsect4} contains the proofs of technical lemmas from Section \ref{sec:cal} whereas Section \ref{sec:examples} presents the proof of the LLN results from Section \ref{sec:LLN}.

\subsection{Notation}

The following notation will be used. For a Polish space $\mathbb{S}$, denote
the corresponding Borel $\sigma$-field by $\mathcal{B}(\mathbb{S})$. 
Denote by $\mathcal{P}(\mathbb{S})$ (resp.\ $\mathcal{M}(\mathbb{S})$) the
space of probability measures (resp. finite measures) on $\mathbb{S}$,
equipped with the topology of weak convergence. Denote by $\mathbb{C}_b(%
\mathbb{S})$ (resp.\ $\mathbb{M}_b(\mathbb{S})$) the space of real bounded
and continuous functions (resp.\ bounded and measurable functions). For $f
\colon \mathbb{S} \to \mathbb{R}$, let $\|f\|_\infty \doteq \sup_{x \in 
\mathbb{S}} |f(x)|$. 
For a Polish space $\mathbb{S}$ and $T>0$, denote by $\mathbb{C}([0,T]:\mathbb{S})$
(resp.\ $\mathbb{D}([0,T]:\mathbb{S})$) the space of continuous functions
(resp.\ right continuous functions with left limits) from $[0,T]$ to $%
\mathbb{S}$, endowed with the uniform topology (resp.\ Skorokhod topology). 
\cg{We recall that a collection $\{ X^n \}$ of $\mathbb{S}$-valued random variables on some probability space $(\Omega, \mathcal{F}, P)$ is
said to be tight, if for each $\varepsilon>0$ there is a compact set $K \subset \mathbb{S}$ such that $\sup_{n}P(X^n \in K^c)\le \varepsilon.$}
A sequence of $\mathbb{D}([0,T]:\mathbb{S})$-valued random variables is said
to be $\mathcal{C}$-tight if it is tight in $\mathbb{D}([0,T]:\mathbb{S})$
and every weak limit point takes values in $\mathbb{C}([0,T]:\mathbb{S})$
a.s. We use the symbol `$\Rightarrow$' to denote convergence in
distribution. 


We denote by $\mathbb{R}^{\infty}$ the space of all real sequences which is
identified with the countable product of copies of $\mathbb{R}$. This space
is equipped with the usual product topology. For ${\boldsymbol{x}}=(x_k)_{k
\in \mathbb{N}}, {\boldsymbol{y}}=(y_k)_{k \in \mathbb{N}}$, we write ${%
\boldsymbol{x}} \le {\boldsymbol{y}}$ if $x_k \le y_k$ for each $k \in 
\mathbb{N}$. We will use the notation $a\doteq b$ to signify that the definition of $a$ is given by the quantity $b$.
 Let $\mathcal{C} \doteq \mathbb{C}([0,T]:\mathbb{R})$, $%
\mathcal{C}_\infty \doteq \mathbb{C}([0,T]:\mathbb{R}^\infty)$, $\mathcal{D}
\doteq \mathbb{D}([0,T]:\mathbb{R})$, $\mathcal{D}_\infty \doteq \mathbb{D}%
([0,T]:\mathbb{R}^\infty)$. 
Let $x^+ \doteq \max \{x,0\}$ for $x \in \mathbb{R}$.
Denote by $\mathbb{R}_+$ the set of all non-negative real numbers. 
Let $\mathbb{N}_0 \doteq \mathbb{N} \cup \{0\}$.
Cardinality of a set $A$ is denoted by $|A|$. For $n \in \mathbb{N}$, let $%
[n] \doteq \{1,2,\dotsc,n\}$.  We use the following conventions: $0\log 0=0$, $0 \log (x/0) =0$ for $x\ge 0$, and  $x \log (x/0) = \infty$ for $x>0$. 

\section{Assumptions and Results}

\label{sec:assuandres} Fix $n \in \mathbb{N}$. We start by describing the
construction of the configuration model of random graphs 
with vertex set $[n]$. Detailed description and further references for the configuration model can be found in \cite[Chapter 7]{Hofstad2016}.

\subsection{The configuration model and assumptions}

Let ${\boldsymbol{d}}(n)=\{d_{i}^{(n)}\}_{i\in \lbrack n]}$ be a degree
sequence, namely a sequence of non-negative integers such that $%
\sum_{i=1}^{n}d_{i}^{(n)}$ is even. Let $2m^{(n)}\doteq
\sum_{i=1}^{n}d_{i}^{(n)}$. We will usually suppress the dependence of $%
d_{i}^{(n)}$ and $m^{(n)}$ on $n$ in the notation. Using the sequence $%
\{d_{i}\}$ we construct a random graph on $n$ labelled vertices $[n]$ as follows: (i)
Associate with each vertex $i\in \lbrack n]$ $d_{i}$ \emph{half-edges}. (ii)
Perform a uniform random matching on the $2m$ half-edges to form $m$ edges
so that every edge is composed of two half-edges. This procedure creates a
random multigraph $G([n],{\boldsymbol{d}}(n))$ with $m$ edges, allowing for
multiple edges between two vertices and self-loops, and is called the \emph{configuration model}
with degree sequence ${\boldsymbol{d}}(n)$. Since we are concerned with
connectivity properties of the resulting graph, vertices with degree zero
play no role in our analysis, and therefore we assume that $d_{i}>0$ for all 
$i\in \lbrack n],~n\geq 1$. We make the following additional assumptions.

\begin{Assumption}
\label{asp:convgN} There exists a probability distribution ${\boldsymbol{p}}%
\doteq \left\{ p_{k}\right\} _{k\in \mathbb{N}}$ on $\mathbb{N}$ such that,
writing $n_{k}^{\scriptscriptstyle(n)}\doteq |\left\{ i\in \lbrack n]:d_{i}=k\right\} |$ for the number of vertices with degree $%
k $, 
$
{n_{k}^{\scriptscriptstyle(n)}}/{n}\rightarrow p_{k}\mbox{ as }%
n\rightarrow \infty ,\mbox{ for all }k\in \mathbb{N}.
$
\end{Assumption}

We will also usually suppress the dependence of $n_{k}^{\scriptscriptstyle(n)}$ on $n$ in the notation.
We make the following assumption on moments of the degree distribution.

\begin{Assumption}
\label{asp:exponential-boundN} There exists some $\varepsilon_{\boldsymbol{p}%
} \in (0,\infty)$ such that $\sup_{n \in \mathbb{N}} \sum_{k=1}^{\infty} 
\frac{n_k}{n}k^{1+\varepsilon_{\boldsymbol{p}}} < \infty$. 
\end{Assumption}

The above two assumptions will be made throughout this work.

\begin{Remark}
\label{rem1.1}
\begin{enumerate}[\upshape (i)]
\item Note that Assumptions \ref{asp:convgN} and \ref{asp:exponential-boundN}%
, along with Fatou's lemma, imply that $\sum_{k=1}^{\infty
}p_{k}k^{1+\varepsilon _{\boldsymbol{p}}}<\infty $. Conversely, if $%
\sum_{k=1}^{\infty }p_{k}k^{\lambda }<\infty $ for some $\lambda \in
(4,\infty )$ and $\{D_{i}\}_{i\in \mathbb{N}}$ is a sequence of i.i.d.\ $%
\mathbb{N}$-valued random variables with common distribution $%
\{p_{k}\}_{k\in \mathbb{N}}$, then using a Borel--Cantelli argument it can be
shown that for a.e.\ $\omega $, Assumptions \ref{asp:convgN} and \ref%
{asp:exponential-boundN} are satisfied with $d_{i}=D_{i}(\omega )$, $i\in
\lbrack n]$, $n\in \mathbb{N}$, and $\varepsilon _{\boldsymbol{p}}=\frac{%
\lambda }{4}-1$.
\item Under Assumptions \ref{asp:convgN} and \ref{asp:exponential-boundN}, $%
\mu\doteq\sum_{k=1}^\infty kp_k < \infty$ and the total number of edges $m = \frac{1}{2} \sum_{i=1}^n d_i$ satisfies $\frac{m}{n} \to \frac{1}{2} \sum_{k=1}^\infty kp_k$ as $n \to \infty$.
\end{enumerate}
\end{Remark}

\subsection{Large Deviation Asymptotics for Component Degree Distributions}

\label{sec:degree_distribution}

We will say that a component of $G([n],{\boldsymbol{d}}(n))$ has degree configuration $\{\bar{n}_{k}\}$ if  the component has $\bar{n}_{k}$ vertices with degree $k$, for $k \in \Nmb$.
Given $\boldsymbol{0} \le {\boldsymbol{q}}=(q_{k},k\in \mathbb{N})\leq {\boldsymbol{p}}$, 
we are interested in the asymptotic exponential rate of decay of the probability of the
event $E^{n,\varepsilon }({\boldsymbol{q}})$ introduced in \eqref{mainevent} that corresponds to the existence 
of a component in $G([n],{\boldsymbol{d}}(n))$
with degree configuration $\{\bar{n}_{k}\}$
  satisfying $(q_{k}-\varepsilon )n\leq \bar{n}_{k}\leq
(q_{k}+\varepsilon )n$, $k\in \mathbb{N}$,
namely, we want to characterize $\lim_{\varepsilon \rightarrow
0}\lim_{n\rightarrow \infty }\frac{1}{n}\log {P}\left\{
E^{n,\varepsilon }({\boldsymbol{q}})\right\} $. 
Note that for there to exist a component with degree configuration $\{nq_k\}$ we must have
$\sum_{k=1}^{\infty }kq_{k}\ge 2\left(\sum_{k=1}^{\infty }q_{k}-\frac{1}{n}\right).$
We will in fact assume a slightly stronger condition: 
\begin{equation}\label{eq:slightstr}
\sum_{k=1}^{\infty }kq_{k}>2\sum_{k=1}^{\infty }q_{k}.
\end{equation}
This condition says that there are strictly more edges than vertices
in the component. Define $\beta \doteq \beta ({\boldsymbol{q}})$ as
follows: $\beta =0$ when $q_{1}=0$, and when $q_{1}>0$, $\beta \in (0,1)$ is
the unique solution (see Remark \ref{rmk:uniqueness_beta} below) of the equation 
\begin{equation*}
\sum_{k=1}^{\infty }kq_{k}=(1-\beta ^{2})\sum_{k=1}^{\infty }\frac{kq_{k}}{%
1-\beta ^{k}}.
\end{equation*}%
Define the function $K({\boldsymbol{q}})$ by 
\begin{equation}
K({\boldsymbol{q}})\doteq \left( \frac{1}{2}\sum_{k=1}^{\infty
}kq_{k}\right) \log (1-\beta ({\boldsymbol{q}})^{2})-\sum_{k=1}^{\infty
}q_{k}\log (1-\beta ({\boldsymbol{q}})^{k})\label{eq:kdefnn}
\end{equation}%
and with $H(\cdot)$ as in \eqref{eq:hdefn}
define
\begin{equation}
	\label{eq:Itil_1}
	{\tilde{I}}_{1}({\boldsymbol{q}})\doteq H({\boldsymbol{q}})+H({	\boldsymbol{p}}-{\boldsymbol{q}})-H({\boldsymbol{p}})+K({\boldsymbol{q}}).
\end{equation}

\begin{Remark}
\label{rmk:uniqueness_beta} The existence and uniqueness of $\beta ({%
\boldsymbol{q}})$ can be seen as follows. For $\alpha \in (0,1)$ consider%
\begin{equation*}
\alpha F(\alpha )\doteq \sum_{k=1}^{\infty }kq_{k}-(1-\alpha
^{2})\sum_{k=1}^{\infty }\frac{kq_{k}}{1-\alpha ^{k}}=\alpha \left(
\sum_{k=3}^{\infty }\frac{\alpha -\alpha ^{k-1}}{1-\alpha ^{k}}%
kq_{k}-q_{1}\right). 
\end{equation*}%
For $k\geq 3$ and $\alpha \in (0,1)$ let
$
F_{k}(\alpha )\doteq (\alpha -\alpha ^{k-1})/(1-\alpha ^{k}).
$
It is easily verified that $F_{k}(\cdot )$ is strictly increasing on $(0,1)$. Thus for 
$\alpha \in (0,1)$, $0=F_{k}(0+)<F_{k}(\alpha )<F_{k}(1-)=\frac{k-2}{k}$,
and so
\begin{equation*}
-q_{1}=F(0+)<F(\alpha )<F(1-)=\sum_{k=3}^{\infty }(k-2)q_{k}-q_{1}.
\end{equation*}%
Since $F$ is continuous on $(0,1)$, $-q_{1}<0$ and $\sum_{k=3}^{\infty
}(k-2)q_{k}-q_{1}=\sum_{k=1}^{\infty }kq_{k}-2\sum_{k=1}^{\infty }q_{k}>0$,
we have the existence and uniqueness of $\beta ({\boldsymbol{q}})$.
\end{Remark}

\begin{Remark}
	\label{rem:finhk}
	We note that for every $\boldsymbol{0} \le {\boldsymbol{q}}=(q_{k},k\in \mathbb{N})\leq {\boldsymbol{p}}$, $K(\boldsymbol{q})$ and $H(\boldsymbol{q})$ are finite. Indeed, the finiteness of $K(\boldsymbol{q})$ is immediate from
	Assumption \ref{asp:exponential-boundN}. To see the finiteness of  $H(\boldsymbol{q})$, note that  on the one hand $\sum_{k=1}^{\infty }{q}_{k}\log {q}_{k} \le 0$ while on the other hand
	\begin{align*}
		\sum_{k=1}^{\infty} q_k \log q_k & = \sum_{k=1}^{\infty} q_k \log \frac{q_k}{2^{-(k+1)}} - (\log 2) \sum_{k=1}^{\infty} (k+1) q_k \\
		& \ge -\left(1-\sum_{k=1}^\infty q_k\right) \log \frac{1-\sum_{k=1}^\infty q_k}{2^{-1}}- (\log 2) \sum_{k=1}^{\infty} (k+1) q_k  > -\infty,
	\end{align*}
	where the first inequality follows from non-negativity of relative entropy and putting mass $1-\sum_{k=1}^\infty q_k$ on $k=0$, and the last inequality once more uses Assumption \ref{asp:exponential-boundN}.
\end{Remark}

The following result gives asymptotics of the event $E^{n,\varepsilon }({%
\boldsymbol{q}})$. The proof of the theorem, which is based on Theorem \ref%
{thm:main-ldp}, is given in Section \ref{sec:pf_LDP_degree}.

\begin{Theorem}
\label{thm:ldg_degree_distribution} 
Suppose $\boldsymbol{0} \le {\boldsymbol{q}} \le {%
\boldsymbol{p}}$ and that \eqref{eq:slightstr} is satisfied.
Then
\begin{enumerate}[(i)]

\item (Upper bound) when $p_1=0$, we have $\beta({\boldsymbol{q}})=0$, $K({%
\boldsymbol{q}})=0$ and 
\begin{equation*}
\limsup_{\varepsilon \to 0} \limsup_{n \to \infty } \frac{1}{n} \log {%
P}\left\{ E^{n,\varepsilon}({\boldsymbol{q}})\right\} \le -{%
\tilde{I}}_1({\boldsymbol{q}}).
\end{equation*}

\item (Lower bound) 
\begin{equation*}
\liminf_{\varepsilon \to 0} \liminf_{n \to \infty } \frac{1}{n} \log {%
P}\left\{ E^{n,\varepsilon}({\boldsymbol{q}})\right\} \ge -{%
\tilde{I}}_1({\boldsymbol{q}}).
\end{equation*}
\end{enumerate}
\end{Theorem}
\begin{Remark}
The proof of Theorem \ref{thm:ldg_degree_distribution} relies on a  large deviation principle for a certain exploration process (see Section \ref{sec:model}) that is given in Theorem \ref{thm:main-ldp}. The latter result does not require the condition $p_1 = 0$. Also note that the lower bound in Theorem \ref{thm:ldg_degree_distribution} does not require the condition $p_1=0$ either. One can also give an upper bound (without requiring $p_1=0$) in terms of a variational formula given by the right side of \eqref{eq:upperbd_to_improve}. When $p_1=0$, this variational expression can be simplified and is seen to be equal to $-{\tilde{I}}_1({\boldsymbol{q}})$. This is shown in Lemma \ref{lem:puhalskii-upper-bound-improvement} whose proof crucially relies on the property $p_1=0$.
Whether the two expressions are equal in general when $p_1\neq 0$ remains an open problem. 
\end{Remark}
As an immediate corollary of Theorem \ref{thm:ldg_degree_distribution}  we have the following result for
$D$-regular graphs, i.e., graphs such that each vertex is of degree $D$. In the following $\limstar$ represents either $\limsup$ or $\liminf$.
\begin{Corollary}
	\label{cordreg}
	{\bf ($D$-regular graphs)}
Suppose that
	there exists some $D \in \mathbb{N}$ with $D \ge 3$, such that $p_k = 0$, $n_k = 0$ for $k \ne D$ and $p_D = 1$, $n_D=n$.
	Fix $q_D \in (0,1]$ and 
	denote by $E^{n,\varepsilon}_D(\boldsymbol{q})$ the event that there is a component of size $N_D \in [n(q_D-\eps), n(q_D+\eps)]$. Then	
	\begin{equation}
		\label{eq:Dreg}
		\limstar_{\varepsilon \to 0} \limstar_{n \to \infty } \frac{1}{n} \log P \left\{ E^{n,\varepsilon}_D({\boldsymbol{q}})\right\} =
		\left(1-\frac{D}{2}\right) \left( q_D \log q_D + (1-q_D) \log (1-q_D) \right).
	\end{equation}
\end{Corollary}

\begin{proof}
	Let $q_k=0$ for $k \in \mathbb{N}\setminus\{D\}$ and let ${\boldsymbol{q}}=\{q_{k},k\in \mathbb{N}\}$.
	Then
	since $p_1=0$, we have $\beta(\boldsymbol{q})=0$ and $K({\boldsymbol{q}})=0$.
	Using \eqref{eq:Itil_1} we have
	\begin{align*}
		{\tilde{I}}_{1}({\boldsymbol{q}}) & = H({\boldsymbol{q}})+H({	\boldsymbol{p}}-{\boldsymbol{q}})-H({\boldsymbol{p}})+K({\boldsymbol{q}}) \\
		& = q_D \log q_D - \frac{Dq_D}{2} \log \left( \frac{Dq_D}{2} \right) + (1-q_D) \log (1-q_D) - \frac{D-Dq_D}{2} \log \left( \frac{D-Dq_D}{2} \right) \\
		& \qquad + \frac{D}{2} \log \left( \frac{D}{2} \right) \\
		& = \left(1-\frac{D}{2}\right) \left( q_D \log q_D + (1-q_D) \log (1-q_D) \right).
	\end{align*}
	The result then follows from Theorem \ref{thm:ldg_degree_distribution}.
\end{proof}

We note that the expression \eqref{eq:Dreg} has the same form when $q_D$ is replaced by $1-q_D$.
This suggests that the most likely way of having a component of size around $nq_D$ in $D$-regular graphs is to let almost all of the remaining $n(1-q_D)$ vertices be in one component.
Indeed, conditioning on having a component of size around $nq_D$, the remaining vertices can be viewed as a smaller configuration model of $D$-regular graphs with about $n(1-q_D)$ vertices.
It then follows from the well known results for the asymptotics of the largest component in the configuration model \cite{MolloyReed1998size,Janson2009new} (and Theorem \ref{thm:LLN}) that these remaining vertices are in one component with high probability.

Based on these observations we make the following conjecture.

\begin{Conjecture}
	\label{conj:Dreg-multi}
	{\bf ($D$-regular graphs, multiple components)}
	Suppose that there exists some $D \in \mathbb{N}$ with $D \ge 3$, such that $p_k = 0$, $n_k = 0$ for $k \ne D$ and $p_D = 1$, $n_D=n$.
	Fix $M \in \Nmb$ and $q_D^{(i)} \in (0,1]$ for each $i=1,\dotsc,M$, such that $\sum_{i=1}^M q_D^{(i)} \le 1$. 
	Let $q_k^{(i)}=0$ for $k \in \mathbb{N}\setminus\{D\}$ and let ${\boldsymbol{q}^{(i)}}=\{q_{k}^{(i)},k\in \mathbb{N}\}$, for each $i=1,\dotsc,M$.
	Let $\boldsymbol{q}^{(M+1)} = \boldsymbol{p}-\sum_{i=1}^M \boldsymbol{q}^{(i)}$.
	Denote by $E^{n,\varepsilon,M}_D$ the event that there are components of sizes $N_D^{(i)} \in [n(q_D^{(i)}-\eps), n(q_D^{(i)}+\eps)]$, $i=1,\dotsc,M$. Then	
	\begin{align*}
		\limstar_{\varepsilon \to 0} \limstar_{n \to \infty } \frac{1}{n} \log P \left\{ E^{n,\varepsilon,M}_D\right\} & = \sum_{i=1}^{M+1} H(\boldsymbol{q}^{(i)}) - H(\boldsymbol{p}) =
		\left(1-\frac{D}{2}\right) \sum_{i=1}^{M+1} q_D^{(i)} \log q_D^{(i)}.
	\end{align*}
\end{Conjecture}

We also note that for each fixed $a \in [0,1]$, the function $[0,a] \ni x \mapsto x\log x + (a-x)\log(a-x) \in (-\infty,0]$ is maximized at $x=0$ and $x=a$.
This suggests that, the most likely way for the largest component to be of certain size, is to let as many of the remaining components as possible have such a size.
Based on this we make the following conjecture on the large deviation behavior of the largest component size for D-regular graphs.

\begin{Conjecture}
	\label{conj:Dreg-max}
	{\bf ($D$-regular graphs, largest component)}
	Suppose that there exists some $D \in \mathbb{N}$ with $D \ge 3$, such that $p_k = 0$, $n_k = 0$ for $k \ne D$ and $p_D = 1$, $n_D=n$.
	For each $x \in [0,1]$, let $q_D^{(x)} = x$, $q_k^{(x)}=0$ for $k \in \mathbb{N}\setminus\{D\}$, and ${\boldsymbol{q}^{(x)}}=\{q_{k}^{(x)},k\in \mathbb{N}\}$.
	Denote by $M^n$ the size of the largest component.
	Then $\frac{M^n}{n}$ satisfies a large deviation principle in $\Rmb_+$ with rate function $I_{max}$ defined by
	\begin{align*}
		I_{max}(x) & = k(x)H(\boldsymbol{q}^{(x)}) + H(\boldsymbol{q}^{(1-xk(x))}) - H(\boldsymbol{p}) =
		\left(1-\frac{D}{2}\right) \left( xk(x) \log x + (1-xk(x)) \log \left(1-xk(x)\right) \right)
	\end{align*}
	for $x \in [0,1]$ and $I_{max}(x)=\infty$ otherwise, where $k(x)=\lfloor \frac{1}{x} \rfloor$ is the largest integer such that $xk(x) \le 1$.
\end{Conjecture}


Recall that $\mu\doteq\sum_{k=1}^\infty kp_k < \infty$.
The following result gives bounds on probabilities of observing a $D$-regular subgraph in a configuration model with a general degree sequence $(p_k)$.

\begin{Corollary}
	\label{cordregsubgraph}
	Suppose that Assumptions \ref{asp:convgN} and \ref{asp:exponential-boundN} hold.
	Also suppose that $p_1=0$.
	Fix $D \in \Nmb$ with $D \ge 3$ such that $p_D > 0$.
	Fix $q_D \in (0, p_D]$.  Denote by $E^{n,\varepsilon}(q)$ the event that the graph has a component that is $D$-regular and has size
	$N_D \in [n(q_D-\eps), n(q_D+\eps)]$.
	Then 
	\begin{align*}
		&\limstar_{\varepsilon \to 0} \limstar_{n \to \infty } \frac{1}{n} \log P \left\{ E^{n,\varepsilon}({\boldsymbol{q}})\right\}\\
		 &\quad = \left(q_D \log q_D + (p_D-q_D) \log (p_D-q_D) - p_D \log p_D\right) \\
		& \quad\quad - \left( \frac{Dq_D}{2} \log \left( \frac{Dq_D}{2} \right) + \frac{\mu-Dq_D}{2} \log \left( \frac{\mu-Dq_D}{2} \right) - \frac{\mu}{2} \log \left( \frac{\mu}{2} \right) \right).
	\end{align*}
\end{Corollary}

\begin{proof}
	Let $q_k=0$ for $k \in \mathbb{N}\setminus\{D\}$ and let ${\boldsymbol{q}}=(q_{k},k\in \mathbb{N})$.
	As before, since $q_1=0$, we have $\beta(\boldsymbol{q})=0$ and $K({\boldsymbol{q}})=0$.
	Using \eqref{eq:Itil_1} we have
	\begin{align*}
		{\tilde{I}}_{1}({\boldsymbol{q}}) & = H({\boldsymbol{q}})+H({	\boldsymbol{p}}-{\boldsymbol{q}})-H({\boldsymbol{p}})+K({\boldsymbol{q}}) \\
		& = q_D \log q_D - \frac{Dq_D}{2} \log \left( \frac{Dq_D}{2} \right) \\
		& \qquad + (p_D-q_D) \log (p_D-q_D) + \sum_{k \ne D} p_k \log p_k - \frac{\mu-Dq_D}{2} \log \left( \frac{\mu-Dq_D}{2} \right) \\
		& \qquad - \sum_{k=1}^\infty p_k \log p_k + \frac{\mu}{2} \log \left( \frac{\mu}{2} \right) \\
		& = \left(q_D \log q_D + (p_D-q_D) \log (p_D-q_D) - p_D \log p_D\right) \\
		& \qquad - \left( \frac{Dq_D}{2} \log \left( \frac{Dq_D}{2} \right) + \frac{\mu-Dq_D}{2} \log \left( \frac{\mu-Dq_D}{2} \right) - \frac{\mu}{2} \log \left( \frac{\mu}{2} \right) \right).
	\end{align*}
	The result then follows from Theorem \ref{thm:ldg_degree_distribution}.
\end{proof}

\begin{Remark}
	\label{rem:conjsize}
Theorem \ref{thm:ldg_degree_distribution} can be used to extract other asymptotic results.  We give below one example without proof.
	Suppose that Assumptions \ref{asp:convgN}
	and \ref{asp:exponential-boundN} hold. Also suppose that $p_1=p_2=0$.
	  Let $r\in (0,1]$ and denote by $E^{n,\varepsilon}_r$ the event that the graph has a component that has size
	$N_r \in [n(r-\eps), n(r+\eps)]$.
	Then
	\begin{align*}
		\limstar_{\varepsilon \to 0} \limstar_{n \to \infty } \frac{1}{n} \log P \left\{ E^{n,\varepsilon}_r\right\} & = 
		\inf_{0\le \boldsymbol{q} \le \boldsymbol{p}:\; \boldsymbol{q}\cdot \boldsymbol{1} =r }\left\{H({\boldsymbol{p}})-H({\boldsymbol{q}}) - H({	\boldsymbol{p}}-{\boldsymbol{q}})\right\}.
	\end{align*}

\end{Remark}

\begin{Remark}
	\color{black}
	There is an important connection between the configuration model and the uniform distribution on the space of all simple graphs (namely graphs which have no multiple edges and self-loops) with a prescribed degree distribution which we now describe. 
	 Given a degree sequence $\mathbf{d}(n)$, let $\Gmb([n], \mathbf{d}(n))$ be the set of all (simple) graphs on vertex set $[n]$ with degree sequence $\mathbf{d}(n)$. Let $\Umb\Mmb_n(\mathbf{d}(n))$ denote the uniform measure on $\Gmb([n], \mathbf{d}(n))$.  Then as is well known (see e.g.\ \cite[Proposition 7.15]{Hofstad2016}), the configuration model satisfies the property that the 
	 conditional distribution of $G([n], \mathbf{d}(n))$, given the event that $G([n], \mathbf{d}(n))$  is simple, is 
	 $\Umb\Mmb_n(\mathbf{d}(n))$.
	Further by \cite{janson2009probability}, under the assumptions of the current paper $P(G([n], \mathbf{d}(n)) \text{ is simple}) \to e^{-(\nu/2 +\nu^2/4)}$ where $\nu = \sum_{k} k(k-1)p_k/\sum_k kp_k$.
	These observations suggest a natural approach to asymptotic questions of the form studied in the current work for 
	(simple) graphs with a prescribed degree distribution. In particular by an elementary Bayes formula calculation it follows that
	if 
	\begin{equation}\label{eq:bayes}
		\frac{\log P(G([n], \mathbf{d}(n)) \text{is simple} \,\big|\, E^{n,\varepsilon}({\boldsymbol{q}}))}{n}\to 0,\end{equation}
	then Theorem \ref{thm:ldg_degree_distribution}  will continue to hold with the configuration model replaced with the uniform distribution on the space of simple graphs with prescribed degree sequence.  In general, characterizing the asymptotics of quantities as in \eqref{eq:bayes}  is key to the large deviation analysis of  $\Umb\Mmb_n(\mathbf{d}(n))$.  Study of these questions is deferred to future work.
	
\end{Remark}


\subsection{Edge-exploration algorithm (EEA)}

\label{sec:eea}

\cgg{Given a degree sequence ${\boldsymbol{d}}(n)$, we now describe a well known dynamic construction of the configuration model $G([n],{\boldsymbol{d}}(n))$ given in \cite{MolloyReed1998size,Janson2009new} by sequentially matching half-edges. Tracking functionals of this dynamic construction, in particular hitting times of zero of the number of so-called active edges (see below) reveals component size information of $G([n],{\boldsymbol{d}}(n))$. Construction given below closely follows  \cite{Janson2009new}.}
This algorithm traverses the graph by exploring all its edges, unlike
typical graph exploration algorithms, which sequentially explore vertices.
At each stage of the algorithm, every vertex in $[n]$ is in one of two possible
states, sleeping or awake, while each half-edge is in one of three states:
sleeping (unexplored), active or dead (removed). \cgg{The exploration process sequentially visits vertices, {\bf awakening} vertices whilst {\bf activating} or {\bf killing} half-edges. 
} 

Write $\mathcal{S}_{\mathbb{V}}(j)$ for the set of  sleeping
vertices at step $j$ and similarly let $\mathcal{S}_{\mathbb{E}}(j),\mathcal{%
A}_{\mathbb{E}}(j)$ be the set of sleeping and
active  half-edge at step $j$. We call a half-edge
\textquotedblleft living\textquotedblright\ if it is either sleeping or
active. Initialize by setting all vertices and half-edges to be in the
sleeping state. For step $j\geq 0$, write $A(j)\doteq |\mathcal{A}_{\mathbb{E%
}}(j)|$ for the number of active half-edges and $V_{k}(j)$ for the number of
sleeping vertices $v\in \mathcal{S}_{\mathbb{V}}(j)$ with degree $k$. Write $\boldsymbol{V}(j)\doteq(V_{k}(j), k\in \mathbb{N})$
for the corresponding vector in $\mathbb{R}_{+}^{\infty }$. 
At step $j=0$, all vertices and half-edges are asleep hence $A(0)=0$ and $%
V_k(0)=n_k$ for $k \ge 1$. The exploration process proceeds as follows:

\begin{enumerate}[\upshape(1)]

\item If the number of active half-edges and sleeping vertices is zero, i.e.\ $A(j) = 0$
and $\boldsymbol{V}(j)={\boldsymbol{0}}$, all vertices and half-edges have been explored
and we terminate the algorithm.

\item If $A(j) = 0$ and $\boldsymbol{V}(j) \neq {\boldsymbol{0}}$, so there exist
sleeping vertices, pick one such vertex with probability proportional to its
degree and
mark the vertex as awake and all its half-edges as active. 
Thus the
transition $(A(j),\boldsymbol{V}(j))$ to $(A(j+1),\boldsymbol{V}(j+1))$ at step $j+1$ takes the form 
\begin{equation*}
(0, {\boldsymbol{v}}) \mapsto (k, {\boldsymbol{v}} - {\boldsymbol{e}}_k) %
\mbox{ with probability } \frac{kv_k}{\sum_{i=1}^\infty i v_i}, \: k \in 
\mathbb{N},
\end{equation*}
where ${\boldsymbol{e}}_k$ is the $k$-th unit vector.

\item If $A(j) > 0$, pick an active half-edge uniformly at random, pair it
with another uniformly chosen living half-edge (either active or sleeping),
say $e^*$, merge both half-edges to form a full edge and kill both
half-edges. If $e^*$ was sleeping when picked, wake the vertex corresponding
to the half-edge $e^*$, and mark all its other half-edges active. Thus in
this case the transition takes the form 
\begin{align*}
(a, {\boldsymbol{v}}) &\mapsto (a-2, {\boldsymbol{v}}) 
\mbox{ with
probability } \frac{a-1}{\sum_{i=1}^\infty i v_i + a-1}, \\
(a, {\boldsymbol{v}}) &\mapsto (a+k-2, {\boldsymbol{v}} - {\boldsymbol{e}}%
_k) \mbox{ with probability } \frac{kv_k}{\sum_{i=1}^\infty i v_i + a-1}, \:
k \in \mathbb{N}.
\end{align*}
\end{enumerate}

The statements in (2) and (3) can be combined as follows: If $A(j) \neq 0$
or $\boldsymbol{V}(j) \neq {\boldsymbol{0}}$, then the transition $(A(j),\boldsymbol{V}(j))$ to $%
(A(j+1),\boldsymbol{V}(j+1))$ takes the form 
\begin{align}  \label{eq:aj-dynamics}
\begin{aligned} (a, {\boldsymbol{v}}) &\mapsto (a-2 \blue{\cdot 1_{\{a > 0\}}}, {\boldsymbol{v}})
\mbox{ with probability } \frac{(a-1)^+}{\sum_{i=1}^\infty i v_i + (a-1)^+},
\\ 
(a, {\boldsymbol{v}}) &\mapsto (a+k-2 \blue{\cdot 1_{\{a > 0\}}}, {\boldsymbol{v}} -
{\boldsymbol{e}}_k) \mbox{ with probability } \frac{kv_k}{\sum_{i=1}^\infty
i v_i + (a-1)^+}, \: k \in \mathbb{N}. \end{aligned}
\end{align}

The random graph $G([n],{\boldsymbol{d}}(n))$ formed at the termination of the above algorithm has the same distribution as the
configuration model with degree sequence ${\boldsymbol{d}}(n)$ \cite%
{molloy1995critical,Janson2009new}.

\begin{Remark}
	\label{rmk:track-component}
	We note that for $j> 0$, $A(j)=0$ if and only if the exploration of a
	component in the random graph $G([n],{\boldsymbol{d}}(n))$ is completed at
	step $j$. Thus the number of edges in a component equals the length of an
	excursion of $\{A(j)\}$ away from $0$ and the largest excursion length gives
	the size of the largest component, namely the number of edges in the
	component with maximal number of edges. 
	The vertices in each component are those that are awakened during
	corresponding excursions.
\end{Remark}

Note that at each step in the EEA, either a new vertex is woken up or two
half-edges are killed. Since there are a total of $n$ vertices and $2m$
half-edges, we have from Assumptions \ref{asp:convgN} and \ref%
{asp:exponential-boundN} that the algorithm terminates in at most $m+n\le n
L $ steps where $L \doteq 1 + \lfloor \sup_n \frac{1}{2}\sum_{k=1}^\infty k 
\frac{n_k}{n} \rfloor < \infty$. We define $A(j) \equiv 0$ and $\boldsymbol{V}(j)
\equiv {\boldsymbol{0}}$ for all $j \ge j_0$ where $j_0$ is the step at
which the algorithm terminates.

\subsection{An equivalent continuous time exploration process}

\label{sec:model} A natural way to study large deviation properties of the
configuration model is through the discrete time sequence $%
\{A(j),\boldsymbol{V}(j)\}_{j\in \mathbb{N}_0}$ in EEA which can be viewed as a
discrete time \textquotedblleft small noise" Markov process. In order to
study large deviations for such a sequence, a standard approach is to
consider the sequence of $\mathbb{C}([0,L]:\mathbb{R}^{\infty})$-valued
random variables obtained by a linear interpolation of $\{A(j),\boldsymbol{V}(j)\}_{j\in 
\mathbb{N}_0}$ over intervals of length $1/n$. As was noted in the
Introduction, the `diminishing rates' feature of the transition kernel %
\eqref{eq:aj-dynamics} makes the large deviations analysis of this sequence
challenging. 
\cg{An alternative approach is to consider a continuous time stochastic process that provides a tractable
construction of the configuration model. We briefly recall one such construction that was introduced in
\cite[Section 4]{Janson2009new}. }

\color{black}
\subsubsection{A simple continuous time construction}
\label{sec:cts-constr-eea}
In 
 \cite[Section 4]{Janson2009new} it was observed that the configuration model can be explored using a continuous time process  constructed using exponential random variables as follows. 
\begin{enumerate}
	\item Every half-edge $e$ is given an independent exponential life-time (call this a clock). Initially, all half-edges and vertices are taken to be sleeping. 
	\item  Whenever the clock of a half-edge rings this half-edge becomes awake and connects to an existing awake  half-edge if such a half-edge exists; otherwise it waits for the next half-edge clock to ring and connects to this half-edge completing a full edge. Both such half-edges are then called dead. If at any point a half-edge of a sleeping vertex awakes, that vertex is then said to be awake.
	\item The process continues until all half-edges are dead at which point the exploration ends.
\end{enumerate}
It is observed in \cite[Section 4]{Janson2009new} that the random graph constructed at the end of the exploration is a realization from the desired configuration model.  

Although the above continuous time construction gives a simple method to produce a sample from the configuration model, it turns out to be hard to directly use it for the study of large deviation problems of interest here. In view of this we present
 below a different continuous time process for the exploration of the configuration model that is obtained by a more direct Poissonization of the Markov chain
 $(A(\cdot), \boldsymbol{V}(\cdot))$ in Section \ref{sec:eea}. 

\subsubsection{A continuous time construction via Poissonization}
Let $N(t)$ be a rate-$n$ Poisson process independent of the processes $(A, \boldsymbol{V})$ of Section \ref{sec:eea} and define $(\tilde A(t), \tilde{\boldsymbol{V}}(t)) \doteq (A(N(t)),  \boldsymbol{V}(N(t))$. 
Then $(\tilde A, \tilde{\boldsymbol{V}})$ gives a natural continuous time process associated with the exploration of the configuration model. We now give a distributionally equivalent representation of this process which is more tractable for a large deviation analysis. The construction given below  ensures that 
$\{(nX^n_0(\cdot)+1, nX^n_k(\cdot)), k\in \mathbb{N}\}$, where $X^n_j$ are processes defined below,  has the same distribution as
$\{\tilde A(\cdot), \tilde V_k(\cdot), k\in \mathbb{N}\}$.

\color{black} 

We begin with some notation that will be needed to formulate the
continuous time model. For a locally compact Polish space $\mathbb{S}$, let $%
\mathcal{M}_{FC}(\mathbb{S})$ be the space of all measures $\nu$ on $(%
\mathbb{S},\mathcal{B}(\mathbb{S}))$ such that $\nu(K)<\infty$ for every
compact $K \subset \mathbb{S}$. We equip $\mathcal{M}_{FC}(\mathbb{S})$ with
the usual vague topology. This topology can be metrized such that $\mathcal{M%
}_{FC}(\mathbb{S})$ is a Polish space (see  \cite%
{BudhirajaDupuisMaroulas2011variational} for one convenient metric). A Poisson random measure (PRM) $%
N$ on a locally compact Polish space $\mathbb{S}$ with intensity measure $\nu \in \mathcal{M}_{FC}(\mathbb{S})$ is an $%
\mathcal{M}_{FC}(\mathbb{S})$-valued random variable such that for each $A
\in \mathcal{B}(\mathbb{S})$ with $\nu(A)<\infty$, $N(A)$ is
Poisson distributed with mean $\nu(A)$ and for disjoint $A_1,\dotsc,A_k \in 
\mathcal{B}(\mathbb{S})$, $N(A_1),\dotsc,N(A_k)$ are
mutually independent random variables (cf.\ \cite{IkedaWatanabe1990SDE}).

Let $(\Omega ,\mathcal{F},{P})$ be a complete probability space on which  are given
 i.i.d.\ PRM $\{N%
_{k}(ds\,dy\,dz)\}_{k\in \mathbb{N}_{0}}$ on $\mathbb{R}_{+}\times \lbrack 0,1]\times 
\mathbb{R}_{+}$ with intensity measure $ds\times dy\times dz$. Let
\begin{equation*}
\hat{\mathcal{F}}_{t}\doteq \sigma \{N_{k}((0,s]\times A\times
B),0\leq s\leq t,A\in \mathcal{B}([0,1]),B\in \mathcal{B}(\mathbb{R}%
_{+}),k\in \mathbb{N}_{0}\},\;t\geq 0
\end{equation*}%
and let $\{\mathcal{F}_{t}\}$ be the ${P}$-completion of this
filtration. Fix $T\in (0,\infty )$. Let $\mathcal{\bar{P}}$ be the $\{%
\mathcal{F}_{t}\}_{0\leq t\leq T}$-predictable $\sigma $-field on $\Omega
\times \lbrack 0,T]$. Let $\bar{\mathcal{A}}_{+}$ be  all $(%
\mathcal{\bar{P}}\otimes \mathcal{B}([0,1]))/\mathcal{B}(\mathbb{R}_{+})$%
-measurable maps from $\Omega \times \lbrack 0,T]\times \lbrack 0,1]$ to $%
\mathbb{R}_{+}$. For $\varphi \in \bar{\mathcal{A}}_{+}$, define a counting
process $N_{k}^{\varphi }$ on $[0,T]\times \lbrack 0,1]$ by 
\begin{equation*}
N_{k}^{\varphi }([0,t]\times A)\doteq \int_{\lbrack 0,t]\times A\times 
\mathbb{R}_{+}}{{1}}_{[0,\varphi (s,y)]}(z)\,N%
_{k}(ds\,dy\,dz),\:t\in \lbrack 0,T],A\in \mathcal{B}([0,1]),k\in \mathbb{N}%
_{0}.
\end{equation*}%
We think of $N_{k}^{\varphi }$ as a controlled random measure, where $%
\varphi $ is the control process that produces a thinning of the point
process $N_{k}$ in a random but non-anticipative manner to produce
a desired intensity. We will write $N_{k}^{\varphi }$ as $N_{k}^{\theta }$
if $\varphi \equiv \theta $ for some constant $\theta \in \mathbb{R}_{+}$.
Note that $N_{k}^{\theta }$ is a PRM  on $[0,T]\times [0,1]$ with intensity $\theta ds \times dy$.
For ${\boldsymbol{x}}=(x_{0},x_{1},x_{2},\dotsc )\in \mathbb{R}\times 
\mathbb{R}_{+}^{\infty }$, let 
\begin{equation}
r({\boldsymbol{x}}) \doteq (x_{0})^{+} + \sum_{k=1}^{\infty }kx_{k}, \quad 
r_{0}({\boldsymbol{x}}) \doteq \frac{(x_{0})^{+}}{r({\boldsymbol{x}})} {{1}}_{\{\blue{r({\boldsymbol{x}}) \in (0,\infty)}\}}, \quad 
r_{k}({\boldsymbol{x}}) \doteq \frac{kx_{k}}{r({\boldsymbol{x}})} {{1}}_{\{\blue{r({\boldsymbol{x}}) \in (0,\infty)}\}}, \quad k\in \mathbb{N}.  \label{eq:r_k}
\end{equation}%
Note that $\sum_{k \in \mathbb{N}_{0}} r_k(\boldsymbol{x}) =1$ whenever $\blue{r({\boldsymbol{x}}) \in (0,\infty)}$.
Recall that ${\boldsymbol{e}}_{k}$ is the $k$-th unit vector in $\mathbb{R}%
^{\infty }$, $k\in \mathbb{N}_{0}$. Define the state process $%
\boldsymbol{X}^{n}(t)=(X_{0}^{n}(t),X_{1}^{n}(t),X_{2}^{n}(t),\dotsc )$ with values in $%
\mathbb{R}\times \mathbb{R}_{+}^{\infty }$ as the solution to the following
SDE: 
\begin{equation}\label{eq:eq903e}
\begin{aligned}
\boldsymbol{X}^{n}(t)=\boldsymbol{X}^{n}(0)& +\frac{1}{n}\int_{[0,t]\times \lbrack 0,1]}{{1}%
}_{\{X_{0}^{n}(s-)\geq 0\}}\left[ -2{\boldsymbol{e}}_{0}\right] {{%
1}}_{[0,r_{0}(\boldsymbol{X}^{n}(s-)))}(y)\,N_{0}^{n}(ds\,dy) \\
\quad & +\sum_{k=1}^{\infty }\frac{1}{n}\int_{[0,t]\times \lbrack 0,1]}{%
{1}}_{\{X_{0}^{n}(s-)\geq 0\}}\left[ (k-2){\boldsymbol{e}}_{0}-{%
\boldsymbol{e}}_{k}\right] {{1}}_{[0,r_{k}(\boldsymbol{X}^{n}(s-)))}(y)%
\,N_{k}^{n}(ds\,dy) \\
\quad & +\sum_{k=1}^{\infty }\frac{1}{n}\int_{[0,t]\times \lbrack 0,1]}{%
{1}}_{\{X_{0}^{n}(s-)<0\}}\left[ k{\boldsymbol{e}}_{0}-{%
\boldsymbol{e}}_{k}\right] {{1}}_{[0,r_{k}(\boldsymbol{X}^{n}(s-)))}(y)%
\,N_{k}^{n}(ds\,dy),
\end{aligned}
\end{equation}
where $\boldsymbol{X}^{n}(0)\doteq \frac{1}{n}(-1,n_{1},n_{2},\dotsc )$. The
existence and uniqueness of solutions to this SDE follows from
the summability of $r_k(\cdot)$. 
\cg{
Indeed, for each $\boldsymbol{z} \in \mathbb{R}\times 
\mathbb{R}_{+}^{\infty }$ and $u \in [0,T]$, the process
$$Z^n(u,\boldsymbol{z}, t) \doteq \frac{1}{n} \int_{(u,t]\times [0,1]} N^n_0(ds\, dy) + \sum_{k=1}^{\infty }\frac{1}{n}\int_{(u,t]\times \lbrack 0,1]}
 {{1}}_{[0,r_{k}(\boldsymbol{z}))}(y) N^n_k(ds\, dy),\; u<t\le T$$
satisfies $Z^n(u,\boldsymbol{z}, T)<\infty$ since $\sum_{k \in \mathbb{N}_{0}} r_k(\boldsymbol{z}) \le1$. 
Together with the mutual independence of the PRM $\{N%
_{k}(ds\,dy\,dz)\}_{k\in \mathbb{N}_{0}}$ this says that the jump instants of the point process $\{Z^n(u,\boldsymbol{z}, t)\}_{u< t \le T}$ can be enumerated as
$$u  <\tau^n_1(\boldsymbol{z})< \cdots \tau^n_{k_n}(\boldsymbol{z})<T$$
where $k_n = nZ^n(u,\boldsymbol{z}, T)$. Thus having constructed the solution of \eqref{eq:eq903e} on $[0,u]$, the solution can be extended to 
$[0, \tau^n_1(\boldsymbol{z})]$, where $\boldsymbol{z} = \boldsymbol{X}^{n}(u)$, and the unique solution of \eqref{eq:eq903e}  is  now obtained by a standard recursive construction from one jump instant to the next. The solution can be written in an explicit form in terms of the atoms  of the PRM $\{N^n_k\}$ which also shows that the solution is a measurable function of the driving PRM.
}
It is not difficult to see that $\frac{1}{n}%
(A(j)-1,V_{1}(j),V_{2}(j),\dotsc )$ in the discrete time EEA can be viewed
as the embedded Markov chain associated with $\boldsymbol{X}^{n}$.
\cg{Namely, denoting the jump instants of the process $\boldsymbol{X}^{n}$ as $\{\sigma^n_j\}$, the collection
$\{(nX^n_0(\sigma_j^n)+1, nX^n_k(\sigma_j^n)), k,j \in \mathbb{N}\}$ has the same distribution as
$\{A(j), V_k(j), k,j \in \mathbb{N}\}$. In particular, for $k \in \mathbb{N}$, $nX^n_k(\sigma_j^n)$ can be interpreted as the number of sleeping vertices with degree $k$ at the $j$-th step of the exploration in the discrete EEA and in view of Remark \ref{rmk:track-component}, the excursions of $X^n_0$ away from $-1/n$ track the components in the configuration model.}
In defining the state process, one could replace $X^n_0(0)$ with the asymptotically equivalent process $X^n_0(0)+1/n$ which starts from $0$ and is more directly comparable with the sequence $A(j)/n$. However some of the expressions are simplified (see, e.g., the formulas for rates in \eqref{eq:r_k} and the transition probabilities in \eqref{eq:aj-dynamics})
 when describing the state in terms of $X^n_0(0)$ instead of $X^n_0(0)+1/n$.
We now rewrite the evolution of $\boldsymbol{X}^{n}$ as follows: 
\begin{align*}
\boldsymbol{X}^{n}(t)& =\boldsymbol{X}^{n}(0)+{\boldsymbol{e}}_{0}\sum_{k=0}^{\infty }\frac{(k-2)}{n}%
\int_{[0,t]\times \lbrack 0,1]}{{1}}_{[0,r_{k}(\boldsymbol{X}^{n}(s-)))}(y)%
\,N_{k}^{n}(ds\,dy) \\
& \quad -\sum_{k=1}^{\infty }{\boldsymbol{e}}_{k}\frac{1}{n}%
\int_{[0,t]\times \lbrack 0,1]}{{1}}_{[0,r_{k}(\boldsymbol{X}^{n}(s-)))}(y)%
\,N_{k}^{n}(ds\,dy) \\
& \quad +{\boldsymbol{e}}_{0}\sum_{k=0}^{\infty }\frac{2}{n}%
\int_{[0,t]\times \lbrack 0,1]}{{1}}_{\{X_{0}^{n}(s-)<0\}}{%
{1}}_{[0,r_{k}(\boldsymbol{X}^{n}(s-)))}(y)\,N_{k}^{n}(ds\,dy).
\end{align*}%
Here the first two integrands do not depend on the sign of $X_{0}^{n}$ and
are interpreted as the main contribution to the evolution. The last sum is a
`reflection' term in the ${\boldsymbol{e}}_{0}$ direction and makes a
contribution of $\frac{2}{n}{\boldsymbol{e}}_{0}$ only when $X_{0}^{n}(s-)<0$%
. For $t\geq 0$ define%
\begin{align}
Y^{n}(t)& \doteq X_{0}^{n}(0)+\sum_{k=0}^{\infty }\frac{k-2}{n}%
\int_{[0,t]\times \lbrack 0,1]}{{1}}_{[0,r_{k}(\boldsymbol{X}^{n}(s-)))}(y)%
\,N_{k}^{n}(ds\,dy),  \label{eq:Y_n} \\
\eta ^{n}(t)& \doteq \sum_{k=0}^{\infty }\frac{2}{n}\int_{[0,t]\times
\lbrack 0,1]}{{1}}_{\{X_{0}^{n}(s-)<0\}}{{1}}%
_{[0,r_{k}(\boldsymbol{X}^{n}(s-)))}(y)\,N_{k}^{n}(ds\,dy).  \label{eq:eta_n}
\end{align}%
Using these we can write 
\begin{align}
X_{0}^{n}(t)& =Y^{n}(t)+\eta ^{n}(t),  \label{eq:X_n_0} \\
X_{k}^{n}(t)& =X_{k}^{n}(0)-\frac{1}{n}\int_{[0,t]\times \lbrack 0,1]}{%
{1}}_{[0,r_{k}(\boldsymbol{X}^{n}(s-)))}(y)\,N_{k}^{n}(ds\,dy),\:k\in \mathbb{N}%
.  \label{eq:X_n_k}
\end{align}%
Here $\eta ^{n}$ is viewed as the regulator function which ensures that $%
X_{0}^{n}(t)\geq -\frac{1}{n}$. Note that for $k \in \mathbb{N}$, $X_{k}^{n}(t)$ is non-increasing
and non-negative. Also, from \eqref{eq:eq903e} we see that $r(\boldsymbol{X}^{n}(t))$ is non-increasing.

\subsection{Rate Function}

\label{sec:rate_def}

The main result of this work gives a large deviation principle for $\{(\boldsymbol{X}^n,
Y^n)\}_{n \in \mathbb{N}}$ in the path space $\mathcal{D}_\infty\times 
\mathcal{D}$. In this section we define the associated rate function $I_T$, \blue{where the subscript $T$ makes explicit the fact that the processes  $\{(\boldsymbol{X}^n,
Y^n)\}_{n \in \mathbb{N}}$ are considered on the time horizon $[0,T]$}.
Including the process $Y^n$ in the LDP is convenient for obtaining large
deviation results, for the degree distribution in giant components, of the
form given in Section \ref{sec:examples}. 

Recall the probability distribution ${\boldsymbol{p}}\doteq \{p_{k}\}_{k\in 
\mathbb{N}}$ introduced in Assumption \ref{asp:convgN}. 
\cg{In order to describe the rate function it will be convenient to introduce the Skorohod map. The use of Skorohod reflection mechanism to describe exploration processes for random graphs goes back to the work of Aldous \cite{aldous1997brownian}.
In the context of large deviation problems for \Erdos random graph models it has also been used in \cite{puhalskii2005stochastic}.}
Let $\Gamma \colon 
\mathcal{C}\rightarrow \mathcal{C}$ denote the one-dimensional Skorokhod map
defined by 
\begin{equation*}
\Gamma (\psi )(t)\doteq \psi (t)-\inf_{0\leq s\leq t}\psi (s)\wedge 0,\;t\in
\lbrack 0,T],\psi \in \mathcal{C}.
\end{equation*}%
%
%
%

Let $\mathcal{C}_T$ be the subset of $\mathcal{C}_\infty \times \mathcal{C}$%
, consisting of those functions $(\boldsymbol{\zeta},\psi)$ 
such that

\begin{enumerate}[\upshape(a)]

\item $\psi(0)=0$, and $\psi$ is absolutely continuous on $[0,T]$.

\item $\zeta_0(t) = \Gamma(\psi)(t)$ for $t \in [0,T]$.

\item For each $k \in \mathbb{N}$, $\zeta_k(0) = p_k$, $\zeta_k$ is
non-increasing and absolutely continuous and $\zeta_k(t) \ge 0$ for $t \in
[0,T]$.
\end{enumerate}
For $(\boldsymbol{\zeta} ,\psi )\in (\mathcal{D}_{\infty }\times \mathcal{D})\setminus 
\mathcal{C}_{T}$, define $I_{T}(\boldsymbol{\zeta} ,\psi )\doteq \infty $. 
For $(\boldsymbol{\zeta}
,\psi )\in \mathcal{C}_{T}$, define 
\begin{equation}
I_{T}(\boldsymbol{\zeta} ,\psi )\doteq \inf_{\boldsymbol{\varphi} \in \mathcal{S}_{T}(\boldsymbol{\zeta} ,\psi
)}\left\{ \sum_{k=0}^{\infty }\int_{[0,T]\times \lbrack 0,1]}\ell (\varphi
_{k}(s,y))\,ds\,dy\right\} .  \label{eq:rate_function}
\end{equation}%
%
%
%
Here for $x\geq 0$, 
\begin{equation}
\ell (x)\doteq x\log x-x+1,  \label{eq:ell}
\end{equation}%
and  $\mathcal{S}_{T}(\boldsymbol{\zeta} ,\psi )$ is the set of all sequences of functions $%
\boldsymbol{\varphi} =(\varphi _{k})_{k\in \mathbb{N}_{0}}$, $\varphi _{k}:[0,T]\times
\lbrack 0,1]\rightarrow \mathbb{R}_{+}$, such that 
\begin{align}
\psi (t)& =\sum_{k=0}^{\infty }(k-2)\int_{[0,t]\times \lbrack 0,1]}{%
{1}}_{[0,r_{k}(\boldsymbol{\zeta} (s)))}(y)\,\varphi _{k}(s,y)ds\,dy, \; t \in [0,T]
\label{eq:psi} \\
\zeta _{k}(t)& =p_{k}-\int_{[0,t]\times \lbrack 0,1]}{{1}}%
_{[0,r_{k}(\boldsymbol{\zeta} (s)))}(y)\,\varphi _{k}(s,y)ds\,dy,k\in \mathbb{N}, \; t \in [0,T].
\label{eq:phi_k}
\end{align}%
%
%
%

\begin{Remark}
\label{rmk:property_ODE} Suppose $(\boldsymbol{\zeta}
,\psi )\in \mathcal{C}_{T}$  satisfies \eqref{eq:psi} and %
\eqref{eq:phi_k} for some $\boldsymbol{\varphi} \in \mathcal{S}_{T}(\boldsymbol{\zeta} ,\psi )$. 

\begin{enumerate}[\upshape(a)]

\item From Assumptions \ref{asp:convgN} and \ref{asp:exponential-boundN} it follows that the
following uniform integrability holds: As $K \to \infty$, 
\begin{equation*}
\sup_{0 \le t \le T}\sum_{k=K}^\infty k \zeta_k(t) \le \sum_{k=K}^\infty k
\sup_{0 \le t \le T} \zeta_k(t) = \sum_{k=K}^\infty k p_k \to 0.
\end{equation*}
This in particular says that $r(\boldsymbol{\zeta}(\cdot)) \in \mathcal{C}$, where $%
r(\cdot)$ is defined in \eqref{eq:r_k}.

\item For any $k \in \mathbb{N}$, whenever $\zeta_k(t_k)=0$ for some $t_k \in
[0,T]$, we must have $\zeta_k(t) = 0$ for all $t \in [t_k,T]$. This follows
since  $\zeta_k$ is non-increasing and non-negative for
every $k$.

\item Whenever $r(\boldsymbol{\zeta}(t^*))=0$ for some $t^* \in [0,T]$, we must have from
part (b) that $\zeta_k(t) = 0$ for all $t \in [t^*,T]$ and $k \in \mathbb{N}$%
. This, together with \eqref{eq:psi}, implies that $\psi(\cdot)$ is
non-increasing on the interval $[t^*,T]$. Hence by property (b) of $\mathcal{%
C}_T$, $\zeta_0(t)$ is non-increasing and non-negative for $t \in [t^*,T]$.
Since $\zeta_0(t^*)=0$, we must then have $\zeta_0(t) = 0$ for $t \in [t^*,T]$%
, which means that $\boldsymbol{\zeta}(t) = \boldsymbol{0}$ for $t \in [t^*,T]$. Thus 
whenever such a $t^*$ exists,
$\boldsymbol{\zeta}(t)= \boldsymbol{0}$ after the time instant 
\begin{equation}
\tau_{\boldsymbol{\zeta}} \doteq \inf \{t \in [0,T]:r(\boldsymbol{\zeta}(t))=0\} \wedge T.
\label{eq:defntauphi}
\end{equation}
\end{enumerate}
\end{Remark}

\subsection{LDP and LLN for the Exploration Process}

\label{sec:main-result}

The following LDP is one of our main results and is  key to the  proof of Theorem \ref{thm:ldg_degree_distribution}.

\begin{Theorem}
\label{thm:main-ldp} The function $I_T$ in \eqref{eq:rate_function} is a
rate function on $\mathcal{D}_\infty \times \mathcal{D}$ and the sequence $%
\{(\boldsymbol{X}^n,Y^n)\}_{n \in \mathbb{N}}$ satisfies a large deviation principle in $%
\mathcal{D}_\infty \times \mathcal{D}$ with rate function $I_T$.
\end{Theorem}

\textbf{Outline of the proof:} Due to the equivalence between a large
deviation principle and a Laplace principle, it suffices to show the
following three statements (cf.\ \cite[Section 1.2]{DupuisEllis2011weak} or 
\cite[Section 1.2]{buddupbook}).

\begin{enumerate}[\upshape(1)]

\item Laplace principle upper bound: For all $h\in \mathbb{C}_{b}(\mathcal{D}%
_{\infty }\times \mathcal{D})$,
\begin{equation}
\limsup_{n\rightarrow \infty} \frac{1}{n}\log {{E}}%
e^{-nh(\boldsymbol{X}^{n},Y^{n})}\le - \inf_{(\boldsymbol{\zeta} ,\psi )\in \mathcal{C}_{\infty }\times 
\mathcal{C}}\{I_{T}(\boldsymbol{\zeta} ,\psi )+h(\boldsymbol{\zeta} ,\psi )\}.  \label{eq:lappriupp}
\end{equation}

\item Laplace principle lower bound: For all $h \in \mathbb{C}_b(\mathcal{D}%
_\infty \times \mathcal{D})$, 
\begin{equation}  \label{eq:lapprilow}
\liminf_{n \to \infty} \frac{1}{n} \log {{E}} e^{-nh(\boldsymbol{X}^{n},Y^{n})} \ge -
\inf_{(\boldsymbol{\zeta} ,\psi ) \in \mathcal{C}_\infty\times \mathcal{C}}
\{I_T(\boldsymbol{\zeta} ,\psi )+h(\boldsymbol{\zeta} ,\psi )\}.
\end{equation}

\item $I_T$ is a rate function on $\mathcal{D}_\infty\times \mathcal{D}$:
For each $M \in [0,\infty)$, $\{(\boldsymbol{\zeta} ,\psi ) \in \mathcal{D}_\infty \times 
\mathcal{D} : I_T(\boldsymbol{\zeta} ,\psi ) \le M \}$ is a compact subset of $\mathcal{D}%
_\infty \times \mathcal{D}$.
\end{enumerate}

Statements (1), (2) and (3) will be shown in Sections \ref{sec:upper}, 
 \ref{sec:lower} and  \ref{sec:rate_function}, respectively.
\begin{Remark}
	\label{rem:remnumcopm}
	 As noted above, the LDP  in Theorem \ref{thm:main-ldp} is a key to the proof of  Theorem \ref{thm:ldg_degree_distribution}. In the next subsection we will show how this LDP can be used to easily give a LLN result.
The LDP can be used to establish other asymptotic results as well.  We give one such example without proof below.
	Denote by $C^n$ the number of components in $G([n],{\boldsymbol{d}}(n))$.
	Then $\eta^n$ defined in \eqref{eq:eta_n} can be used to represent  $\frac{C^n}{n}$.
	\cg{Such an observation in the context of \Erdos random graphs was first made in \cite{aldous1997brownian} (see also \cite{puhalskii2005stochastic}).}
	Note that whenever the EEA starts to explore a new component, $X^n_0$ will jump from $-\frac{1}{n}$ and as a result, $\eta^n$ will increase by $\frac{2}{n}$.
	Therefore
	\begin{equation*}
		\frac{C^n}{n} = \sup_{t >0} \frac{\eta^n(t)}{2} = \lim_{T \to \infty} \frac{\eta^n(T)}{2}. 
	\end{equation*}
	Observe from \eqref{eq:X_n_0} that $\eta^n = X^n_0 - Y^n$, and that for large deviation asymptotics one can assume
	that the  EEA  terminates by time $N \doteq 1 + \lfloor \sup_n \frac{1}{2}\sum_{k=1}^\infty k \frac{n_k^{(n)}}{n} \rfloor < \infty$ (see Lemma \ref{lem:choosing-T} and its proof for precise details).
	Using this fact, Theorem \ref{thm:main-ldp}, and the contraction principle one can establish that
	$\frac{C^n}{n}$ satisfies a large deviation principle in $\mathbb{R}_+$ with rate function $\hat{I}$ defined by
		\begin{equation*}
			\hat{I}(x) = \lim_{T \to \infty} \inf_{(\boldsymbol{\zeta},\psi) \in \mathcal{C}_T : \zeta_0(T)-\psi(T) = 2x} I_T(\boldsymbol{\zeta},\psi). 
		\end{equation*}
	The rate function $\hat{I}(x)$ has the following alternative representation.
	\begin{equation*}
	\Scale[0.9]{\hat{I}(x)  = \inf_{(\boldsymbol{\zeta},\psi) \in \mathcal{C}_N : \zeta_0(N)-\psi(N) = 2x, \boldsymbol{\zeta}(N)=\zero} \int_0^N \left[ r_0(\boldsymbol{\zeta}(t)) \ell\left( -\frac{\psi'(t)+\sum_{k=1}^\infty (k-2) \zeta'_k(t)}{2r_0(\boldsymbol{\zeta}(t))} \right) 
	 + \sum_{k=1}^\infty r_k(\boldsymbol{\zeta}(t)) \ell\left(-\frac{\zeta'_k(t)}{r_k(\boldsymbol{\zeta}(t))}\right) \right] dt.} 
	\end{equation*}
\end{Remark}
\subsubsection{Law of large number limits}

\label{sec:LLN} The LDP in Theorem \ref{thm:main-ldp}
can be used to identify the LLN limit $(\boldsymbol{\zeta} ,\psi )$ of
the exploration process $(\boldsymbol{X}^{n},Y^{n})$, which corresponds to the \textit{unique}
pair satisfying $I_{T}(\boldsymbol{\zeta} ,\psi )=0$. In particular we recover well known
results for the asymptotics of the largest component in the configuration
model \cite{MolloyReed1998size,Janson2009new}. We assume the following
strengthened version of Assumption \ref{asp:exponential-boundN}.

\begin{Assumption}
\label{asp:exponential-boundN-S}   
$\sup_{n \in \mathbb{N}} \sum_{k=1}^{\infty} \frac{n_k}{n}k^{2} < \infty$. 
\end{Assumption}
\begin{Remark}
	Under our standing assumptions, namely Assumption \ref{asp:convgN} and \ref{asp:exponential-boundN}, one can show by following the arguments in  Section \ref{sec:repnWCCP} that $\{(\boldsymbol{X}^n,
Y^n)\}_{n \in \mathbb{N}}$ is tight and any weak limit point $(\boldsymbol{\zeta}
,\psi )$ of this sequence is in $\mathcal{C}_{T}$  and satisfies \eqref{eq:psi} and %
\eqref{eq:phi_k} with $\varphi_k =1$ for $k \in \mathbb{N}_0$. However it seems hard to argue the uniqueness of this limiting system of equations without additional conditions. Instead we show that if  Assumption \ref{asp:exponential-boundN} is replaced with the stronger condition in
Assumption \ref{asp:exponential-boundN-S} then there is an explicit trajectory $(\boldsymbol{\zeta}
,\psi )$ for which the rate function vanishes and in fact it is the unique such trajectory. This is the content of  Theorem \ref{thm:LLN}  and
Proposition \ref{prop:uniqueness_LLN}. From these results the LLN follows immediately. Whether the LLN holds under the weaker Assumption \ref{asp:exponential-boundN} is an open problem.
\end{Remark}

Recall $\mu \doteq \sum_{k=1}^{\infty }kp_{k}$ and note that $\mu <\infty $.
 Define, for $z\in \lbrack 0,1]$, 
\begin{equation*}
G_{0}(z)\doteq \sum_{k=1}^{\infty }p_{k}z^{k}\;\;\mbox{ and }%
\;\;G_{1}(z)\doteq \sum_{k=1}^{\infty }\frac{kp_{k}}{\mu }z^{k-1}.
\end{equation*}%
Define $F_{s}(t)\doteq G_{0}(s)-G_{0}(st)$ for $s\in (0,1]$ and $t\in
\lbrack 0,1]$. Then $F_{s}\colon \lbrack 0,1]\rightarrow \lbrack 0,G_{0}(s)]$
is strictly decreasing and continuous. Let $F_{s}^{-1}(\cdot )$ denote the
inverse of $F_{s}$. Define 
\begin{equation*}
f_{s}(t)\doteq \left\{ 
\begin{array}{ll}
F_{s}^{-1}(t) & \mbox{ when }0\leq t\leq G_{0}(s), \\ 
0 & \mbox{ when }t>G_{0}(s).
\end{array}%
\right.
\end{equation*}%
Then $f_{s}(t)$ is strictly decreasing until it hits zero. Note that in
particular, $f_{1}(t)=F_{1}^{-1}(t){{1}}_{[0,1]}(t)$. Define $%
f_{0}(t)\doteq 0$ for $t\geq 0$.

Fix $T\geq \frac{\mu }{2}$. The following theorem together with Proposition \ref{prop:uniqueness_LLN}
characterizes the unique $(\boldsymbol{\zeta} ,\psi )\in \mathcal{C}_{T}$ that minimizes
the rate function $I_{T}(\boldsymbol{\zeta} ,\psi )$. Letting 
\begin{equation*}
\nu \doteq \frac{\sum_{k=1}^{\infty }k(k-1)p_{k}}{\sum_{k=1}^{\infty }kp_{k}}%
,
\end{equation*}%
part 1 of the theorem considers the subcritical and critical cases $\nu \leq 1$, where the size of the largest component is $o(n)$, while part 2 considers the supercritical case $\nu >1$, where the size of the largest component is $O(n)$. Proofs of Theorem \ref{thm:LLN}  and
Proposition \ref{prop:uniqueness_LLN} are provided in Section \ref{sec:examples}.

\begin{Theorem}
\phantomsection
\label{thm:LLN} 
Suppose that Assumptions \ref{asp:convgN} and \ref{asp:exponential-boundN-S}
hold.

\begin{enumerate}[\upshape(1)]

\item Suppose $\sum_{k=1}^\infty k(k-2)p_k \le 0$. Define $\boldsymbol{\zeta}(t) =
(\zeta_k(t))_{k \in \mathbb{N}_0}$ and $\psi(t)$ by 
\begin{align*}
\zeta_0(t) & \doteq 0, \zeta_k(t) \doteq p_k(f_1(t))^k, k \in \mathbb{N}, \\
\psi(t) & \doteq -2 \int_0^t r_0(\boldsymbol{\zeta}(s))\,ds + \sum_{k=1}^\infty (k-2) (p_k
- \zeta_k(t)).
\end{align*}
Then $(\boldsymbol{\zeta} ,\psi ) \in \mathcal{C}_T$ and $I_T(\boldsymbol{\zeta} ,\psi )=0$.

\item Suppose $\sum_{k=1}^\infty k(k-2)p_k > 0$. If $p_1 > 0$, then there
exists a unique $\rho \in (0,1)$ such that $G_1(\rho) = \rho$. If $p_1 = 0$, 
$G_1(\rho) = \rho$ with $\rho\doteq 0$. Define $\tau = \frac{\mu}{2}%
(1-\rho^2)>0$ and define $\boldsymbol{\zeta}(t) = (\zeta_k(t))_{k \in \mathbb{N}_0}$ and $%
\psi(t)$ by 
\begin{align*}
\zeta_0(t) & \doteq \left[ \mu-2t -\mu \sqrt{1-2t/\mu}G_1 ( \sqrt{1-2t/\mu}) %
\right] {{1}}_{[0,\tau]}(t), \\
\zeta_k(t) & \doteq \left\{ 
\begin{array}{ll}
p_k(1-{2t}/{\mu})^{k/2} & \mbox{ when } 0 \le t \le \tau, \\ 
p_k \rho^k (f_\rho(t - \tau))^k & \mbox{ when } t > \tau,
\end{array}
\right. k \in \mathbb{N}, \\
\psi(t) & \doteq -2 \int_0^t r_0(\boldsymbol{\zeta}(s))\,ds + \sum_{k=1}^\infty (k-2) (p_k
- \zeta_k(t)).
\end{align*}
Then $(\boldsymbol{\zeta} ,\psi ) \in \mathcal{C}_T$ and $I_T(\boldsymbol{\zeta} ,\psi )=0$.
\end{enumerate}
\end{Theorem}

The following proposition says that there is a unique $(\boldsymbol{\zeta} ,\psi )$ satisfying $I_T(\boldsymbol{\zeta} ,\psi ) = 0$, so that this pair is the law of large numbers limit.

\begin{Proposition}
\label{prop:uniqueness_LLN} Suppose Assumptions \ref{asp:convgN} and \ref%
{asp:exponential-boundN-S} hold. Then the pair $(\boldsymbol{\zeta} ,\psi )$ defined in
Theorem \ref{thm:LLN} is the unique element of $\mathcal{D}_\infty\times 
\mathcal{D}$ such that $I_T(\boldsymbol{\zeta} ,\psi ) = 0$.
\end{Proposition}


\section{Representation and Weak Convergence of Controlled Processes}

\label{sec:repnWCCP} We will use the following useful representation formula
proved in \cite{BudhirajaDupuisMaroulas2011variational}. For the second
equality in the theorem see the proof of Theorem $2.4$ in \cite%
{BudhirajaChenDupuis2013large}. The representation in the cited papers is
given in terms of a single Poisson random measure with points in a locally
compact Polish space. However for the current work it is convenient to
formulate the representation in terms of a countable sequence of independent
Poisson random measures on $[0,T]\times \lbrack 0,1]$. This representation
is immediate from the results in \cite%
{BudhirajaDupuisMaroulas2011variational} and \cite%
{BudhirajaChenDupuis2013large} by viewing the countable sequence of Poisson
random measures with points in $[0,T]\times \lbrack 0,1]$ and intensity the
Lebesgue measure $\lambda _{T}$ on $[0,T]\times \lbrack 0,1]$ as a single
PRM with points in the augmented space $[0,T]\times \lbrack 0,1]\times 
\mathbb{N}_{0}$ with intensity $\lambda _{T}\otimes \varrho $, where $%
\varrho $ is the counting measure on $\mathbb{N}$. Recall that $\bar{%
\mathcal{A}}_{+}$ denotes the class of $(\mathcal{\bar{P}}\times \mathcal{B}%
([0,1]))/\mathcal{B}(\mathbb{R}_{+})$-measurable maps from $\Omega \times
\lbrack 0,T]\times \lbrack 0,1]$ to $\mathbb{R}_{+}$. For each $m\in \mathbb{%
N}$ let 
\begin{align*}
& \bar{\mathcal{A}}_{b,m}\doteq \{(\varphi _{k})_{k\in \mathbb{N}%
_{0}}:\varphi _{k}\in \bar{\mathcal{A}}_{+}\mbox{ for each }k\in \mathbb{N}%
_{0}\mbox{ such that for all }(\omega ,t,y)\in \Omega \times \lbrack
0,T]\times \lbrack 0,1], \\
& \quad \qquad \qquad \qquad \qquad {1}/{m}\leq \varphi _{k}(\omega
,t,y)\leq m\mbox{ for }k\leq m\mbox{ and }\varphi _{k}(\omega ,t,y)=1\mbox{
for }k>m\}
\end{align*}%
and let $\bar{\mathcal{A}}_{b}\doteq \cup _{m=1}^{\infty }\bar{\mathcal{A}}%
_{b,m}$. Recall the function $\ell $ defined in \eqref{eq:ell}.

\begin{Theorem}
\label{thm:var_repn} Let $F \in \mathbb{M}_b([\mathcal{M}_{FC}([0,T]\times[%
0,1])]^\infty)$. Then for $\theta > 0$, 
\begin{align*}
-\log {{E}} e^{-F((N_k^\theta)_{k \in \mathbb{N}_0})} & =
\inf_{\varphi_k\in \bar{\mathcal{A}}_+, k \in \mathbb{N}_0} {{E}} %
\left[ \theta \sum_{k=0}^\infty \int_{[0,T] \times [0,1]}
\ell(\varphi_k(s,y))\,ds\,dy + F((N_k^{\theta\varphi_k})_{k \in \mathbb{N}%
_0}) \right] \\
& = \inf_{\boldsymbol{\varphi} = (\varphi_k)_{k \in \mathbb{N}_0} \in \bar{\mathcal{A}}%
_b} {{E}} \left[ \theta \sum_{k=0}^\infty \int_{[0,T] \times [0,1]}
\ell(\varphi_k(s,y))\,ds\,dy + F((N_k^{\theta\varphi_k})_{k \in \mathbb{N}%
_0}) \right].
\end{align*}
\end{Theorem}

Fix $h\in \mathbb{C}_{b}(\mathcal{D}_{\infty }\times \mathcal{D})$. Since $%
(\boldsymbol{X}^{n},Y^{n})$ can be written as $\Psi ((N_{k}^{n})_{k\in \mathbb{N}_{0}})$
for some measurable function $\Psi $ from $[\mathcal{M}_{FC}([0,T]\times
\lbrack 0,1])]^{\infty }$ to $\mathcal{D}_{\infty }\times \mathcal{D}$, we
have from the second equality in Theorem \ref{thm:var_repn} that with $%
(\theta ,F)=(n,nh\circ \Psi )$, 
\begin{equation}
-\frac{1}{n}\log {{E}}e^{-nh(\boldsymbol{X}^{n},Y^{n})}=\inf_{\boldsymbol{\varphi}
^{n}=(\varphi _{k}^{n})_{k\in \mathbb{N}_{0}}\in \bar{\mathcal{A}}_{b}}{%
{E}}\left\{ \sum_{k=0}^{\infty }\int_{[0,T]\times \lbrack 0,1]}\ell
(\varphi _{k}^{n}(s,y))\,ds\,dy+h(\bar{\boldsymbol{X}}^{n},{\bar{Y}}%
^{n})\right\} .  \label{eq:mainrepn17}
\end{equation}%
Here $(\bar{\boldsymbol{X}}^{n},{\bar{Y}}^{n})=\Psi
((N_{k}^{n\varphi _{k}^{n}})_{k\in \mathbb{N}_{0}})$, which solves the
controlled analogue of \eqref{eq:Y_n}--\eqref{eq:X_n_k}, namely $\bar{\boldsymbol{X}}%
^{n}(0)\doteq \frac{1}{n}(-1,n_{1},n_{2},\dotsc )$, and for $%
t\in \lbrack 0,T]$, 

\begin{align}
{\bar{Y}}^{n}(t)& ={\bar{X}}_{0}^{n}(0)+\sum_{k=0}^{\infty }\frac{k-2}{n}\int_{[0,t]\times \lbrack 0,1]}{%
{1}}_{[0,r_{k}(\bar{\boldsymbol{X}}^{n}(s-)))}(y)\,N_{k}^{n\varphi _{k}^{n}}(ds\,dy)
\label{eq:Ybar_n_upper_temp} \\
{\bar{X}}_{0}^{n}(t)& ={\bar{Y}}^{n}(t) +\frac{2}{n}\sum_{k=0}^{\infty }\int_{[0,t]\times \lbrack 0,1]}{%
{1}}_{\{{\bar{X}}_{0}^{n}(s-)<0\}}{{1}}%
_{[0,r_{k}(\bar{\boldsymbol{X}}^{n}(s-)))}(y)\,N_{k}^{n\varphi
_{k}^{n}}(ds\,dy),  \label{eq:Xbar_n_0_upper_temp} \\
{\bar{X}}_{k}^{n}(t)& ={\bar{X}}_{k}^{n}(0)-\frac{1%
}{n}\int_{[0,t]\times \lbrack 0,1]}{{1}}_{[0,r_{k}(\bar{\boldsymbol{X}}%
^{n}(s-)))}(y)\,N_{k}^{n\varphi _{k}^{n}}(ds\,dy),\; k\in \mathbb{%
N}.  \label{eq:Xbar_n_k_upper_temp}
\end{align}%
There is a bar in the notation $\bar{\boldsymbol{X}}^n, \bar{Y}^n$ (and $\bar{\nu}^n$ defined in \eqref{eq:nu_n_upper} below) to indicate that these are `controlled' processes, given in terms of the 
control sequence $\boldsymbol{\varphi}^{n}\doteq (\varphi _{k}^{n})_{k\in \mathbb{N}_{0}}$.
We will occasionally suppress the dependence on $\boldsymbol{\varphi}^{n}$ in the notation and will make this dependence explicit if there are multiple controls (e.g.\ as in Section \ref{sec:upper})

In the proof of both the upper and lower bound we will show it is sufficient to
consider a sequence $\{\varphi _{k}^{n}\in \bar{\mathcal{A}}_{+},k\in 
\mathbb{N}_{0}\}$ that satisfies the following uniform bound for some $%
M_{0}<\infty$: 
\begin{equation}
\sup_{n\in \mathbb{N}}\sum_{k=0}^{\infty }\int_{[0,T]\times \lbrack
0,1]}\ell (\varphi _{k}^{n}(s,y))\,ds\,dy\leq M_{0},\mbox{ a.s. }{%
P}.  \label{eq:cost_bd_upper}
\end{equation}
In the rest of this section we study tightness and convergence properties of
controlled processes $(\bar{\boldsymbol{X}}^{n},{\bar{Y}}^{n})$
that are driven by controls $\{\varphi_k^n\}$ that satisfy the above a.s.
bound. 

From \eqref{eq:Ybar_n_upper_temp}--\eqref{eq:Xbar_n_k_upper_temp} we can
rewrite
\begin{align}
{\bar{Y}}^{n}(t)& ={\bar{X}}_{0}^{n}(0)+\sum_{k=0}^{\infty }(k-2){\bar{B}}%
_{k}^{n}(t),  \label{eq:Ybar_n_upper} \\
{\bar{X}}_{0}^{n}(t)& ={\bar{Y}}^{n}(t)+{\bar{\eta}}^{n}(t),
\label{eq:Xbar_n_0_upper} \\
{\bar{X}}_{k}^{n}(t)& ={\bar{X}}_{k}^{n}(0)-{\bar{B}}_{k}^{n}(t),k\in 
\mathbb{N},  \label{eq:Xbar_n_k_upper}
\end{align}%
where 
\begin{align}
{\bar{B}}_{k}^{n}(t)& \doteq \frac{1}{n}\int_{[0,t]\times \lbrack 0,1]}{%
{1}}_{[0,r_{k}(\bar{\boldsymbol{X}}^{n}(s-)))}(y)\,N_{k}^{n\varphi
_{k}^{n}}(ds\,dy),k\in \mathbb{N}_{0},  \label{eq:Bbar_n_k_upper} \\
{\bar{\eta}}^{n}(t)& \doteq \sum_{k=0}^{\infty }\frac{2}{n}\int_{[0,t]\times
\lbrack 0,1]}{{1}}_{\{{\bar{X}}_{0}^{n}(s-)<0\}}{{1}}%
_{[0,r_{k}(\bar{\boldsymbol{X}}^{n}(s-)))}(y)\,N_{k}^{n\varphi _{k}^{n}}(ds\,dy)  \notag
\\
& =\sum_{k=1}^{\infty }\frac{2}{n}\int_{[0,t]\times \lbrack 0,1]}{%
{1}}_{\{{\bar{X}}_{0}^{n}(s-)<0\}}{{1}}_{[0,r_{k}(\bar{\boldsymbol{X}}^{n}(s-)))}(y)\,N_{k}^{n\varphi _{k}^{n}}(ds\,dy).
\label{eq:etabar_n_upper}
\end{align}%
Here the last line follows on observing that ${{1}}_{\{{\bar{X}}_{0}^{n}(s-)<0\}}{{1}}_{[0,r_{0}(\bar{\boldsymbol{X}}^{n}(s-)))}(y)\equiv 0$.

Since $m_1 \doteq \sup_{n \in \mathbb{N}} \sum_{k=1}^\infty k \frac{n_k}{n}
< \infty$ by Assumption \ref{asp:exponential-boundN}, \blue{using \eqref{eq:r_k}} we have $-\frac{1}{n} \le {\bar{X}}^n_0(t) \le m_1$, $0 \le r(\bar{\boldsymbol{X}}^n(t)) \le m_1$ and $0 \le {\bar{X}}^n_k(t) \le \frac{n_k}{n}$ for $t \in [0,T]$. 
\cg{In particular, the nonnegativity of $\bar{X}_{k}^{n}(t)$ is an immediate consequence of the evolution equation \eqref{eq:Xbar_n_k_upper_temp}  on observing that $r_k(\bar{\boldsymbol{X}}^{n}(s-))=0$ if 
$\bar{X}_{k}^{n}(s-)=0$ and that the jumps of $\bar{X}_{k}^{n}$ are of size $1/n$.}
Also note that both $r(\bar{\boldsymbol{X}}^n(\cdot))$ and ${\bar{X}}^n_k(\cdot)$ for $k \in \mathbb{N}$ are non-increasing.

The following lemma summarizes some elementary properties of $\ell$. For
part (a) we refer to \cite[Lemma 3.1]{BudhirajaDupuisGanguly2015moderate},
and part (b) is an easy calculation that is omitted.

\begin{Lemma}
\phantomsection
\label{lem:property_ell}

\begin{enumerate}[\upshape(a)] 

\item For each $\beta > 0$, there exists $\gamma(\beta) \in (0,\infty)$ such
that $\gamma(\beta) \to 0$ as $\beta \to \infty$ and $x \le \gamma(\beta)
\ell(x)$, for $x \ge \beta > 1$.

\item For $x \ge 0$, $x \le \ell(x) + 2$.
\end{enumerate}
\end{Lemma}

The next lemma gives some uniform integrability properties for the
control sequence $\boldsymbol{\varphi}^n$.

\begin{Lemma}
\label{lem:UI_upper} 
For $K\in \mathbb{N}$ define%
\begin{equation}
{\bar{U}}_{K}\doteq \sup_{n\in \mathbb{N}}{{E}}\left\{
\sum_{k=K}^{\infty }\int_{[0,T]\times \lbrack 0,1]}k\varphi _{k}^{n}(s,y){%
{1}}_{[0,r_{k}(\bar{\boldsymbol{X}}^{n}(s)))}(y)\,ds\,dy\right\} .
\label{eq:Ubar_K}
\end{equation}%
Then ${\bar{U}}_{K}<\infty $ for each $K\in \mathbb{N}$ and $%
\lim_{K\rightarrow \infty }{\bar{U}}_{K}=0$.
\end{Lemma}

\begin{proof}
	From \eqref{eq:Bbar_n_k_upper} and \eqref{eq:Xbar_n_k_upper} it follows that
	\begin{equation*}
		{\bar{U}}_K = \sup_{n \in \mathbb{N}} {{E}} \left\{ \sum_{k=K}^\infty k {\bar{B}}^n_k(T) \right\} = \sup_{n \in \mathbb{N}} {{E}} \left\{ \sum_{k=K}^\infty k \left[ {\bar{X}}^n_k(0) - {\bar{X}}^n_k(T) \right] \right\} \le \sup_{n \in \mathbb{N}} \sum_{k=K}^\infty \frac{kn_k}{n}.
	\end{equation*}
	Recalling $\varepsilon_{\boldsymbol{p}} \in (0,\infty)$ introduced in Assumption \ref{asp:exponential-boundN}, we have
	\begin{equation*}
		\sup_{n \in \mathbb{N}} \sum_{k=K}^\infty \frac{kn_k}{n} \le 
		K^{-\varepsilon_{\boldsymbol{p}}} \sup_{n \in \mathbb{N}} \sum_{k=1}^\infty \frac{n_k}{n} k^{1+ \varepsilon_{\boldsymbol{p}}} \to 0
	\end{equation*}
	as $K \to \infty$.
	The result follows.
%
\end{proof}

The following lemma proves some key tightness properties. Write $\bar{\boldsymbol{B}}^n
\doteq \{{\bar{B}}^n_k\}_{k \in \mathbb{N}_0}$.
Define $\bar{\boldsymbol{\nu}}^{n}\doteq \{\bar{\nu}_{k}^{n}\}_{k\in 
\mathbb{N}_{0}}$, where for $k\in \mathbb{N}_{0}$,
\begin{equation}
\bar{\nu} _{k}^{n}([0,t]\times A)\doteq \int_{\lbrack 0,t]\times
A}\varphi _{k}^{n}(s,y)\,ds\,dy,\quad t\in \lbrack 0,T],A\in \mathcal{B}%
([0,1]).  \label{eq:nu_n_upper}
\end{equation}%

\begin{Lemma}
\label{lem:tightness} Suppose that the bound in \eqref{eq:cost_bd_upper} is
satisfied. Then the sequence of random variables $\{(\bar{\boldsymbol{\nu}}^{n}, \bar{\boldsymbol{X}}^n, {\bar{Y}}^n, \bar{\boldsymbol{B}}^n, {\bar{\eta}}^n)\}$ is tight in $[%
\mathcal{M}([0,T]\times[0,1])]^\infty \times \mathcal{D}_\infty \times 
\mathcal{D} \times \mathcal{D}_\infty \times \mathcal{D}$.
\end{Lemma}

\begin{proof}
	We will argue the tightness of $\{\bar{\boldsymbol{\nu}}^{n}\}$ in $[\mathcal{M}([0,T]\times[0,1])]^\infty$ and the $\mathcal{C}$-tightness of $\{{\bar{\boldsymbol{X}}}^n\}$, $\{{\bar{Y}}^n\}$, $\{{\bar{\boldsymbol{B}}}^n\}$, and
	$\{{\bar{\eta}}^n\}$ in $\mathcal{D}_\infty$, $\mathcal{D}$, $\mathcal{D}_\infty$, and $\mathcal{D}$ respectively.
	This will complete the proof.
	
	Consider first $\{\bar{\boldsymbol{\nu}}^{n}\}$. 
	Note that $[0,T] \times [0,1]$ is a compact metric space.
	Also from Lemma \ref{lem:property_ell}(b) and \eqref{eq:cost_bd_upper} we have a.s.\ for each $k \in \mathbb{N}_0$,
	\begin{equation*}
		\bar{\nu}^{n}_k([0,T]\times[0,1]) = \int_{[0,T] \times [0,1]} \varphi_k^n(s,y) \,ds\,dy \le \int_{[0,T] \times [0,1]} \left( \ell(\varphi_k^n(s,y))+2 \right) ds\,dy \le M_0 + 2T.
	\end{equation*}
	Hence $\{\bar{\nu}^{n}_k\}$ is tight in $\mathcal{M}([0,T]\times[0,1])$.	
	
	Next, since ${\bar{X}}^n_k(0) \in [0,1]$ for $k \in \mathbb{N}$ a.s.,  we see from \eqref{eq:Xbar_n_0_upper} and \eqref{eq:Xbar_n_k_upper} that $\mathcal{C}$-tightness of $\{{\bar{\boldsymbol{X}}}^n\}$ in $\mathcal{D}_\infty$ follows once we show $\mathcal{C}$-tightness of $\{{\bar{Y}}^n\}$, $\{{\bar{\boldsymbol{B}}}^n\}$ and  $\{{\bar{\eta}}^n\}$.

	We now show that $\{({\bar{Y}}^n(t), {\bar{\boldsymbol{B}}}^n(t), {\bar{\eta}}^n(t)))\}$ is tight for each $t$.
	From \eqref{eq:Ybar_n_upper}, \eqref{eq:Bbar_n_k_upper} and \eqref{eq:etabar_n_upper},
	 
	\begin{align*}
		& {{E}} \left[ |{\bar{Y}}^n(t)| + \sum_{k=0}^\infty |{\bar{B}}^n_k(t)| + |{\bar{\eta}}^n(t)| \right] \\
		& \le \frac{1}{n}+ \sum_{k=0}^\infty [|k-2|+1] {{E}}|{\bar{B}}^n_k(t)| + {{E}}|{\bar{\eta}}^n(t)| \\
		& \le \frac{1}{n}+{{E}} \sum_{k=0}^\infty \int_{[0,T] \times [0,1]} [|k-2| + 1 + 2 \cdot {{1}}_{\{k\ge 1\}}] \varphi^n_k(s,y) {{1}}_{[0,r_k(\bar{\boldsymbol{X}}^n(s)))}(y) \,ds\,dy \\
		& \le \frac{1}{n}+ 3 {{E}} \int_{[0,T] \times [0,1]} \varphi^n_0(s,y) \,ds\,dy + 4{\bar{U}}_1,
	\end{align*}
	where the last line uses the definition of ${\bar{U}}_1$ in \eqref{eq:Ubar_K}.
	From Lemma \ref{lem:property_ell}(b) and \eqref{eq:cost_bd_upper}, we have
	\begin{equation*}
		{{E}} \int_{[0,T] \times [0,1]} \varphi^n_0(s,y) \,ds\,dy \le {{E}} \int_{[0,T] \times [0,1]} \left[ \ell(\varphi^n_0(s,y))+2 \right] ds\,dy \le M_0 + 2T.
	\end{equation*}
	Therefore $\sup_{n \in \mathbb{N}} {{E}} \left[ |{\bar{Y}}^n(t)| + \sum_{k=0}^\infty |{\bar{B}}^n_k(t)| + |{\bar{\eta}}^n(t)| \right] < \infty$ and we have tightness of $\{({\bar{Y}}^n(t), \bar{\boldsymbol{B}}^n(t), {\bar{\eta}}^n(t)))\}$ in $\mathbb{R} \times \mathbb{R}^\infty \times \mathbb{R}$ for each $t \in [0,T]$.
	
	We now consider fluctuations of $({\bar{Y}}^n, \bar{\boldsymbol{B}}^n, {\bar{\eta}}^n)$.
	Recall the filtration $\{\mathcal{F}_t\}_{0 \le t \le T}$.
	For $\delta \in [0,T]$, let $\mathcal{T}^\delta$ be the collection of all $[0,T-\delta]$-valued stopping times $\tau$.
	Note that for $\tau \in \mathcal{T}^\delta$,
	\begin{equation*}
		{{E}} |{\bar{Y}}^n(\tau+\delta) - {\bar{Y}}^n(\tau)| \le {{E}} \left[ \sum_{k=0}^\infty (k+2) \left| {\bar{B}}^n_k(\tau+\delta) - {\bar{B}}^n_k(\tau) \right|  \right].
	\end{equation*}
	Thus in order to argue tightness of $\{({\bar{Y}}^n, \bar{\boldsymbol{B}}^n, {\bar{\eta}}^n)\}$, by the Aldous--Kurtz tightness criterion (cf.\ \cite[Theorem 2.7]{Kurtz1981approximation}) it suffices to show that
	\begin{equation}
		\label{eq:tightness_upper}
		\limsup_{\delta \to 0} \limsup_{n \to \infty} \sup_{\tau \in \mathcal{T}^\delta} {{E}} \left[ \sum_{k=0}^\infty (k+2) \left| {\bar{B}}^n_k(\tau+\delta) - {\bar{B}}^n_k(\tau) \right| + \left| {\bar{\eta}}^n(\tau+\delta) - {\bar{\eta}}^n(\tau) \right| \right] = 0.
	\end{equation}
	From \eqref{eq:Bbar_n_k_upper} and \eqref{eq:etabar_n_upper} it follows that for every $K \in \mathbb{N}$ and $M \in (0,\infty)$,
	\begin{align*}
		& {{E}} \left[ \sum_{k=0}^\infty (k+2) \left| {\bar{B}}^n_k(\tau+\delta) - {\bar{B}}^n_k(\tau) \right| + \left| {\bar{\eta}}^n(\tau+\delta) - {\bar{\eta}}^n(\tau) \right| \right] \\
		& \le {{E}} \sum_{k=0}^\infty \int_{(\tau,\tau+\delta] \times [0,1]} (k+4) \varphi^n_k(s,y)  {{1}}_{[0,r_k(\bar{\boldsymbol{X}}^n(s)))}(y) \,ds\,dy \\
		& \le {{E}} \sum_{k=0}^{K-1} \left[ \int_{(\tau,\tau+\delta] \times [0,1]} (k+4) \varphi^n_k(s,y) {{1}}_{\{ \varphi^n_k(s,y) > M\}}  \,ds\,dy \right. \\
		& \qquad \left. + \int_{(\tau,\tau+\delta] \times [0,1]} (k+4) \varphi^n_k(s,y) {{1}}_{\{ \varphi^n_k(s,y) \le M\}}  \,ds\,dy \right] + 5 {\bar{U}}_K.
	\end{align*}
	Using Lemma \ref{lem:property_ell}(a) and \eqref{eq:cost_bd_upper}, we can bound the last display by
	\begin{align*}
		& {{E}} \sum_{k=0}^{K-1} \int_{(\tau,\tau+\delta] \times [0,1]} (K+3)\gamma(M) \ell(\varphi^n_k(s,y)) \,ds\,dy + K(K+3)M\delta + 5 {\bar{U}}_K \\
		& \quad \le (K+3)\gamma(M) M_0 + K(K+3)M\delta + 5 {\bar{U}}_K.
	\end{align*}
	Therefore
	\begin{align*}
		& \limsup_{\delta \to 0} \limsup_{n \to \infty} \sup_{\tau \in \mathcal{T}^\delta} {{E}} \left[ \sum_{k=0}^\infty (k+2) \left| {\bar{B}}^n_k(\tau+\delta) - {\bar{B}}^n_k(\tau) \right| + \left| {\bar{\eta}}^n(\tau+\delta) - {\bar{\eta}}^n(\tau) \right| \right] \\
		& \quad \le (K+3)\gamma(M) M_0 + 5 {\bar{U}}_K.
	\end{align*}
	Taking $M \to \infty$ and then $K \to \infty$, we have from Lemmas \ref{lem:property_ell}(a) and  \ref{lem:UI_upper} that \eqref{eq:tightness_upper} holds.
	Finally $\mathcal{C}$-tightness is immediate from the following a.s.\ bounds, Assumption \ref{asp:exponential-boundN}, \blue{and \cite[Theorem 13.4]{Billingsley1999}}: for any $k \in \mathbb{N}_0$, $K \in \mathbb{N}$ and $t \in (0,T]$,
	\begin{align*}
		|{\bar{B}}^n_k(t)- {\bar{B}}^n_k(t-)|  \le \frac{1}{n}, \;
		|{\bar{\eta}}^n(t)- {\bar{\eta}}^n(t-)|  \le \frac{2}{n}, \;
		|{\bar{Y}}^n_k(t)- {\bar{Y}}^n_k(t-)|  \le \frac{K}{n} + \sum_{j=K+1}^\infty \frac{jn_j}{n}.
	\end{align*}
	This completes the proof.
\end{proof}

Next we will characterize weak limit points of $\{(\bar{\boldsymbol{\nu}} ^{n},\bar{\boldsymbol{X}}^{n},{\bar{Y}}^{n},\bar{\boldsymbol{B}}^{n},{\bar{\eta}}^{n})\}$. For that, we
need the following notation. For $k\in \mathbb{N}_{0}$ define the
compensated process 
\begin{equation*}
{\tilde{N}}_{k}^{n\varphi _{k}^{n}}(ds\,dy)\doteq N_{k}^{n\varphi
_{k}^{n}}(ds\,dy)-n\varphi _{k}^{n}(s,y)\,ds\,dy.
\end{equation*}%
Then ${\tilde{N}}_{k}^{n\varphi _{k}^{n}}([0,t]\times A)$ is an $\{\mathcal{F%
}_{t}\}$-martingale for $A\in \mathcal{B}([0,1])$ and $k\in \mathbb{N}_{0}$.
Let
\begin{equation}
{\bar{B}}_{k}^{n}(t)={\tilde{B}}_{k}^{n}(t)+{\hat{B}}_{k}^{n}(t), \; t \in [0,T], \; k \in \mathbb{N}_{0}.
\label{eq:Bbar_n_k_upper_decomp}
\end{equation}%
where 
\begin{equation*}
{\tilde{B}}_{k}^{n}(t)\doteq \frac{1}{n}\int_{[0,t]\times \lbrack 0,1]}{%
{1}}_{[0,r_{k}(\bar{\boldsymbol{X}}^{n}(s-)))}(y)\,{\tilde{N}}_{k}^{n\varphi
_{k}^{n}}(ds\,dy)
\end{equation*}%
is an $\{\mathcal{F}_{t}\}$-martingale and 
\begin{equation*}
{\hat{B}}_{k}^{n}(t)\doteq \int_{\lbrack 0,t]\times \lbrack 0,1]}{%
{1}}_{[0,r_{k}(\bar{\boldsymbol{X}}^{n}(s)))}(y)\varphi
_{k}^{n}(s,y)\,ds\,dy.
\end{equation*}%
Write $\tilde{\boldsymbol{B}}^n \doteq ({\tilde{B}}^n_k)_{k \in \mathbb{N}_0}$ and $\hat{\boldsymbol{B}}^n \doteq ({\hat{B}}^n_k)_{k \in \mathbb{N}_0}$.
Let $\lambda _{t}$ be Lebesgue measure on $%
[0,t]\times \lbrack 0,1]$.

We have the following characterization of the weak limit points.
Recall $\mathcal{S}_{T}(\boldsymbol{\zeta} ,\psi )$ defined in \eqref{eq:psi} and \eqref{eq:phi_k}.

\begin{Lemma}
\label{lem:char_limit} Suppose Assumptions \ref{asp:convgN} and \ref%
{asp:exponential-boundN} hold. Also assume that the bound %
\eqref{eq:cost_bd_upper} is satisfied and suppose that $(\bar{\boldsymbol{\nu}}^{n}, \bar{\boldsymbol{X}}^n, {\bar{Y}}^n, \bar{\boldsymbol{B}}^n, {\bar{\eta}}^n)$ converges along a
subsequence, in distribution, to $(\bar{\boldsymbol{\nu}}, \bar{\boldsymbol{X}}, {\bar{Y}}, \bar{\boldsymbol{B}}, {%
\bar{\eta}}) \in [\mathcal{M}([0,T]\times[0,1])]^\infty \times \mathcal{D}%
_\infty \times \mathcal{D} \times \mathcal{D}_\infty \times \mathcal{D}$
given on some probability space $(\Omega^*,\mathcal{F}^*,{P}^*)$%
. Then the following holds ${P}^*$-a.s.

\begin{enumerate}[\upshape(a)]

\item For each $k \in \mathbb{N}_0$, $\bar{\nu}_k \ll \lambda_T$.

\item $(\bar{\boldsymbol{X}},{\bar{Y}},\bar{\boldsymbol{B}},{\bar{\eta}})\in \mathcal{C}_{\infty
}\times \mathcal{C}\times \mathcal{C}_{\infty }\times \mathcal{C}$, and for $%
t\in \lbrack 0,T]$%
\begin{align}
{\bar{X}}_{k}(t)& =p_{k}-{\bar{B}}_{k}(t)\geq 0,k\in \mathbb{N},
\label{eq:Xbar_k_upper} \\
{\bar{Y}}(t)& =\sum_{k=0}^{\infty }(k-2){\bar{B}}_{k}(t),
\label{eq:Ybar_upper} \\
{\bar{X}}_{0}(t)& ={\bar{Y}}(t)+{\bar{\eta}}(t)\geq 0.
\label{eq:Xbar_0_upper}
\end{align}

\item For $k\in \mathbb{N}_{0}$ let
$
\varphi _{k}(s,y)\doteq \frac{d\bar{\nu} _{k}}{d\lambda _{T}}(s,y),\;(s,y)\in
\lbrack 0,T]\times \lbrack 0,1].
$
Then for $t\in \lbrack 0,T]$ and $k\in \mathbb{N}_{0}$%
\begin{equation}
\label{eq:Bbar_k_upper}
{\bar{B}}_{k}(t)=\int_{[0,t]\times \lbrack 0,1]}{{1}}_{[0,r_{k}(\bar{\boldsymbol{X}}(s)))}(y)\,\varphi _{k}(s,y)ds\,dy.
\end{equation}

\item ${\bar{X}}_{0}=\Gamma ({\bar{Y}})$. In particular, $(\bar{\boldsymbol{X}},{\bar{Y}%
})\in \mathcal{C}_{T}$ and $\boldsymbol{\varphi} \in \mathcal{S}_{T}(\bar{\boldsymbol{X}},{\bar{Y}})$%
.
\end{enumerate}
\end{Lemma}

\begin{proof}
	Assume without loss of generality that $(\bar{\boldsymbol{\nu}}^{n}, \bar{\boldsymbol{X}}^n, {\bar{Y}}^n, \bar{\boldsymbol{B}}^n, {\bar{\eta}}^n) \Rightarrow (\bar{\boldsymbol{\nu}}, \bar{\boldsymbol{X}}, {\bar{Y}}, \bar{\boldsymbol{B}}, {\bar{\eta}})$ along the whole sequence as $n \to \infty$. 
	
	(a) 
	This is an immediate consequence of the bound in \eqref{eq:cost_bd_upper} and Lemma A.1 of \cite{BudhirajaChenDupuis2013large}.

	(b) 
	The first statement is an immediate consequence of the $\mathcal{C}$-tightness argued in the proof of Lemma \ref{lem:tightness}. 
	Then using \eqref{eq:Xbar_n_k_upper}, Assumption \ref{asp:convgN} and the fact that ${\bar{X}}^n_k(t) \ge 0$ a.s., we have \eqref{eq:Xbar_k_upper}.
	Next, note that 
	 by Assumption \ref{asp:exponential-boundN}, as $K \to \infty$
	\begin{equation}
		\label{eq:UI_Xbar_n}
		\sup_{n \in \mathbb{N}} \sup_{0 \le t \le T} \left| \sum_{k=K}^\infty (k-2){\bar{B}}^n_k(t) \right| \le \sup_{n \in \mathbb{N}} \sum_{k=K}^\infty \frac{kn_k}{n} \le K^{-\varepsilon_{\boldsymbol{p}}} \sup_{n \in \mathbb{N}} \sum_{k=K}^\infty \frac{n_k}{n} k^{1+\varepsilon_{\boldsymbol{p}}} \to 0,
	\end{equation}
	\cg{where in obtaining the first inequality we have used the fact that 
	 due to \eqref{eq:Xbar_n_k_upper} and the nonnegativity of $\bar{X}_{k}^{n}(t)$, $\bar B^n_k(t) \le \bar X^n_k(0)$.}
	Hence $\sum_{k=0}^\infty (k-2) {\bar{B}}^n_k \Rightarrow \sum_{k=0}^\infty (k-2) {\bar{B}}_k \in \mathcal{C}$.
	From this and \eqref{eq:Ybar_n_upper} we see that \eqref{eq:Ybar_upper} holds.
	Next, since $({\bar{Y}}^n,{\bar{\eta}}^n) \Rightarrow ({\bar{Y}},{\bar{\eta}}) \in \mathcal{C}^2$ and ${\bar{X}}^n_0(t) \ge -\frac{1}{n}$ a.s., we have from \eqref{eq:Xbar_n_0_upper} that \eqref{eq:Xbar_0_upper} holds.
	
	(c)
	By Doob's inequality, as $n \to \infty$
	\begin{align*}
		{{E}} \sum_{k=0}^\infty \sup_{0 \le t \le T} |{\tilde{B}}^n_k(t)|^2 
		& \le \frac{4}{n} {{E}} \sum_{k=0}^\infty \int_{[0,T] \times [0,1]} \varphi^n_k(s,y)  {{1}}_{[0,r_k(\bar{\boldsymbol{X}}^n(s)))}(y) \,ds\,dy \\
		& \le \frac{4}{n} {{E}} \sum_{k=0}^\infty \int_{[0,T] \times [0,1]} \left[ \ell(\varphi^n_k(s,y)) + 2 \right] {{1}}_{[0,r_k(\bar{\boldsymbol{X}}^n(s)))}(y) ds\,dy\\
		& \le \frac{4}{n} \left( M_0 + 2 T \right) \to 0,
	\end{align*}
	where the second inequality follows from Lemma \ref{lem:property_ell}(b) and the third inequality follows from \eqref{eq:cost_bd_upper}.
	Therefore as $n \to \infty$
	\begin{equation}
		\label{eq:cvg_Atil_Btil_upper}
		{\tilde{\boldsymbol{B}}}^n \Rightarrow \boldsymbol{0}.
	\end{equation}	
	By appealing to the Skorokhod representation theorem, we can assume without loss of generality that $(\bar{\boldsymbol{\nu}}^{n}, \bar{\boldsymbol{X}}^n, {\bar{Y}}^n, \bar{\boldsymbol{B}}^n, {\bar{\eta}}^n, {\tilde{\boldsymbol{B}}}^n) \to (\bar{\boldsymbol{\nu}}, \bar{\boldsymbol{X}}, {\bar{Y}}, \bar{\boldsymbol{B}}, {\bar{\eta}}, \boldsymbol{0})$ a.s.\ on $(\Omega^*,\mathcal{F}^*,{P}^*)$, namely there exists some event $F \in \mathcal{F}^*$ such that ${P}^*(F^c) = 0$ and 
	\begin{equation*}
		(\bar{\boldsymbol{\nu}}^{n}, \bar{\boldsymbol{X}}^n, {\bar{Y}}^n, \bar{\boldsymbol{B}}^n, {\bar{\eta}}^n, {\tilde{\boldsymbol{B}}}^n) \to (\bar{\boldsymbol{\nu}}, \bar{\boldsymbol{X}}, {\bar{Y}}, \bar{\boldsymbol{B}}, {\bar{\eta}}, \boldsymbol{0}) \mbox{ on } F.
	\end{equation*}
	Fix ${\bar{\omega}} \in F$. The rest of the argument will be made for such an ${\bar{\omega}}$ which will be suppressed from the notation.
	From \eqref{eq:UI_Xbar_n} we have that as $n \to \infty$ $$r(\bar{\boldsymbol{X}}^n(t)) = ({\bar{X}}^n_0(t))^+ + \sum_{k=1}^\infty k {\bar{X}}^n_k(t) \to ({\bar{X}}_0(t))^+ + \sum_{k=1}^\infty k {\bar{X}}_k(t) = r(\bar{\boldsymbol{X}}(t))$$
	uniformly in $t \in [0,T]$, and $r(\bar{\boldsymbol{X}}(\cdot))$ is continuous \blue{and hence bounded}.
	Let ${\bar{\tau}} \doteq \tau_{\bar{\boldsymbol{X}}}$, where $\tau_{\bar{\boldsymbol{X}}}$ is defined through \eqref{eq:defntauphi}, namely
	${\bar{\tau}} = \inf \{ t \in [0,T] : r(\bar{\boldsymbol{X}}(t)) = 0\} \wedge T$.
	We will argue that \eqref{eq:Bbar_k_upper} holds for all $t < {\bar{\tau}}$, $t = {\bar{\tau}}$ and $t >{\bar{\tau}}$.

	For $t < {\bar{\tau}}$, we have $r(\bar{\boldsymbol{X}}(t)) > 0$.
	Hence for each $k \in \mathbb{N}_0$,
	\begin{equation}\label{eq:cgceofindic}
		 {{1}}_{[0,r_k(\bar{\boldsymbol{X}}^n(s)))}(y) \to  {{1}}_{[0,r_k(\bar{\boldsymbol{X}}(s)))}(y)
	\end{equation}
	as $n \to \infty$ for $\lambda_t$-a.e.\ $(s,y) \in [0,t] \times [0,1]$
	since $\lambda_t\{(y,s): y = r_k(\bar{\boldsymbol{X}}(s))\}=0$.
	From \eqref{eq:cgceofindic} and the uniform integrability of $(s,y) \mapsto ({{1}}_{[0,r_k(\bar{\boldsymbol{X}}^n(s)))}(y) -{{1}}_{[0,r_k(\bar{\boldsymbol{X}}(s)))}(y))  \varphi^n_k(s,y)$ (with respect to the normalized Lebesgue measure on
	$[0,T]\times [0,1]$) which follows from \eqref{eq:cost_bd_upper} and the superlinearity of $\ell$, we have that
	$${\hat{B}}^n_k(t) - \int_{[0,t] \times [0,1]}  {{1}}_{[0,r_k(\bar{\boldsymbol{X}}(s)))}(y) \, \varphi_k^n(s,y) ds\,dy \to 0.$$
	Also, from the bound in \eqref{eq:cost_bd_upper} it follows that
	$$\int_{[0,t] \times [0,1]}  {{1}}_{[0,r_k(\bar{\boldsymbol{X}}(s)))}(y) \, \varphi_k^n(s,y) ds\,dy \to \int_{[0,t] \times [0,1]}  {{1}}_{[0,r_k(\bar{\boldsymbol{X}}(s)))}(y) \, \varphi_k(s,y) ds\,dy.$$
	Combining the two convergence statements we have
	\begin{equation}
		{\hat{B}}^n_k(t) \to \int_{[0,t] \times [0,1]}  {{1}}_{[0,r_k(\bar{\boldsymbol{X}}(s)))}(y) \, \varphi_k(s,y) ds\,dy.\label{eq:cgcebhatkt}
	\end{equation}
	The above convergence along with \eqref{eq:Bbar_n_k_upper_decomp} and \eqref{eq:cvg_Atil_Btil_upper} gives \eqref{eq:Bbar_k_upper} for $t < {\bar{\tau}}$.
Since \eqref{eq:Bbar_k_upper} holds for $t < {\bar{\tau}}$, it also holds for $t = {\bar{\tau}}$ by continuity of $\bar{\boldsymbol{B}}$ and of the right side in \eqref{eq:Bbar_k_upper}.	
	
	Now suppose $T\ge t > {\bar{\tau}}$.
	Since $r(\bar{\boldsymbol{X}}(\cdot))$ is continuous, we see from the definition of ${\bar{\tau}}$ that $r(\bar{\boldsymbol{X}}({\bar{\tau}})) = 0$.
	Noting that $r(\bar{\boldsymbol{X}}^n(\cdot))$ is non-negative and non-increasing, so is $r(\bar{\boldsymbol{X}}(\cdot))$.
	Therefore $r(\bar{\boldsymbol{X}}(t)) = 0$ and $\bar{\boldsymbol{X}}(t) = \boldsymbol{0}$ for ${\bar{\tau}} \le t \le T$.
	From this we see that the right hand side of \eqref{eq:Bbar_k_upper} remains constant for ${\bar{\tau}} \le t \le T$ and it suffices to show that $\bar{\boldsymbol{B}}(t) = \bar{\boldsymbol{B}}({\bar{\tau}})$ for ${\bar{\tau}} < t \le T$.
	 From \eqref{eq:Bbar_n_k_upper} it follows that, for $k \in \mathbb{N}$,
	\begin{equation}
		\label{eq:cvg_upper_Bbar_k}
		\sup_{{\bar{\tau}} < t \le T} |{\bar{B}}^n_k(t) - {\bar{B}}^n_k({\bar{\tau}})| = {\bar{B}}^n_k(T) - {\bar{B}}^n_k({\bar{\tau}}) = {\bar{X}}^n_k({\bar{\tau}}) - {\bar{X}}^n_k(T) \le {\bar{X}}^n_k({\bar{\tau}}),
	\end{equation}
	which converges to ${\bar{X}}_k({\bar{\tau}}) = 0$ as $n \to \infty$.
	Hence ${\bar{B}}_k(t) = {\bar{B}}_k({\bar{\tau}})$ for ${\bar{\tau}} < t \le T$ and this gives \eqref{eq:Bbar_k_upper} for each $k \in \mathbb{N}$.
	Next we show ${\bar{B}}_0(t) = {\bar{B}}_0({\bar{\tau}})$ for ${\bar{\tau}} < t \le T$.
	From \eqref{eq:Ybar_n_upper} and \eqref{eq:Xbar_n_0_upper},
	\begin{align}
		& \sup_{{\bar{\tau}} < t \le T} |{\bar{B}}^n_0(t) - {\bar{B}}^n_0({\bar{\tau}})| \notag \\
		& \le \sup_{{\bar{\tau}} < t \le T} |{\bar{X}}^n_0(t) - {\bar{X}}^n_0({\bar{\tau}})| + \sup_{{\bar{\tau}} < t \le T} |{\bar{\eta}}^n(t) -{\bar{\eta}}^n({\bar{\tau}})| + \sum_{k=1}^\infty |k-2| \sup_{{\bar{\tau}} < t \le T} |{\bar{B}}^n_k(t) - {\bar{B}}^n_k({\bar{\tau}})|. \label{eq:cvg_upper_Bbar_0}
	\end{align}
	Since ${\bar{X}}^n_0(t) \ge -\frac{1}{n}$, we have
	\begin{align*}
		\sup_{{\bar{\tau}} < t \le T} |{\bar{X}}^n_0(t) - {\bar{X}}^n_0({\bar{\tau}})| & \le \sup_{{\bar{\tau}} < t \le T} |{\bar{X}}^n_0(t)| + |{\bar{X}}^n_0({\bar{\tau}})| 
		 \le \sup_{{\bar{\tau}} < t \le T} ({\bar{X}}^n_0(t))^+ + \frac{1}{n} + ({\bar{X}}^n_0({\bar{\tau}}))^+ + \frac{1}{n} \\
		& \le \sup_{{\bar{\tau}} < t \le T} r(\bar{\boldsymbol{X}}^n(t)) + r(\bar{\boldsymbol{X}}^n({\bar{\tau}})) + \frac{2}{n} 
		 \le 2r(\bar{\boldsymbol{X}}^n({\bar{\tau}})) + \frac{2}{n},
	\end{align*}
	where the last line follows from the fact that $r(\bar{\boldsymbol{X}}^n(t))$ is non-increasing for $t \in [0,T]$.
	From \eqref{eq:etabar_n_upper} and \eqref{eq:Bbar_n_k_upper} it follows that
	\begin{align*}
		& \sup_{{\bar{\tau}} < t \le T} |{\bar{\eta}}^n(t) -{\bar{\eta}}^n({\bar{\tau}})| \\
		& = \sup_{{\bar{\tau}} < t \le T} 2 \sum_{k=1}^\infty \frac{1}{n} \int_{({\bar{\tau}},t] \times [0,1]}  {{1}}_{\{{\bar{X}}^n_0(u-) < 0\}} {{1}}_{[0,r_k(\bar{\boldsymbol{X}}^n(u-)))}(y) \, N^{n\varphi^n_k}_k(du\,dy) \\
		& \le \sup_{{\bar{\tau}} < t \le T} 2 \sum_{k=1}^\infty \frac{1}{n} \int_{({\bar{\tau}},t] \times [0,1]}  {{1}}_{[0,r_k(\bar{\boldsymbol{X}}^n(u-)))}(y) \, N^{n\varphi^n_k}_k(du\,dy) \\
		& = \sup_{{\bar{\tau}} < t \le T} 2 \sum_{k=1}^\infty |{\bar{B}}^n_k(t) - {\bar{B}}^n_k({\bar{\tau}})|. 
	\end{align*}
	Combining above two estimates with \eqref{eq:cvg_upper_Bbar_0}, we see that as $n \to \infty$,
	\begin{align}
		\sup_{{\bar{\tau}} < t \le T} |{\bar{B}}^n_0(t) - {\bar{B}}^n_0({\bar{\tau}})| & \le 2r(\bar{\boldsymbol{X}}^n({\bar{\tau}})) + \frac{2}{n} + \sup_{{\bar{\tau}} < t \le T} \sum_{k=1}^\infty (k+4) |{\bar{B}}^n_k(t) - {\bar{B}}^n_k({\bar{\tau}})| \nonumber\\
		& \le 2r(\bar{\boldsymbol{X}}^n({\bar{\tau}})) + \frac{2}{n} + \sum_{k=1}^\infty (k+4) {\bar{X}}^n_k({\bar{\tau}}) \le 7 r(\bar{\boldsymbol{X}}^n({\bar{\tau}})) + \frac{2}{n} \label{eq:middl}\\
		& \to 7 r(\bar{\boldsymbol{X}}({\bar{\tau}})) = 0,\nonumber
	\end{align}
	where the second inequality follows from \eqref{eq:cvg_upper_Bbar_k}.
Since we have proved \eqref{eq:Bbar_k_upper} for all $t < {\bar{\tau}}$, $t = {\bar{\tau}}$ and $t >{\bar{\tau}}$, part (c) follows.

	(d)
	From \eqref{eq:Xbar_0_upper} and a well known characterization of the solution of the Skorohod problem (see, e.g., \cite[Section 3.6.C]{KaratzasShreve1991brownian}), it suffices to show that ${\bar{\eta}}(0) = 0$, ${\bar{\eta}}(t) \ge 0$, ${\bar{\eta}}(t)$ is non-decreasing for $t \in [0,T]$ and $\int_0^T {\bar{X}}_0(t) \, {\bar{\eta}}(dt) = 0$.	
	Since ${\bar{\eta}}^n(0) = 0$, ${\bar{\eta}}^n(t) \ge 0$ and ${\bar{\eta}}^n(t)$ is non-decreasing for $t \in [0,T]$, so is ${\bar{\eta}}$.
	It remains to show $\int_0^T {\bar{X}}_0(t) \, {\bar{\eta}}(dt) = 0$.
	Note that ${\bar{\eta}}^n(t)$ increases only when ${\bar{X}}^n_0(t-)<0$, namely ${\bar{X}}^n_0(t-) = - \frac{1}{n}$.
	Therefore
	\begin{equation*}
		\int_0^T \left( {\bar{X}}^n_0(t-)+\frac{1}{n} \right) {\bar{\eta}}^n(dt) = 0.
	\end{equation*}
	From this we have
	\begin{equation}\label{eq:cvg_etabar_upper}
		\Scale[0.9]{\begin{aligned}
		\left| \int_0^T {\bar{X}}_0(t) \, {\bar{\eta}}(dt) \right| 
		& = \left| \int_0^T {\bar{X}}_0(t) \, {\bar{\eta}}(dt) - \int_0^T \left( {\bar{X}}^n_0(t-)+\frac{1}{n} \right) {\bar{\eta}}^n(dt) \right| \\
		& \le \left| \int_0^T {\bar{X}}_0(t) \, {\bar{\eta}}(dt) - \int_0^T {\bar{X}}_0(t) \, {\bar{\eta}}^n(dt) \right| + \int_0^T |{\bar{X}}_0(t) - {\bar{X}}^n_0(t-)| \, {\bar{\eta}}^n(dt) + \frac{{\bar{\eta}}^n(T)}{n}. 
	\end{aligned}}
	\end{equation}
	Since both ${\bar{\eta}}^n$ and ${\bar{\eta}}$ are non-decreasing, we see that ${\bar{\eta}}^n \to {\bar{\eta}}$ as finite measures on $[0,T]$.
	Combining this with the fact that ${\bar{X}}_0 \in \mathbb{C}_b([0,T]:\mathbb{R})$, we get $$\left| \int_0^T {\bar{X}}_0(t) \, {\bar{\eta}}(dt) - \int_0^T {\bar{X}}_0(t) \, {\bar{\eta}}^n(dt) \right| \to 0$$ as $n \to \infty$.
	Also from continuity of ${\bar{X}}_0$, we have uniform convergence of ${\bar{X}}^n_0$ to ${\bar{X}}_0$ and hence
	\begin{equation*}
		\int_0^T |{\bar{X}}_0(t) - {\bar{X}}^n_0(t-)| \, {\bar{\eta}}^n(dt) + \frac{{\bar{\eta}}^n(T)}{n} \le \left( \sup_{0 \le t \le T} |{\bar{X}}^n_0(t-)-{\bar{X}}_0(t)| + \frac{1}{n} \right) {\bar{\eta}}^n(T) \to 0
	\end{equation*}
	as $n \to \infty$.	
	Combining these two convergence results with \eqref{eq:cvg_etabar_upper},
	we see that
	$
		\int_0^T {\bar{X}}_0(t) \, {\bar{\eta}}(dt) = 0.
	$
	This proves part (d)  and  completes the proof.
\end{proof}

\section{Laplace upper bound}

\label{sec:upper}

In this section we prove the Laplace upper bound \eqref{eq:lappriupp}.

From \eqref{eq:mainrepn17}, for every $n \in \mathbb{N}$, we can choose ${\tilde{\boldsymbol{\varphi}}}^n \doteq ({\tilde{\varphi}}^n_k)_{k \in \mathbb{N}_0} \in 
\bar{\mathcal{A}}_b$ such that 
\begin{equation*}
-\frac{1}{n} \log {{E}} e^{-nh(\boldsymbol{X}^{n},Y^{n})} \ge {{E}} \left\{
\sum_{k=0}^\infty \int_{[0,T] \times [0,1]} \ell({\tilde{\varphi}}_k^n(s,y))
\, ds\,dy + h({\bar{\boldsymbol{X}}}^{n,{\tilde{\boldsymbol{\varphi}}}^n}, {\bar{Y}}^{n,{\tilde{\boldsymbol{\varphi}}}^n}) \right\} - \frac{1}{n},
\end{equation*}
where $(\bar{\boldsymbol{X}}^{n,{\tilde{\boldsymbol{\varphi}}}^n}, {\bar{Y}}^{n,{\tilde{\boldsymbol{\varphi}}}%
^n}) $ are defined by \eqref{eq:Ybar_n_upper_temp}--%
\eqref{eq:Xbar_n_k_upper_temp} by replacing $\boldsymbol{\varphi}^n$ with ${\tilde{\boldsymbol{\varphi}}}%
^n$. Since $\|h\|_\infty < \infty$, 
\begin{align*}
\sup_{n \in \mathbb{N}} {{E}} \sum_{k=0}^\infty \int_{[0,T] \times
[0,1]} \ell({\tilde{\varphi}}_k^n(s,y)) \, ds\,dy & \color{black} \le \sup_{n \in \mathbb{N}} \left[ -\frac{1}{n} \log {{E}} e^{-nh(\boldsymbol{X}^{n},Y^{n})} - E h({\bar{\boldsymbol{X}}}^{n,{\tilde{\boldsymbol{\varphi}}}^n}, {\bar{Y}}^{n,{\tilde{\boldsymbol{\varphi}}}^n}) + \frac{1}{n} \right] \\
& \le 2\|h\|_\infty + 1
\doteq M_h.
\end{align*}
Now we modify ${\tilde{\boldsymbol{\varphi}}}^n$ so that the last inequality holds not in
the sense of expectation, but rather almost surely, for a possibly larger
constant [see \eqref{eq:cost_bd_upper}]. Fix $\sigma \in (0,1)$ and define 
\begin{equation*}
{\tilde{\tau}}^n \doteq \inf \left\{t \in [0,T] : \sum_{k=0}^\infty
\int_{[0,t] \times [0,1]} \ell({\tilde{\varphi}}_k^n(s,y)) \, ds\,dy > 2 M_h
\|h\|_\infty / \sigma \right\} \wedge T.
\end{equation*}
For $k \in \mathbb{N}_0$, letting 
$
\varphi^n_k(s,y) \doteq {\tilde{\varphi}}^n_k(s,y) {{1}}_{\{s \le 
{\tilde{\tau}}^n\}} + {{1}}_{\{s > {\tilde{\tau}}^n\}}$, $(s,y) \in
[0,T]\times[0,1]$,
we have ${\boldsymbol{\varphi}}^n \doteq ({\varphi}^n_k)_{k \in \mathbb{N}_0} \in \bar{\mathcal{A}}_b$ since ${\tilde{\tau}}^n$
is an $\{\mathcal{F}_t\}$-stopping time. Also 
\begin{equation*}
{{E}} \sum_{k=0}^\infty \int_{[0,T] \times [0,1]}
\ell(\varphi_k^n(s,y)) \, ds\,dy \le {{E}} \sum_{k=0}^\infty
\int_{[0,T] \times [0,1]} \ell({\tilde{\varphi}}_k^n(s,y)) \, ds\,dy
\end{equation*}
and 
\begin{align*}
{P}(\boldsymbol{\varphi}^n \ne{\tilde{\boldsymbol{\varphi}}}^n) & \le {P}
\left( \sum_{k=0}^\infty \int_{[0,T] \times [0,1]} \ell({\tilde{\varphi}}%
_k^n(s,y)) \, ds\,dy > 2 M_h \|h\|_\infty / \sigma \right) \\
& \le \frac{\sigma}{2 M_h \|h\|_\infty} {{E}} \sum_{k=0}^\infty
\int_{[0,T] \times [0,1]} \ell({\tilde{\varphi}}_k^n(s,y)) \, ds\,dy 
 \le \frac{\sigma}{2 \|h\|_\infty}.
\end{align*}
Letting $(\bar{\boldsymbol{X}}^{n,{\boldsymbol{\varphi}}^n}, {\bar{Y}}^{n,{\boldsymbol{\varphi}}%
^n}) $ be defined through \eqref{eq:Ybar_n_upper_temp}--%
\eqref{eq:Xbar_n_k_upper_temp} using $\boldsymbol{\varphi}^n$, we have
\begin{equation*}
\left| {{E}} h(\bar{\boldsymbol{X}}^{n,\boldsymbol{\varphi}^n},{\bar{Y}}^{n,\boldsymbol{\varphi}^n}) - {%
{E}} h(\bar{\boldsymbol{X}}^{n,{\tilde{\boldsymbol{\varphi}}}^n},{\bar{Y}}^{n,{\tilde{\boldsymbol{\varphi}}%
}^n}) \right| \le 2 \|h\|_\infty {P}(\boldsymbol{\varphi}^n \ne{\tilde{%
\boldsymbol{\varphi}}}^n) \le \sigma.
\end{equation*}
Hence we have 
\begin{equation*}
-\frac{1}{n} \log {{E}} e^{-nh(\boldsymbol{X}^{n},Y^{n})} \ge {{E}} \left\{
\sum_{k=0}^\infty \int_{[0,T] \times [0,1]} \ell(\varphi_k^n(s,y)) \, ds\,dy
+ h(\bar{\boldsymbol{X}}^{n,\boldsymbol{\varphi}^n},{\bar{Y}}^{n,\boldsymbol{\varphi}^n}) \right\} - \frac{1}{n}
- \sigma
\end{equation*}
and 
\begin{equation}  \label{eq:cost_bd_upper_B}
\sup_{n \in \mathbb{N}} \sum_{k=0}^\infty \int_{[0,T] \times [0,1]}
\ell(\varphi_k^n(s,y)) \, ds\,dy \le 2M_h \|h\|_{\infty}/\sigma \doteq K_0, %
\mbox{ a.s. } {P}.
\end{equation}

Now we can complete the proof of the Laplace upper bound. Recall that 
$h \in 
\mathbb{C}_b(\mathcal{D}_\infty \times \mathcal{D})$.
Write $(\bar{\boldsymbol{\nu}}^{n},\bar{\boldsymbol{X}}^{n},{\bar{Y}}^{n}) \doteq (\bar{\boldsymbol{\nu}}^{n,\boldsymbol{\varphi}^n}, \bar{\boldsymbol{X}}^{n,{\boldsymbol{\varphi}}^n}, {\bar{Y}}^{n,{\boldsymbol{\varphi}}%
^n}) $, where $\bar{\boldsymbol{\nu}}^{n,\boldsymbol{\varphi}^n}$ is as defined in \eqref{eq:nu_n_upper} using $\boldsymbol{\varphi}^n$.
Noting from %
\eqref{eq:cost_bd_upper_B} that \eqref{eq:cost_bd_upper} is satisfied with $%
M_0= K_0$, we have from Lemma \ref{lem:tightness} that $\{(\bar{\boldsymbol{\nu}}^{n},%
\bar{\boldsymbol{X}}^{n},{\bar{Y}}^{n})\}$ is tight. Assume without
loss of generality that $(\bar{\boldsymbol{\nu}}^{n},\bar{\boldsymbol{X}}^{n},{\bar{Y}}%
^{n})$ converges along the whole sequence weakly to $(\bar{\boldsymbol{\nu}},\bar{\boldsymbol{X}}%
,{\bar{Y}})$, given on some probability space $(\Omega^*,\mathcal{F}^*,{%
P}^*)$. By Lemma \ref{lem:char_limit} we have $(\bar{\boldsymbol{X}},{\bar{%
Y}}) \in \mathcal{C}_T$ and $\bar{\boldsymbol{\nu}} = \bar{\boldsymbol{\nu}}^{\boldsymbol{\varphi}}$ for some $\boldsymbol{\varphi} \in 
\mathcal{S}_T(\bar{\boldsymbol{X}},{\bar{Y}})$ a.s.\ ${P}^*$, where $\bar{\boldsymbol{\nu}}^{\boldsymbol{\varphi}}$ is as defined in \eqref{eq:nu_n_upper} using $\boldsymbol{\varphi}$. 
Owing to the topology used for the measure component and the relation \eqref{eq:nu_n_upper},
 Lemma A.1 in \cite%
{BudhirajaChenDupuis2013large} (see also \cite[Appendix A.4.3, Lemma A.11]{buddupbook}) implies the lower semicontinuity of the cost that is needed for the second inequality below.
Using
Fatou's lemma and the definition of $I_T$ in \eqref{eq:rate_function}
\begin{align*}
\liminf_{n \to \infty} -\frac{1}{n} \log {{E}} e^{-nh(\boldsymbol{X}^{n},Y^{n})} & \ge
\liminf_{n \to \infty} {{E}} \left\{ \sum_{k=0}^\infty \int_{[0,T]
\times [0,1]} \ell(\varphi_k^n(s,y)) \, ds\,dy + h(\bar{\boldsymbol{X}}^n,{\bar{Y}}^n)
- \frac{1}{n} - \sigma \right\} \\
& \ge {{E}}^* \left\{ \sum_{k=0}^\infty \int_{[0,T] \times [0,1]}
\ell(\varphi_k(s,y)) \, ds\,dy + h(\bar{\boldsymbol{X}},{\bar{Y}}) \right\} - \sigma \\
& \ge \inf_{(\boldsymbol{\zeta} ,\psi ) \in \mathcal{D}_\infty \times \mathcal{D}}
\{I_T(\boldsymbol{\zeta} ,\psi ) + h(\boldsymbol{\zeta} ,\psi )\} - \sigma.
\end{align*}
Since $\sigma \in (0,1)$ is arbitrary, this
completes the proof of the Laplace upper bound.

\section{Laplace lower bound}

\label{sec:lower}

In this section we prove the Laplace lower bound \eqref{eq:lapprilow}.

The following lemma, which shows unique solvability of the ODE \eqref{eq:psi}
and \eqref{eq:phi_k} for controls $\boldsymbol{\varphi}$ in a suitable class, is key in
the proof.

\begin{Lemma}
\label{lem:uniqueness} Fix $\sigma \in (0,1)$. Given $(\boldsymbol{\zeta} ,\psi ) \in 
\mathcal{C}_T$ with $I_T(\boldsymbol{\zeta} ,\psi ) < \infty$, there exists $\boldsymbol{\varphi}^* \in 
\mathcal{S}_T(\boldsymbol{\zeta} ,\psi )$ such that

\begin{enumerate}[\upshape(a)]

\item $\sum_{k=0}^\infty \int_{[0,T] \times [0,1]} \ell(\varphi_k^*(s,y)) \,
ds\,dy \le I_T(\boldsymbol{\zeta} ,\psi ) + \sigma$.

\item If $(\tilde{\boldsymbol{\zeta}},\tilde\psi)$ is another pair in $\mathcal{C}_T$ such
that $\boldsymbol{\varphi}^* \in \mathcal{S}_T(\tilde{\boldsymbol{\zeta}},\tilde\psi)$, then $%
(\tilde{\boldsymbol{\zeta}},\tilde\psi) = (\boldsymbol{\zeta} ,\psi )$.
\end{enumerate}

\end{Lemma}

\begin{proof}
	Since $I_T(\boldsymbol{\zeta} ,\psi ) < \infty$, we can choose some $\boldsymbol{\varphi} \in \mathcal{S}_T(\boldsymbol{\zeta} ,\psi )$ such that
	\begin{equation*}
		\sum_{k=0}^\infty \int_{[0,T] \times [0,1]} \ell(\varphi_k(s,y)) \, ds\,dy \le I_T(\boldsymbol{\zeta} ,\psi ) + \frac{\sigma}{2}.
	\end{equation*}
	Next we will modify $\boldsymbol{\varphi}$ to get the desired $\boldsymbol{\varphi}^*$.
	\blue{
	For $k \in \Nmb_0$, let 
	\begin{align*}
		\rho_k(t) & \doteq 1_{\{r_k(\boldsymbol{\zeta}(t)) = 0\}} + \frac{\int_0^1 {{1}}_{[0,r_k(\boldsymbol{\zeta}(t)))}(y) \varphi_k(t,y)\,dy}{r_k(\boldsymbol{\zeta}(t))} 1_{\{r_k(\boldsymbol{\zeta}(t)) \ne 0\}}, \\
		\varphitil_k(t,y) & \doteq \rho_k(t) {{1}}_{[0,r_k(\boldsymbol{\zeta}(t)))}(y) + {{1}}_{[r_k(\boldsymbol{\zeta}(t)),1]}(y).
	\end{align*}
	Then $$\int_{[0,t]\times \lbrack 0,1]}{{1}}_{[0,r_{k}(\boldsymbol{\zeta} (s)))}(y)\,\varphitil_{k}(s,y)\,ds\,dy = \int_{[0,t]\times \lbrack 0,1]}{{1}}_{[0,r_{k}(\boldsymbol{\zeta} (s)))}(y)\,\varphi_{k}(s,y)\,ds\,dy$$ and hence $(\varphitil_k)_{k \in \Nmb_0} \in \Smc_T(\boldsymbol{\zeta},\psi)$.
	}	
	Since $\ell$ is convex and nonnegative and $\ell(1)=0$, 
	\blue{
	we have
	\begin{align*}
		\int_{[0,T] \times [0,1]} \ell(\varphitil_k(s,y)) \, ds\,dy & = \int_0^T 1_{\{r_k(\boldsymbol{\zeta}(s)) \ne 0\}} r_k(\boldsymbol{\zeta}(s)) \ell(\rho_k(s)) \, ds \le \int_{[0,T] \times [0,1]} \ell(\varphi_k(s,y)) \, ds\,dy
	\end{align*}
	for each $k \in \Nmb_0$.
	Therefore we can assume without loss of generality (and abusing notation) that
	$
		\varphi_k(t,y) = \rho_k(t) {{1}}_{[0,r_k(\boldsymbol{\zeta}(t)))}(y) + {{1}}_{[r_k(\boldsymbol{\zeta}(t)),1]}(y)
	$
	for some $\rho_k(t) \in [0,\infty)$, for each $k \in \mathbb{N}_0$ and $(t,y) \in [0,T]\times[0,1]$.
	}
	Fix $\varepsilon \in (0,1)$.  We will shrink the support of $\boldsymbol{\varphi}$ to get the desired $\boldsymbol{\varphi}^*$ for sufficiently small $\varepsilon$.
	For $t \in [0,T]$, let
	\begin{equation*}
		\varphi_k^\varepsilon(t,y) = \frac{\rho_k(t)}{1-\varepsilon} {{1}}_{[0,(1-\varepsilon)r_k(\boldsymbol{\zeta}(t)))}(y) + {{1}}_{[(1+\varepsilon)r_k(\boldsymbol{\zeta}(t)),1]}(y). 
	\end{equation*}
	Clearly $\boldsymbol{\varphi}^\varepsilon \in \mathcal{S}_T(\boldsymbol{\zeta} ,\psi )$.	
	Note that $\varphi_k^\varepsilon(t,y) = 0$ for $(1-\varepsilon)r_k(\boldsymbol{\zeta}(t)) < y < (1+\varepsilon)r_k(\boldsymbol{\zeta}(t))$, which will be a key when we prove uniqueness in part (b).
	Recall $\tau_{\boldsymbol{\zeta}}$ introduced in \eqref{eq:defntauphi}.
	Then
	\begin{align*}
		& \sum_{k=0}^\infty \int_{[0,T] \times [0,1]} \ell(\varphi_k^\varepsilon(t,y)) \, dt\,dy - \sum_{k=0}^\infty \int_{[0,T] \times [0,1]} \ell(\varphi_k(t,y)) \, dt\,dy \\
		& = \sum_{k=0}^\infty \int_0^{\tau_{\boldsymbol{\zeta}}} \left[ (1-\varepsilon)r_k(\boldsymbol{\zeta}(t)) \ell\left(\frac{\rho_k(t)}{1-\varepsilon}\right) + 2\varepsilon r_k(\boldsymbol{\zeta}(t)) \ell(0) - r_k(\boldsymbol{\zeta}(t)) \ell(\rho_k(t)) \right] dt \\
		& = \sum_{k=0}^\infty \int_0^{\tau_{\boldsymbol{\zeta}}} r_k(\boldsymbol{\zeta}(t)) \bigg[ \left( \rho_k(t) \log \left(\frac{\rho_k(t)}{1-\varepsilon}\right) - \rho_k(t) + 1-\varepsilon \right) + 2\varepsilon  \\
		& \qquad  - \left( \rho_k(t) \log \rho_k(t) - \rho_k(t) + 1 \right) \bigg] dt \\
		& = \sum_{k=0}^\infty \int_0^{\tau_{\boldsymbol{\zeta}}} r_k(\boldsymbol{\zeta}(t)) \left[ \rho_k(t) \log \left(\frac{1}{1-\varepsilon}\right) + \varepsilon \right] dt.
	\end{align*}
	From Lemma \ref{lem:property_ell}(b) we have
	\begin{align*}
		& \sum_{k=0}^\infty \int_{[0,T] \times [0,1]} \ell(\varphi_k^\varepsilon(t,y)) \, dt\,dy - \sum_{k=0}^\infty \int_{[0,T] \times [0,1]} \ell(\varphi_k(t,y)) \, dt\,dy \\
		& \le \sum_{k=0}^\infty \int_0^{\tau_{\boldsymbol{\zeta}}} r_k(\boldsymbol{\zeta}(t)) \left[ \left( \ell(\rho_k(t)) + 2 \right) \log \left(\frac{1}{1-\varepsilon}\right) + \varepsilon \right] dt \\
		& = \log \left(\frac{1}{1-\varepsilon}\right)\sum_{k=0}^\infty \int_{[0,T] \times [0,1]} \ell(\varphi_k(t,y)) \, dt\,dy  + 2 \tau_{\boldsymbol{\zeta}} \log \left(\frac{1}{1-\varepsilon}\right) + \tau_{\boldsymbol{\zeta}} \varepsilon \\
		& \le \left(I_T(\boldsymbol{\zeta} ,\psi ) + \frac{\sigma}{2}\right)\log \left(\frac{1}{1-\varepsilon}\right) + 2T\log \left(\frac{1}{1-\varepsilon}\right) + T \varepsilon .
	\end{align*}	
	Choosing $\varepsilon$ small enough so that the last display is no larger than $\frac{\sigma}{2}$, we have
	\begin{align*}
		\sum_{k=0}^\infty \int_{[0,T] \times [0,1]} \ell(\varphi_k^\varepsilon(s,y)) \, ds\,dy \le \sum_{k=0}^\infty \int_{[0,T] \times [0,1]} \ell(\varphi_k(s,y)) \, ds\,dy + \frac{\sigma}{2} \le I_T(\boldsymbol{\zeta} ,\psi ) + \sigma.
	\end{align*}
	Part (a) then holds with $\boldsymbol{\varphi}^* = \boldsymbol{\varphi}^\varepsilon$ for  such an $\varepsilon$.
	
	We now show that  part (b) is satisfied with such a $\boldsymbol{\varphi}^*$. Suppose that, in addition to $(\boldsymbol{\zeta} ,\psi )$, there is another pair of $({\tilde{\boldsymbol{\zeta}}},{\tilde{\psi}})$ such that $({\tilde{\boldsymbol{\zeta}}},{\tilde{\psi}}) \in \mathcal{C}_T$ and $\boldsymbol{\varphi}^* \in \mathcal{S}_T({\tilde{\boldsymbol{\zeta}}},{\tilde{\psi}})$.
	Let 
	$\tau \doteq \inf \{ t \in [0,T] : \boldsymbol{\zeta}(t) \ne {\tilde{\boldsymbol{\zeta}}}(t) \} \wedge T.$
	We claim that $\tau = T$.
	Once the claim is verified, it follows from continuity of $\boldsymbol{\zeta}$ and ${\tilde{\boldsymbol{\zeta}}}$ that $\boldsymbol{\zeta}(t) = {\tilde{\boldsymbol{\zeta}}}(t)$ for all $t \in [0,T]$.
	Then from \eqref{eq:psi},  $\psi = {\tilde{\psi}}$ proving part (b).
	
	Now we prove the claim that $\tau = T$. We will argue via contradiction.
	Suppose that $\tau < T$.
	To complete the proof, it suffices to reach the following contradiction
	\begin{equation}
		\label{eq:uniqueness_contradiction}
		\boldsymbol{\zeta}(t) = {\tilde{\boldsymbol{\zeta}}}(t), t \in [\tau,\tau+\delta] \mbox{ for some } \delta > 0.
	\end{equation}	
	From definition of $\tau$ and \eqref{eq:psi} it follows that $(\boldsymbol{\zeta}(t),r(\boldsymbol{\zeta}(t)),\psi(t)) = ({\tilde{\boldsymbol{\zeta}}}(t),r({\tilde{\boldsymbol{\zeta}}}(t)),{\tilde{\psi}}(t))$ for all $t < \tau$.
	From Remark \ref{rmk:property_ODE}(a) we have that $r(\boldsymbol{\zeta}(\cdot)), r({\tilde{\boldsymbol{\zeta}}}(\cdot)) \in \mathcal{C}$.
	Then by continuity, $(\boldsymbol{\zeta}(t),r(\boldsymbol{\zeta}(t)),\psi(t)) = ({\tilde{\boldsymbol{\zeta}}}(t),r({\tilde{\boldsymbol{\zeta}}}(t)),{\tilde{\psi}}(t))$ for all $t \le \tau$.
	If $r(\boldsymbol{\zeta}(\tau)) = r({\tilde{\boldsymbol{\zeta}}}(\tau)) = 0$, then from Remark \ref{rmk:property_ODE}(c) we have $\boldsymbol{\zeta}(t) = {\tilde{\boldsymbol{\zeta}}}(t) = \boldsymbol{0}$ for all $t \ge \tau$, which gives \eqref{eq:uniqueness_contradiction}.	
	Now we show \eqref{eq:uniqueness_contradiction} for the remaining case: $r(\boldsymbol{\zeta}(\tau)) = r({\tilde{\boldsymbol{\zeta}}}(\tau)) > 0$.
	For this, note that by continuity of $r(\boldsymbol{\zeta})$ and $r({\tilde{\boldsymbol{\zeta}}})$, there exists some $\delta > 0$ such that for all $t \in [\tau,\tau+\delta]$, 
	\begin{equation}
		\label{eq:uniqueness_key_fraction}
		r(\boldsymbol{\zeta}(t)) > 0, r({\tilde{\boldsymbol{\zeta}}}(t)) > 0, \left| \frac{r(\boldsymbol{\zeta}(t))}{r({\tilde{\boldsymbol{\zeta}}}(t))} - 1 \right| < \varepsilon,
	\end{equation}
	where $\varepsilon$ is as in part (a) and recall that $\boldsymbol{\varphi}^*=\boldsymbol{\varphi}^\varepsilon$.
	We will argue in two steps.\\
	
	Step $1$: We will prove that
	\begin{equation}
		\label{eq:step_1_claim}
		\zeta_k(t) = {\tilde{\zeta}}_k(t) \mbox{ for all } t \in [\tau,\tau+\delta], k \in \mathbb{N}.
	\end{equation}
	Suppose not, namely there exists $k\in \mathbb{N}$ such that
	$\tau_k \doteq \inf \{t \in [\tau,\tau+\delta] : \zeta_k(t) \ne {\tilde{\zeta}}_k(t) \}\wedge T$
	satisfies  $\tau \le \tau_k < \tau+\delta$.
	By continuity, we have $\zeta_k(t) = {\tilde{\zeta}}_k(t)$ for $t \le \tau_k$.
	We must have $\zeta_k(\tau_k) = {\tilde{\zeta}}_k(\tau_k) > 0$, since otherwise $\zeta_k(\tau_k) = {\tilde{\zeta}}_k(\tau_k) = 0$ and so from Remark \ref{rmk:property_ODE}(b)  $\zeta_k(t) = {\tilde{\zeta}}_k(t) = 0$ for all $t \ge \tau_k$, which contradicts the definition of $\tau_k$.
	From \eqref{eq:uniqueness_key_fraction} it then follows that
	\begin{align*}
		r_k(\boldsymbol{\zeta}(\tau_k)) & = \frac{k\zeta_k(\tau_k)}{r(\boldsymbol{\zeta}(\tau_k))} > 0, \\
		| r_k(\boldsymbol{\zeta}(\tau_k)) - r_k({\tilde{\boldsymbol{\zeta}}}(\tau_k)) | & = \left| \frac{k\zeta_k(\tau_k)}{r(\boldsymbol{\zeta}(\tau_k))} - \frac{k{\tilde{\zeta}}_k(\tau_k)}{r({\tilde{\boldsymbol{\zeta}}}(\tau_k))} \right| = \frac{k\zeta_k(\tau_k)}{r(\boldsymbol{\zeta}(\tau_k))} \left| 1 - \frac{r(\boldsymbol{\zeta}(\tau_k)}{r({\tilde{\boldsymbol{\zeta}}}(\tau_k))} \right| < \varepsilon r_k(\boldsymbol{\zeta}(\tau_k)).
	\end{align*}
Once more by continuity, there exists some $\delta_k>0$ such that last two inequalities hold for $t \in [\tau_k,\tau_k+\delta_k]$, namely $$r_k(\boldsymbol{\zeta}(t)) > 0, \: (1-\varepsilon)r_k(\boldsymbol{\zeta}(t)) < r_k({\tilde{\boldsymbol{\zeta}}}(t)) < (1+\varepsilon)r_k(\boldsymbol{\zeta}(t)).$$
	From construction of $\varphi^\varepsilon$, we see that for $t \in [\tau_k,\tau_k+\delta_k]$,
	\begin{equation*}
		\int_{(\tau_k,t] \times [0,1]} {{1}}_{[0,r_k({\tilde{\boldsymbol{\zeta}}}(s)))}(y) \varphi_k^\varepsilon(s,y) \, ds\,dy = \int_{(\tau_k,t] \times [0,1]}  {{1}}_{[0,r_k(\boldsymbol{\zeta}(s)))}(y) \varphi_k^\varepsilon(s,y) \, ds\,dy.
	\end{equation*}
	It then follows from \eqref{eq:phi_k} that $\zeta_k(t) = {\tilde{\zeta}}_k(t)$ for all $t \le \tau_k+\delta_k$.
	This contradicts the definition of $\tau_k$.
	Therefore \eqref{eq:step_1_claim} must hold.\\
	
	Step $2$: We will prove that
	\begin{equation}
		\label{eq:step_2_claim}
		\zeta_0(t) = {\tilde{\zeta}}_0(t) \mbox{ for all } t \in [\tau,\tau+\delta].
	\end{equation}
	Let $\eta(t) \doteq \zeta_0(t) - \psi(t)$ and ${\tilde{\eta}}(t) \doteq {\tilde{\zeta}}_0(t) - {\tilde{\psi}}(t)$.
	From properties of the Skorokhod map $\Gamma$ (see, e.g., \cite[Section 3.6.C]{KaratzasShreve1991brownian}), we have that
	\begin{align}
		& \eta(0) = 0, \eta(t) \mbox{ is non-decreasing and } \int_0^T \zeta_0(t) \, \eta(dt) = 0, \label{eq:step_2_eta} \\
		& {\tilde{\eta}}(0) = 0, {\tilde{\eta}}(t) \mbox{ is non-decreasing and } \int_0^T {\tilde{\zeta}}_0(t) \, {\tilde{\eta}}(dt) = 0 \label{eq:step_2_etatil}.
	\end{align}
	Consider $[\zeta_0(t) - {\tilde{\zeta}}_0(t)]^2$.
	Since $\zeta_0,\psi,{\tilde{\zeta}}_0,{\tilde{\psi}}$ are absolutely continuous, we have for $t \in [\tau,\tau+\delta]$,
	\begin{equation}
	\Scale[0.9]
	{\begin{aligned}
		 (\zeta_0(t) - {\tilde{\zeta}}_0(t))^2 
		& = (\zeta_0(\tau) - {\tilde{\zeta}}_0(\tau))^2 + 2 \int_\tau^t (\zeta_0(s) - {\tilde{\zeta}}_0(s)) (\zeta_0'(s) - {\tilde{\zeta}}_0'(s)) \, ds  \\
		& = 2 \int_\tau^t (\zeta_0(s) - {\tilde{\zeta}}_0(s)) (\psi'(s) - {\tilde{\psi}}'(s)) \, ds + 2 \int_\tau^t (\zeta_0(s) - {\tilde{\zeta}}_0(s)) (\eta_0'(s) - {\tilde{\eta}}_0'(s)) \, ds. 
	\end{aligned}}
	\label{eq:step_2_phi}
	\end{equation}
	From \eqref{eq:psi} and \eqref{eq:phi_k} we see that for $t \in [\tau,\tau+\delta]$,
	\begin{align*}
		\psi(t) & = \sum_{k=1}^\infty (k-2) (p_k-\zeta_k(t)) - 2\int_{[0,t] \times [0,1]} {{1}}_{[0,r_0(\boldsymbol{\zeta}(s)))}(y) \, \varphi_0^\varepsilon(s,y) ds\,dy, \\
		{\tilde{\psi}}(t) & = \sum_{k=1}^\infty (k-2) (p_k-{\tilde{\zeta}}_k(t)) - 2\int_{[0,t] \times [0,1]} {{1}}_{[0,r_0({\tilde{\boldsymbol{\zeta}}}(s)))}(y) \, \varphi_0^\varepsilon(s,y) ds\,dy.
	\end{align*}
	Taking the difference of these two displays and using \eqref{eq:step_1_claim}, we have that for $t \in [\tau,\tau+\delta]$,
	\begin{equation}
		\psi(t) - {\tilde{\psi}}(t) = - 2 \int_{[0,t] \times [0,1]} \left( {{1}}_{[0,r_0(\boldsymbol{\zeta}(s)))}(y) - {{1}}_{[0,r_0({\tilde{\boldsymbol{\zeta}}}(s)))}(y) \right) \varphi_0^\varepsilon(s,y) \, ds\,dy. \label{eq:step_2_psi}
	\end{equation}	
	Since for each fixed $y \ge 0$ the function $x\mapsto \frac{x}{x+y}$ is non-decreasing on $(-y,\infty)$, we have from \eqref{eq:step_1_claim} and \eqref{eq:uniqueness_key_fraction} that if for some
	$t \in [\tau,\tau+\delta]$,  $\zeta_0(t) \ge {\tilde{\zeta}}_0(t)$, then
	\begin{align*}
		\Scale[0.95]{r_0(\boldsymbol{\zeta}(t)) = \frac{\zeta_0(t)}{\zeta_0(t) + \sum_{k=1}^\infty k \zeta_k(t)} 
		= \frac{\zeta_0(t)}{\zeta_0(t) + \sum_{k=1}^\infty k {\tilde{\zeta}}_k(t)} \ge \frac{{\tilde{\zeta}}_0(t)}{{\tilde{\zeta}}_0(t) + \sum_{k=1}^\infty k {\tilde{\zeta}}_k(t)} = r_0({\tilde{\boldsymbol{\zeta}}}(t)).}	
	\end{align*}
	Therefore for $t \in [\tau,\tau+\delta]$,
	\begin{equation*}
		{{1}}_{[0,r_0(\boldsymbol{\zeta}(t)))}(y) \ge {{1}}_{[0,r_0({\tilde{\boldsymbol{\zeta}}}(t)))}(y) \mbox{ when } \zeta_0(t) \ge {\tilde{\zeta}}_0(t)
	\end{equation*}
	and similarly
	\begin{equation*}
		{{1}}_{[0,r_0(\boldsymbol{\zeta}(t)))}(y) \le {{1}}_{[0,r_0({\tilde{\boldsymbol{\zeta}}}(t)))}(y) \mbox{ when } \zeta_0(t) \le {\tilde{\zeta}}_0(t).
	\end{equation*}
	Combining these two inequalities with \eqref{eq:step_2_psi}, we see that
	\begin{equation}
		\label{eq:step_2_mono_1}
		(\zeta_0(s) - {\tilde{\zeta}}_0(s)) (\psi'(s) - {\tilde{\psi}}'(s)) \le 0, \mbox{ a.e. } s \in [\tau,\tau+\delta].
	\end{equation}
	Next from \eqref{eq:step_2_eta} and \eqref{eq:step_2_etatil} we see that for $t \in [\tau,\tau+\delta]$,
	\begin{align*}
		\int_\tau^t {{1}}_{\{\zeta_0(s) > {\tilde{\zeta}}_0(s)\}} (\zeta_0(s) - {\tilde{\zeta}}_0(s)) (\eta_0'(s) - {\tilde{\eta}}_0'(s)) \, ds & \le \int_\tau^t {{1}}_{\{\zeta_0(s) > {\tilde{\zeta}}_0(s)\}} (\zeta_0(s) - {\tilde{\zeta}}_0(s)) \eta_0'(s) \, ds \\
		& \le \int_\tau^t {{1}}_{\{\zeta_0(s) > 0\}} \zeta_0(s) \, \eta_0(ds) 
		 = 0,
	\end{align*}
	and similarly
	\begin{equation*}
		\int_\tau^t {{1}}_{\{\zeta_0(s) < {\tilde{\zeta}}_0(s)\}} (\zeta_0(s) - {\tilde{\zeta}}_0(s)) (\eta_0'(s) - {\tilde{\eta}}_0'(s)) \, ds \le 0.
	\end{equation*}
	Combining these two inequalities with \eqref{eq:step_2_mono_1} and \eqref{eq:step_2_phi}, we have for $t \in [\tau,\tau+\delta]$,
	$
		[\zeta_0(t) - {\tilde{\zeta}}_0(t)]^2 \le 0,
	$
	proving \eqref{eq:step_2_claim}.	
Combining \eqref{eq:step_1_claim} and \eqref{eq:step_2_claim}  gives \eqref{eq:uniqueness_contradiction} and completes the proof.	
\end{proof}

We can now complete the proof of the Laplace lower bound. Fix $h \in \mathbb{%
C}_b(\mathcal{D}_\infty \times \mathcal{D})$ and $\sigma \in (0,1)$. Fix
some $\sigma$-optimal $(\boldsymbol{\zeta}^*,\psi^*) \in \mathcal{C}_T$ with $%
I_T(\boldsymbol{\zeta}^*,\psi^*) < \infty$, namely 
\begin{equation*}
I_T(\boldsymbol{\zeta}^*,\psi^*) + h(\boldsymbol{\zeta}^*,\psi^*) \le \inf_{(\boldsymbol{\zeta} ,\psi ) \in \mathcal{D}%
_\infty \times \mathcal{D}} \left\{ I_T(\boldsymbol{\zeta} ,\psi ) + h(\boldsymbol{\zeta} ,\psi ) \right\} +
\sigma.
\end{equation*}
Let $\boldsymbol{\varphi}^* \in \mathcal{S}_T(\boldsymbol{\zeta}^*,\psi^*)$ be as in Lemma \ref%
{lem:uniqueness} (with $(\boldsymbol{\zeta},\psi)$ there replaced by $(\boldsymbol{\zeta}^*,\psi^*)$). 
For each $n \in \mathbb{N}$ and $(s,y) \in [0,T] \times [0,1]$, consider the
deterministic control 
\begin{align*}
\varphi^n_k(s,y) & \doteq \frac{1}{n} {{1}}_{\{\varphi^*_k(s,y)
\le \frac{1}{n}\}} + \varphi^*_k(s,y) {{1}}_{\{\frac{1}{n} <
\varphi^*_k(s,y) < n\}} + n {{1}}_{\{\varphi^*_k(s,y) \ge n\}}, k
\le n, \\
\varphi^n_k(s,y) & \doteq 1, k > n.
\end{align*}
Then $\boldsymbol{\varphi}^n \doteq (\varphi^n_k) \in \bar{\mathcal{A}}_b$ and from %
\eqref{eq:mainrepn17} we have 
\begin{equation*}
-\frac{1}{n} \log {{E}} e^{-nh(\boldsymbol{X}^{n},Y^{n})} \le {{E}} \left\{
\sum_{k=0}^\infty \int_{[0,T] \times [0,1]} \ell(\varphi_k^n(s,y)) \, ds\,dy
+ h({\bar{\boldsymbol{X}}}^{n},{\bar{Y}}^{n}) \right\},
\end{equation*}
where $({\bar{\boldsymbol{X}}}^{n}, {\bar{Y}}^{n})$ are given as in %
\eqref{eq:Ybar_n_upper_temp}--\eqref{eq:Xbar_n_k_upper_temp}. Noting that
for all $n \in \mathbb{N}$, $k \in \mathbb{N}_0$ and $(s,y) \in [0,T] \times
[0,1]$, $\ell(\varphi^n_k(s,y)) \le \ell(\varphi^*_k(s,y))$, we have from
Lemma \ref{lem:uniqueness}(a) that \eqref{eq:cost_bd_upper} holds with $M_0$
replaced by $I_T(\boldsymbol{\zeta}^*,\psi^*) + 1$. Define $\{\bar{\boldsymbol{\nu}}^{n}\}$ as in \eqref{eq:nu_n_upper} with controls $\boldsymbol{\varphi}^n$. From Lemma \ref{lem:tightness} it follows that $%
\{(\bar{\boldsymbol{\nu}}^{n},{\bar{\boldsymbol{X}}}^{n},{\bar{Y}}^{n})\}$ is
tight. Assume without loss of generality that $(\bar{\boldsymbol{\nu}}^{n},{\bar{\boldsymbol{X}}}^{n},{\bar{Y}}^{n})$ converges along the whole sequence
weakly to $(\bar{\boldsymbol{\nu}},{\bar{\boldsymbol{X}}},{\bar{Y}})$, given on some probability space $%
(\Omega^*,\mathcal{F}^*,{P}^*)$. From the construction of $%
\boldsymbol{\varphi}^n$ we must have $\bar{\boldsymbol{\nu}} = \bar{\boldsymbol{\nu}}^{\boldsymbol{\varphi}^*}$ a.s.\ ${P}^*$, where $\bar{\boldsymbol{\nu}}^{\boldsymbol{\varphi}^*}$ is as defined in \eqref{eq:nu_n_upper} using $\boldsymbol{\varphi}^*$.
By Lemma \ref{lem:char_limit} we have $({\bar{\boldsymbol{X}}},{\bar{Y}}) \in \mathcal{C}%
_T$ and $\boldsymbol{\varphi}^* \in \mathcal{S}_T({\bar{\boldsymbol{X}}},{\bar{Y}})$ a.s.\ ${%
P}^*$. From Lemma \ref{lem:uniqueness}(b) it now follows that $(%
{\bar{\boldsymbol{X}}},{\bar{Y}}) = (\boldsymbol{\zeta}^*,\psi^*)$ a.s.\ ${P}^*$. Finally,
from Lemma \ref{lem:uniqueness}(a), 
\begin{align*}
\limsup_{n \to \infty} -\frac{1}{n} \log {{E}} e^{-nh(\boldsymbol{X}^{n},Y^{n})} & \le
\limsup_{n \to \infty} {{E}} \left\{ \sum_{k=0}^\infty \int_{[0,T]
\times [0,1]} \ell(\varphi_k^n(s,y)) \, ds\,dy + h({\bar{\boldsymbol{X}}}^n,{\bar{Y}}^n)
\right\} \\
& \le \sum_{k=0}^\infty \int_{[0,T] \times [0,1]} \ell(\varphi_k^*(s,y)) \,
ds\,dy + {{E}}^* h({\bar{\boldsymbol{X}}}, \bar Y) \\
& = \sum_{k=0}^\infty \int_{[0,T] \times [0,1]} \ell(\varphi_k^*(s,y)) \,
ds\,dy + h(\boldsymbol{\zeta}^*,\psi^*) \\
& \le I_T(\boldsymbol{\zeta}^*,\psi^*) + h(\boldsymbol{\zeta}^*,\psi^*) + \sigma \\
& \le \inf_{(\boldsymbol{\zeta} ,\psi ) \in \mathcal{D}_\infty \times \mathcal{D}} \left\{
I_T(\boldsymbol{\zeta} ,\psi ) + h(\boldsymbol{\zeta} ,\psi ) \right\} + 2\sigma.
\end{align*}
Since $\sigma \in (0,1)$ is arbitrary, this completes the proof of the
Laplace lower bound.

\section{Compact Sub-level Sets}

\label{sec:rate_function}
In this section we prove that the function $I_T$ defined in %
\eqref{eq:rate_function} is a rate function, namely the set $\Gamma_N \doteq
\{ (\boldsymbol{\zeta} ,\psi ) \in \mathcal{D}_\infty \times \mathcal{D} : I_T(\boldsymbol{\zeta} ,\psi )
\le N \}$ is compact for each fixed $N \in [0,\infty)$.
Since the proof (as is usual) is very similar to the proof of the Laplace upper bound we will only provide details on steps that are significantly different.

Take any sequence $\{(\boldsymbol{\zeta}^n,\psi^n)_{n \in \mathbb{N}}\} \subset \Gamma_N$.
Then $(\boldsymbol{\zeta}^n,\psi^n) \in \mathcal{C}_T$ and there exists some $\frac{1}{n}$%
-optimal $\boldsymbol{\varphi}^n \in \mathcal{S}_T(\boldsymbol{\zeta}^n,\psi^n)$, namely 
\begin{equation}  \label{eq:cost_bd_rate}
\sum_{k=0}^\infty \int_{[0,T] \times [0,1]} \ell(\varphi_k^n(s,y)) \, ds\,dy
\le I_T(\boldsymbol{\zeta}^n,\psi^n) + \frac{1}{n} \le N + \frac{1}{n}.
\end{equation}
Recalling \eqref{eq:psi} and \eqref{eq:phi_k} and letting $\eta^n(t) \doteq
\zeta^n_0(t) - \psi^n(t)$, we can write for $t \in [0,T]$, 
\begin{equation}
\zeta^n_0(t) = \Gamma(\psi^n)(t) = \psi^n(t) + \eta^n(t) = \sum_{k=0}^\infty
(k-2) B_k^n(t) + \eta^n(t),  \label{eq:phi_n_0_rate}
\end{equation}
where 
\begin{equation}
B_k^n(t) \doteq \int_{[0,t] \times [0,1]} {{1}}%
_{[0,r_k(\boldsymbol{\zeta}^n(s)))}(y) \, \varphi_k^n(s,y) \,ds\,dy, k \in \mathbb{N}_0.
\label{eq:B_n_k_rate}
\end{equation}
From standard properties of the one-dimensional Skorokhod Problem we have 
\begin{equation}
\eta^n(0) = 0, \eta^n(t) \mbox{ is non-decreasing and } \int_0^T {%
{1}}_{\{\zeta^n_0(t)>0\}} \, \eta^n(dt) = 0.  \label{eq:eta_n_rate}
\end{equation}

%

Write $\boldsymbol{B}^n = (B^n_k)_{n \in \mathbb{N}_0}$ and let $\boldsymbol{\nu}^n$ be
defined as in \eqref{eq:nu_n_upper} with deterministic controls $\boldsymbol{\varphi}^n$. The following lemma shows that $%
\{(\boldsymbol{\nu}^n,\boldsymbol{\zeta}^n,\psi^n,\boldsymbol{B}^n,\eta^n)\}$ is pre-compact. The proof is
similar to that of Lemma \ref{lem:tightness} and is therefore omitted.

\begin{Lemma}
\label{lem:pre_compact_rate} $\{(\boldsymbol{\nu}^n,\boldsymbol{\zeta}^n,\psi^n,\boldsymbol{B}^n,\eta^n)\} 
$ is pre-compact in $[\mathcal{M}_{FC}([0,T]\times[0,1])]^\infty \times 
\mathcal{C}_\infty \times \mathcal{C} \times \mathcal{C}_\infty \times 
\mathcal{C}$.
\end{Lemma}

The following lemma characterizes limit points of $(\boldsymbol{\nu}^n,\boldsymbol{\zeta}^n,%
\psi^n,\boldsymbol{B}^n,\eta^n)$. 

\begin{Lemma}
\label{lem:cvg_rate} Suppose $(\boldsymbol{\nu}^n,\boldsymbol{\zeta}^n,\psi^n,\boldsymbol{B}^n,\eta^n)$
converges along a subsequence to $(\boldsymbol{\nu},\boldsymbol{\zeta},\psi,\boldsymbol{B},\eta) \in [\mathcal{M}%
([0,T]\times[0,1])]^\infty \times \mathcal{C}_\infty \times \mathcal{C}
\times \mathcal{C}_\infty \times \mathcal{C}$. Then the following hold.

\begin{enumerate}[\upshape(a)]

\item For each $k \in \mathbb{N}_0$, $\nu_k \ll \lambda_T$, and letting $%
\varphi_k \doteq \frac{d\nu_k}{d\lambda_T}$, 
$
\sum_{k=0}^\infty \int_{[0,T] \times [0,1]} \ell(\varphi_k(s,y)) \, ds\,dy
\le N.
$

\item For each $t\in[0,T]$, 
\begin{align*}  \label{eq:phi_k_rate}
\zeta_0(t) & = \Gamma(\psi)(t) = \psi(t) + \eta(t), \;\; \psi(t)  = \sum_{k=0}^\infty (k-2) B_k(t)\\
\zeta_k(t) & = p_k - B_k(t), \; k \in \mathbb{N}.
\end{align*}

\item For each $t\in[0,T]$, 
\begin{equation}
B_k(t) = \int_{[0,t] \times [0,1]} {{1}}_{[0,r_k(\boldsymbol{\zeta}(s)))}(y)
\, \varphi_k(s,y) ds\,dy, k \in \mathbb{N}_0,  \label{eq:B_k_rate}
\end{equation}
and in particular $(\boldsymbol{\zeta}, \psi) \in \mathcal{C}_T$ and $\boldsymbol{\varphi} \in \mathcal{S}%
_T(\boldsymbol{\zeta}, \psi)$.
\end{enumerate}
\end{Lemma}

\begin{proof}
	Assume without loss of generality that
	\begin{equation}
		\label{eq:cvg_rate_joint}
		(\boldsymbol{\nu}^n,\boldsymbol{\zeta}^n,\psi^n,\boldsymbol{B}^n,\eta^n) \to (\boldsymbol{\nu},\boldsymbol{\zeta},\psi,\boldsymbol{B},\eta)
	\end{equation} 
	as $n \to \infty$ along the whole sequence.
	Much of the proof is similar to that of Lemma \ref{lem:char_limit} except the proof of \eqref{eq:B_k_rate} for $k=0$. Thus we only give details for the latter statement.
	

%
	From \eqref{eq:cvg_rate_joint} and arguments similar to Lemma \ref{lem:UI_upper}  it follows that 
	\begin{equation*}
		r(\boldsymbol{\zeta}^n(t)) = (\zeta^n_0(t))^+ + \sum_{k=1}^\infty k \zeta^n_k(t) \to (\zeta_0(t))^+ + \sum_{k=1}^\infty k \zeta_k(t) = r(\boldsymbol{\zeta}(t))
	\end{equation*}
	uniformly in $t \in [0,T]$ as $n \to \infty$.
	Therefore $r(\boldsymbol{\zeta}(\cdot))$ is continuous.
	Let $\tau \doteq \inf \{ t \in [0,T] : r(\boldsymbol{\zeta}(t)) = 0\} \wedge T$.
	We will argue that \eqref{eq:B_k_rate}, for $k=0$, holds for all $t < \tau$, $t = \tau$ and $t > \tau$. The proof of the cases
	 $t < \tau$ and $t = \tau$ is similar to that of \eqref{eq:Bbar_k_upper} and is omitted.
	
	
	Now consider 	 $T\ge t > \tau$. 
	From \eqref{eq:eta_n_rate} and \eqref{eq:phi_n_0_rate}, for $\tau < t \le T$,
	\begin{equation*}
		|\eta^n(t) - \eta^n(\tau)| = \int_\tau^t \, d\eta^n(s) = \int_\tau^t {{1}}_{\{\zeta^n_0(s) = 0\}} \, d\eta^n(s) = \int_\tau^t {{1}}_{\{\zeta^n_0(s) = 0\}} \, \left (d\zeta^n_0(s) - \sum_{k=0}^\infty (k-2)dB^n_k(s)\right).
	\end{equation*}
	From \eqref{eq:B_n_k_rate} we see that  $\int_\tau^t {{1}}_{\{\zeta^n_0(s) = 0\}} \, dB^n_0(s) = 0$.
	Also since $\zeta_0^n$ is non-negative and absolutely continuous, we have ${{1}}_{\{\zeta^n_0(s) = 0\}} (\zeta^n_0)'(s) = 0$ for a.e.\ $s \in [0,T]$.
	Therefore
	\begin{equation*}
		|\eta^n(t) - \eta^n(\tau)| 
		\le \sum_{k=1}^\infty |k-2| |B^n_k(t) - B^n_k(\tau)|.
	\end{equation*}
	Applying the triangle inequality to \eqref{eq:phi_n_0_rate} and using this estimate, we see that
	\begin{align*}
		 \sup_{\tau < t \le T} |B^n_0(t) - B^n_0(\tau)| 
		 \le \sup_{\tau < t \le T} |\zeta^n_0(t) - \zeta^n_0(\tau)| + 2 \sum_{k=1}^\infty |k-2|\sup_{\tau < t \le T} |B^n_k(t) - B^n_k(\tau)|.
	\end{align*}
	Now as in the proof of \eqref{eq:middl} we have
	$\sup_{\tau < t \le T} |B^n_0(t) - B^n_0(\tau)|  \le 
							4 r(\boldsymbol{\zeta}^n(\tau))$,
which converges to $4 r(\boldsymbol{\zeta}(\tau))=0$ as $n \to \infty$.
	Hence $B_0(t) = B_0(\tau)$ for $\tau < t \le T$ and this gives \eqref{eq:B_k_rate} for $k=0$.
	
	Since we have proved \eqref{eq:B_k_rate} for $k=0$ and all $t < \tau$, $t = \tau$ and $t > \tau$, the proof is complete.
\end{proof}

\noindent\textbf{Proof of compact sub-level sets $\Gamma_M$:} Now we are
ready to prove that $\Gamma_M$ is compact for each fixed $M \in [0,\infty)$.
Recall $(\boldsymbol{\zeta}^n,\psi^n)$ introduced above %
\eqref{eq:cost_bd_rate} and $\boldsymbol{\nu}^n$ introduced above Lemma \ref{lem:pre_compact_rate}. From Lemma \ref{lem:pre_compact_rate} we have
pre-compactness of $\{(\boldsymbol{\nu}^n,\boldsymbol{\zeta}^n,\psi^n)\}$ in $[\mathcal{M}%
([0,T]\times[0,1])]^\infty \times \mathcal{C}_\infty \times \mathcal{C}$.
Assume without loss of generality that $(\boldsymbol{\nu}^n,\boldsymbol{\zeta}^n,\psi^n)$
converges along the whole sequence to some $(\boldsymbol{\nu},\boldsymbol{\zeta},\psi)$.
By Lemma \ref{lem:cvg_rate} $(\boldsymbol{\zeta},\psi) \in \mathcal{C}_T$ and $\boldsymbol{\nu}= \boldsymbol{\nu}^{\boldsymbol{\varphi}}$,
where for $k \in \mathbb{N}_0$, ${\nu}_k^{\boldsymbol{\varphi}}$ is as defined by the right side of  \eqref{eq:nu_n_upper} replacing ${\varphi}_k^n$ with ${\varphi}_k$,
and 
\begin{equation*}
I_T(\boldsymbol{\zeta} ,\psi ) \le \sum_{k=0}^\infty \int_{[0,T] \times [0,1]}
\ell(\varphi_k(s,y)) \, ds\,dy \le M.
\end{equation*}
Therefore $(\boldsymbol{\zeta} ,\psi ) \in \Gamma_M$ which proves that $\Gamma_M$ is
compact. 

\begin{Remark}
\label{rmk:unique_varphi} Suppose that for all $n \in \mathbb{N}$,  $(\boldsymbol{\zeta}^n,\psi^n) =
(\boldsymbol{\zeta} ,\psi )$ for some $(\boldsymbol{\zeta} ,\psi ) \in \mathcal{C}_T$ with $I_T(\boldsymbol{\zeta} ,\psi ) <
\infty$ and $M = I_T(\boldsymbol{\zeta} ,\psi )$. Then taking $\boldsymbol{\varphi}^n$  satisfying \eqref{eq:cost_bd_rate} (with $(\boldsymbol{\zeta}^n,\psi^n)$ replaced with $(\boldsymbol{\zeta} ,\psi )$),
we see from the above argument that there
exists some $\boldsymbol{\varphi} \in \mathcal{S}_T(\boldsymbol{\zeta} ,\psi )$ such that 
\begin{equation*}
I_T(\boldsymbol{\zeta} ,\psi ) \le \sum_{k=0}^\infty \int_{[0,T] \times [0,1]}
\ell(\varphi_k(s,y)) \, ds\,dy \le I_T(\boldsymbol{\zeta} ,\psi ),
\end{equation*}
namely $I_T(\boldsymbol{\zeta} ,\psi )$ is achieved at some $\boldsymbol{\varphi} \in \mathcal{S}%
_T(\boldsymbol{\zeta} ,\psi )$.
\end{Remark}

\section{Calculus of Variations Problem}
\label{sec:cal}

In this section we study a calculus of variations problem that is key in  proof of Theorem \ref{thm:ldg_degree_distribution}.
We begin by giving an overview of the proof strategy. Let $0 \le \boldsymbol{q} \le \boldsymbol{p}$. 
First note that, \cg{in view of Remark \ref{rmk:track-component} and since, as noted in Section \ref{sec:model},  $\{(nX^n_0(\sigma_j^n)+1, nX^n_k(\sigma_j^n)), k,j \in \mathbb{N}\}$ has the same distribution as
$\{A(j), V_k(j), k,j \in \mathbb{N}\}$, where $\{\sigma^n_j\}$ denote the jump instants of the process $\boldsymbol{X}^{n}$, 
the set $E^{n,\varepsilon}(\boldsymbol{q})$ can be written, in distributionally equivalent form (namely the probabilities of the events on the left and the right of the display below are the same), as}
\begin{align}
	E^{n,\varepsilon}(\boldsymbol{q}) 
	& = \{ \exists\, t_1, t_2 \in [0,\infty) \text{ such that } X^n_0(t_1-) = X^n_0(t_2) = - 1/n, X^n_0(t) > - 1/n \mbox{ for } t\in [t_1, t_2), \nonumber\\
	& \qquad |X_k^n(t_1-) - X_k^n(t_2) - q_k| \le \varepsilon \mbox{ for all } k \in \mathbb{N} \}. \label{eq:enverps}
\end{align}
Here $t_1$ (resp.\ $t_2$) corresponds to the time instant the first vertex (resp.\ the last edge) in a component is woken up (resp.\ is formed).

For $t_2 \ge t_1 \ge 0$ and $(\boldsymbol{\zeta},\psi) \in \mathcal{C}_{t_2}$, define
\begin{equation}
	\label{eq:I_t1_t2}
	I_{t_1,t_2}(\boldsymbol{\zeta},\psi) \doteq \inf_{\boldsymbol{\varphi} \in \mathcal{S}_{t_2}(\boldsymbol{\zeta},\psi)} \sum_{k=0}^\infty \int_{[t_1,t_2] \times [0,1]} \ell(\varphi_k(s,y)) \,ds\,dy.
\end{equation}
Further for $ \boldsymbol{x}^{(1)}, \boldsymbol{x}^{(2)} \in \mathbb{R}_+^\infty$, define
\begin{align*}
	& \mathcal{J}^0_{t_1,t_2} (\boldsymbol{x}^{(1)}, \boldsymbol{x}^{(2)}) \doteq \{ (\boldsymbol{\zeta},\psi) \in \mathcal{C}_{t_2}: \boldsymbol{\zeta}(t_1) = \boldsymbol{x}^{(1)}, \boldsymbol{\zeta}(t_2) = \boldsymbol{x}^{(2)} \}, \\
	& \mathcal{J}^1_{t_1,t_2}(\boldsymbol{x}^{(1)}, \boldsymbol{x}^{(2)}) \doteq \{ (\boldsymbol{\zeta},\psi) \in \mathcal{J}^0_{t_1,t_2} (\boldsymbol{x}^{(1)}, \boldsymbol{x}^{(2)}): \psi(t) \ge \psi(t_1)-x^{(1)}_0 \mbox{ for } t \in (t_1,t_2) \}, \\
	& \mathcal{J}^2_{t_1,t_2}(\boldsymbol{x}^{(1)}, \boldsymbol{x}^{(2)}) \doteq \{ (\boldsymbol{\zeta},\psi) \in \mathcal{J}^1_{t_1,t_2} (\boldsymbol{x}^{(1)}, \boldsymbol{x}^{(2)}): d r(\boldsymbol{\zeta}(t))/dt = -2 \mbox{ for a.e. } t \in (t_1,t_2) \},
\end{align*}
and
\begin{equation}
	\label{eq:def-I-j}
	I_{t_1,t_2}^j(\boldsymbol{x}^{(1)}, \boldsymbol{x}^{(2)}) \doteq \inf_{(\boldsymbol{\zeta},\psi) \in \mathcal{J}^j_{t_1,t_2}(\boldsymbol{x}^{(1)}, \boldsymbol{x}^{(2)})} I_{t_1,t_2}(\boldsymbol{\zeta},\psi), \quad j=0,1,2. 
\end{equation}
Here as usual, the infimum over an empty set is infinity.

 The proof of Theorem \ref{thm:ldg_degree_distribution} proceeds through the following steps.  Let $\tau \doteq \frac{1}{2} \sum_{k=1}^\infty kq_k$
 and assume $\sum_{k=1}^\infty k q_k  > 2\sum_{k=1}^\infty q_k$.
 \blue{Note that the limit as $\varepsilon \to 0$ in fact exists because the set $E^{n,\varepsilon}(\boldsymbol{q})$ is decreasing as $\varepsilon$ decreases.}
 \begin{itemize}
	 \item Lemma \ref{lem:puhalskii-lower-bound} shows the lower bound 
	 		\begin{equation}\label{star1227}
	 			\cg{\liminf_{\varepsilon \to 0}} \liminf_{n \to \infty } \frac{1}{n} \log P( E^{n,\varepsilon}(\boldsymbol{q})) \ge - I^2_{0,\tau}( (0,\boldsymbol{p}),  (0 ,\boldsymbol{p} - \boldsymbol{q})).
		\end{equation}
\item In Lemma \ref{lem:puhalskii-upper-bound}
	we show the upper bound 
	\begin{equation}
		\label{eq:upperbd_to_improve*}
		\cg{\limsup_{\varepsilon \to 0}} \limsup_{n \to \infty } \frac{1}{n} \log P( E^{n,\varepsilon}(\boldsymbol{q})) \le - \inf_{\boldsymbol{q} \le \bar{\boldsymbol{p}} \le \boldsymbol{p}, t_1 \ge 0}[ I^0_{0,t_1}((0,\boldsymbol{p}), (0,\bar{\boldsymbol{p}})) + I^2_{t_1,t_1+\tau}((0,\bar{\boldsymbol{p}}), (0,\bar{\boldsymbol{p}} - \boldsymbol{q}))].
	\end{equation}
\item Lemma \ref{lem:puhalskii-upper-bound-improvement} shows that when $p_1=0$  the upper and lower bounds coincide.
\item Finally Proposition \ref{prop:minimizer-summary} shows that 
$$I^2_{0,\tau}( (0,\boldsymbol{p}),  (0 ,\boldsymbol{p} - \boldsymbol{q})) = H({\boldsymbol{q}}) + H(\boldsymbol{p}-{\boldsymbol{q}}) - H({\boldsymbol{p}}) + K({\boldsymbol{q}})$$
completing the proof of Theorem \ref{thm:ldg_degree_distribution}.
\end{itemize}
Note that for $(\boldsymbol{\zeta},\psi) \in \mathcal{J}^1_{t_1,t_2}(\boldsymbol{x}^{(1)}, \boldsymbol{x}^{(2)})$,
$\zeta_0(t) = {x}_0^{(1)} + \psi(t)- \psi(t_1)$ for $t \in [t_1, t_2]$.
Intuitively, on the event $\{(\boldsymbol{X}^n, Y^n) \in \mathcal{J}^1_{t_1,t_2}(\boldsymbol{x}^{(1)}, \boldsymbol{x}^{(2)})\}$
the exploration remains in the same component over $[t_1,t_2]$, and on the smaller event $\{(\boldsymbol{X}^n, Y^n) \in \mathcal{J}^2_{t_1,t_2}(\boldsymbol{x}^{(1)}, \boldsymbol{x}^{(2)})\}$
 the  exploration pace  matches  that for the discrete-time exploration process (with time steps of length $1/n$), in which 
 at each   step $2$ half-edges are killed. The main idea in the proof of the theorem is that in characterizing the asymptotics of
 the probability of interest one can restrict to $\mathcal{J}^2_{0,\tau}( (0,\boldsymbol{p}),  (0 ,\boldsymbol{p} - \boldsymbol{q}))$,
 which roughly means that one can restrict to trajectories that avoid the boundary and whose evolution matches that of the original discrete time process of interest removing the artificial ``continuous time'' aspect of the evolution.

Define for $\boldsymbol{x} = (x_k)_{k \in \mathbb{N}_0} \in \mathbb{R}_+^\infty$ and $\boldsymbol{\beta} = (\beta_k)_{k \in \mathbb{N}_0} \in \mathbb{R} \times [-1,0]^\infty$ with $\sum_{k=1}^\infty \beta_k \ge -1$,
\begin{equation}
	\label{eq:L}
	\cg{L(\boldsymbol{x},\boldsymbol{\beta}) \doteq 
	 \sum_{k=0}^\infty \nu(k|\boldsymbol{\beta}) \log \left( \frac{\nu(k|\boldsymbol{\beta})}{\mu(k|\boldsymbol{x})} \right), \;
	   L_k(\boldsymbol{x},\boldsymbol{\beta}) \doteq \nu(k|\boldsymbol{\beta}) \log \left( \frac{\nu(k|\boldsymbol{\beta})}{\mu(k|\boldsymbol{x})} \right),}
\end{equation}
where
\begin{align}
	\nu(0|\boldsymbol{\beta}) & \doteq 1+\sum_{k=1}^\infty \beta_k, \quad \nu(k|\boldsymbol{\beta}) \doteq -\beta_k, \quad k \in \mathbb{N}, \label{eq:nu} \\
	\mu(k|\boldsymbol{x}) & \doteq r_k(\boldsymbol{x}), \boldsymbol{x} \ne \boldsymbol{0}, \quad \mu(k|\boldsymbol{x}) \doteq 1_{\{k=0\}}, \boldsymbol{x} = \boldsymbol{0}, \quad k \in \mathbb{N}_0. \label{eq:mu}
\end{align}
We set $L(\boldsymbol{x},\boldsymbol{\beta}) = \infty$, if $\boldsymbol{\beta}  \in \mathbb{R} \times [-1,0]^\infty$ and $\sum_{k=1}^\infty \beta_k < -1$.
Note that $\beta_0$ actually does not play a role in the definition of $L(\boldsymbol{x},\boldsymbol{\beta})$ or $\nu(\cdot|\boldsymbol{\beta})$.
Later on $(\boldsymbol{x},\boldsymbol{\beta})$ will be usually replaced by $(\boldsymbol{\zeta}(t),\boldsymbol{\zeta}'(t))$ for some absolutely continuous path \blue{$\boldsymbol{\zeta}=(\zeta_k)_{k \in \Nmb_0}$ and $t \ge 0$, where $\boldsymbol{\zeta}'(t) \doteq (\zeta_k'(t))_{k \in \Nmb_0}$}.


In the next six lemmas
$\boldsymbol{x}^{(1)} \doteq (x^{(1)}_0,\boldsymbol{p}^{(1)})$ and $\boldsymbol{x}^{(2)} \doteq (x^{(2)}_0,\boldsymbol{p}^{(2)})$ where $x^{(1)}_0,x^{(2)}_0 \in \mathbb{R}_+$ and $\zero \le \boldsymbol{p}^{(2)} \le \boldsymbol{p}^{(1)} \le \boldsymbol{p}$.
Let $\newbd \doteq \boldsymbol{x}^{(1)} - \boldsymbol{x}^{(2)}$. 
Define
\begin{equation}
	\label{eq:tau}
	\ttau(\boldsymbol{x}^{(1)},\boldsymbol{x}^{(2)}) \doteq \frac{1}{2} (r(\boldsymbol{x}^{(1)}) - r(\boldsymbol{x}^{(2)})) = \frac{1}{2} \left((x^{(1)}_0-x^{(2)}_0) + \sum_{k=1}^\infty k (p_k^{(1)}-p_k^{(2)})\right).
\end{equation}
\blue{We write $\ttau \equiv \ttau(\boldsymbol{x}^{(1)},\boldsymbol{x}^{(2)})$ for short in the next six lemmas.}
The following lemma relates $I^1,I^2$ and $L$.

\begin{Lemma}
	\label{lem:I1I2L}
	Fix $t_1 \in [0,\infty)$.
	Suppose $\ttau \ge 0$.
	Let $\boldsymbol{x}^{(0)} \doteq (0,\boldsymbol{p})$.
	Suppose there exists some $(\boldsymbol{\zeta}^*,\psi^*) \in \mathcal{J}^0_{0,t_1}(\boldsymbol{x}^{(0)}, \boldsymbol{x}^{(1)})$ such that $I_{0,t_1}(\boldsymbol{\zeta}^*,\psi^*) < \infty$.
	Then
	\begin{equation}
		\label{eq:I1I2}
		\inf_{t_2 \ge t_1} I^1_{t_1,t_2}(\boldsymbol{x}^{(1)}, \boldsymbol{x}^{(2)}) = I^2_{t_1,t_1+\ttau}(\boldsymbol{x}^{(1)}, \boldsymbol{x}^{(2)}).
	\end{equation}
	Furthermore, for $(\boldsymbol{\zeta},\psi) \in \mathcal{J}^2_{t_1,t_1+\ttau}(\boldsymbol{x}^{(1)}, \boldsymbol{x}^{(2)})$,
	\begin{equation}
		\label{eq:IL}
		I_{t_1,t_1+\ttau}(\boldsymbol{\zeta},\psi) = \int_{t_1}^{t_1+\ttau} L(\boldsymbol{\zeta}(s), \boldsymbol{\zeta}'(s))\,ds,
	\end{equation}
	and if $I_{t_1,t_1+\ttau}(\boldsymbol{\zeta},\psi)<\infty$, then $\sum_{k=1}^{\infty} \zeta'_k(t) \ge -1$ for a.e.\ $t \in [t_1, t_1+\ttau]$.
	In particular,
	\begin{equation}
		\label{eq:I1I2L}
		\inf_{t_2 \ge t_1} I^1_{t_1,t_2}(\boldsymbol{x}^{(1)}, \boldsymbol{x}^{(2)}) = I^2_{t_1,t_1+\ttau}(\boldsymbol{x}^{(1)}, \boldsymbol{x}^{(2)}) = \inf_{(\boldsymbol{\zeta},\psi) \in \mathcal{J}^2_{t_1,t_1+\ttau}(\boldsymbol{x}^{(1)}, \boldsymbol{x}^{(2)})} \int_{t_1}^{t_1+\ttau} L(\boldsymbol{\zeta}(s), \boldsymbol{\zeta}'(s))\,ds.
	\end{equation}
\end{Lemma}
\begin{Lemma}
	\label{lem:uniqbeta}
	Suppose that
	$\sum_{k=1}^\infty k \new_k + \new_0 > 2 \sum_{k=1}^\infty \new_k$ and, $x_0^{(2)} > 0$ or $\new_1 > 0$.
	Then there is a unique $\beta \equiv \beta (\boldsymbol{x}^{(1)}, \boldsymbol{x}^{(2)}) \in (0,1)$ such that
	\begin{equation}
		\label{eq:def-beta}
		\sum_{k=1}^\infty k \new_k = (1-\beta^2)\sum_{k=1}^\infty \frac{k \new_k}{1-\beta^k} + x_0^{(2)} - \beta^2 x_0^{(1)}.
	\end{equation}
\end{Lemma}

The construction given below will be used to give an explicit representation for the minimizer for the right side of
\eqref{eq:I1I2L}.

\begin{Construction}
	\label{cons:cont}
Suppose that either (i) or (ii) holds, where
\begin{enumerate}[\upshape(i)]
\item
 $x_0^{(2)}=0$ and $\new_1=0$.
\item $\sum_{k=1}^\infty k \new_k + \new_0 > 2 \sum_{k=1}^\infty \new_k$ and, $x_0^{(2)} > 0$ or $\new_1 > 0$.
\end{enumerate}
 Let $\beta \equiv \beta (\boldsymbol{x}^{(1)}, \boldsymbol{x}^{(2)}) \in [0,1)$ be $0$ in case (i) and 
 be the unique solution in $(0,1)$ of  \eqref{eq:def-beta} 
in case (ii) (as ensured by Lemma \ref{lem:uniqbeta}).
Note that $\beta$ satisfies \eqref{eq:def-beta} in both cases (i) and (ii).

Define $\ttau$ as in \eqref{eq:tau} and suppose that $\ttau \ge 0$. Let $\tilde \ttau \doteq \ttau/(1-\beta^2)$ and $\tilde \new_k \doteq \new_k/(1-\beta^k)$ for $k \in \mathbb{N}$.
Fix $t_1\ge 0$ and let $\boldsymbol{x}^{(0)}$,
	 $(\boldsymbol{\zeta}^*,\psi^*)$ be as in Lemma \ref{lem:I1I2L}.
Define $({\boldsymbol{\tilde\zeta}},\tilde{\psi})$ by $({\boldsymbol{\tilde\zeta}}(t),\tilde{\psi}(t)) = (\boldsymbol{\zeta}^*(t),\psi^*(t))$ for $t \in [0,t_1]$ and 	for $t \in [t_1,t_1+\ttau]$
\begin{align}
	\tilde{\zeta}_k(t) & \doteq p_k^{(1)} - \tilde \new_k \left[ 1- \left( 1-\frac{t-t_1}{\tilde \ttau}\right)^{k/2} \right], \quad k\in \mathbb{N}, \label{eq:def-minimizer} \\
	\tilde{\zeta}_0(t) & \doteq x_0^{(1)} + \sum_{k=1}^\infty k(p_k^{(1)}-\tilde{\zeta}_k(t)) - 2(t-t_1), \label{eq:def-minimizer-0} \\
	\tilde{\psi}(t) & \doteq \tilde{\psi}(t_1) + \sum_{k=1}^\infty k(p_k^{(1)}-\tilde{\zeta}_k(t)) - 2(t-t_1). \label{eq:def-minimizer-psi}
\end{align}
\end{Construction}

The next two lemmas give some properties of the various quantities  in the above construction.
Let
\begin{align*}
	\Xi &\doteq \Big\{(\boldsymbol{x}^{(1)}, \boldsymbol{x}^{(2)}): \mbox{ for } i= 1,2,\; \boldsymbol{x}^{(i)} \doteq (x^{(i)}_0,\boldsymbol{p}^{(i)}),
x^{(i)}_0 \in \mathbb{R}_+,\\
&\quad \quad   \zero \le \boldsymbol{p}^{(2)} \le \boldsymbol{p}^{(1)} \le \boldsymbol{p}
\mbox{ and } \sum_{k=1}^\infty k (p^{(1)}_k - p^{(2)}_k) + (x^{(1)}_0- x^{(2)}_0)> 2 \sum_{k=1}^\infty (p^{(1)}_k - p^{(2)}_k)\Big\}.
\end{align*}
We will equip $\Xi$ with the topology corresponding to coordinatewise convergence.
\begin{Lemma}
	\label{lem:lemctybeta}
	Both $\beta$ and $x_0^{(2)} \log \beta$ are continuous 
	on $\Xi$: for $(\boldsymbol{x}^{(1),n}, \boldsymbol{x}^{(2),n}) \in \Xi$ with 
	$(\boldsymbol{x}^{(1),n}, \boldsymbol{x}^{(2),n}) \to (\boldsymbol{x}^{(1)}, \boldsymbol{x}^{(2)}) \in \Xi$,
	$\beta^n \doteq \beta (\boldsymbol{x}^{(1),n}, \boldsymbol{x}^{(2),n}) \to \beta (\boldsymbol{x}^{(1)}, \boldsymbol{x}^{(2)}) \doteq \beta$ and $x_0^{(2),n} \log \beta^n \to x_0^{(2)} \log \beta$.
\end{Lemma}

\begin{Lemma}
	\label{lem:minimizer-def}
	Suppose that $\ttau \ge 0$.
	Also suppose that $\sum_{k=1}^\infty k \new_k + \new_0 > 2 \sum_{k=1}^\infty \new_k$. Fix $t_1\ge 0$.
	Let
	$(\boldsymbol{\zeta}^*,\psi^*)$ be as in Lemma \ref{lem:I1I2L} and $({\boldsymbol{\tilde\zeta}},\tilde{\psi})$ be as introduced in Construction \ref{cons:cont}.
	Then 
	\begin{enumerate}[\upshape(a)]
		\item
		$\ttau \le \tilde \ttau = \frac{1}{2} \left( x^{(1)}_0 + \sum_{k=1}^\infty k \tilde{\new}_k \right)$.
	\item
		$({\boldsymbol{\tilde\zeta}},\tilde{\psi}) \in \mathcal{J}^2_{t_1,t_1+\ttau}(\boldsymbol{x}^{(1)}, \boldsymbol{x}^{(2)})$.
	\item
		$\tilde{\zeta}_0(t) > 0$ for $t \in (t_1,t_1+\ttau)$. 
	\end{enumerate}
\end{Lemma}

The next lemma calculates $\int_{t_1}^{t_1+\ttau} L({\boldsymbol{\tilde\zeta}}(s), {\boldsymbol{\tilde\zeta}}'(s))\,ds$ for $({\boldsymbol{\tilde\zeta}},\tilde{\psi})$ introduced in Construction \ref{cons:cont}.
Recall that $\boldsymbol{z}= \boldsymbol{x}^{(1)}- \boldsymbol{x}^{(2)}$.

\begin{Lemma}
	\label{lem:minimizer-cost}
	Suppose that $\ttau \ge 0$. Suppose that either (i) or (ii) in Construction \ref{cons:cont} is satisfied. Also, let 
	$(\boldsymbol{\zeta}^*,\psi^*)$ be as in Lemma \ref{lem:I1I2L} with some $t_1\ge 0$ and 
	let $({\boldsymbol{\tilde\zeta}},\tilde{\psi})$ be given as in Construction \ref{cons:cont}.
	Define the function $\tilde{K}(\boldsymbol{x}^{(1)}, \boldsymbol{x}^{(2)})$ by
	\begin{equation*}
		\tilde{K}(\boldsymbol{x}^{(1)}, \boldsymbol{x}^{(2)}) \doteq \frac{\new_0+\sum_{k=1}^\infty k \new_k}{2} \log(1-\beta^2) - \sum_{k=1}^\infty \new_k \log(1-\beta^k) + x_0^{(2)} \log \beta.
	\end{equation*}
	For $\boldsymbol{x} \in \mathbb{R} \times \mathbb{R}_+^\infty$ such that $x_0+ \sum_{k=1}^{\infty} k x_k \ge 0$, define  $\tilde{H}(\boldsymbol{x})$ by
	\begin{equation*}
		\tilde{H}(\boldsymbol{x}) \doteq \sum_{k=1}^\infty x_k \log x_k - \frac{x_0+\sum_{k=1}^\infty k x_k}{2} \log \frac{x_0+\sum_{k=1}^\infty k x_k}{2}.
	\end{equation*}
	Then
	\begin{equation*}
		\int_{t_1}^{t_1+\ttau} L({\boldsymbol{\tilde\zeta}}(s), {\boldsymbol{\tilde\zeta}}'(s))\,ds = \tilde{H}(\boldsymbol{\new}) + \tilde{H}(\boldsymbol{x}^{(2)}) - \tilde{H}(\boldsymbol{x}^{(1)}) + \tilde{K}(\boldsymbol{x}^{(1)}, \boldsymbol{x}^{(2)}) < \infty.
	\end{equation*}
	Moreover, the right hand side is lower semicontinuous in $(\boldsymbol{x}^{(1)}, \boldsymbol{x}^{(2)}) \in \Xi$, namely
	for $(\boldsymbol{x}^{(1),n}, \boldsymbol{x}^{(2),n}) \in \Xi$ with 
	$(\boldsymbol{x}^{(1),n}, \boldsymbol{x}^{(2),n}) \to (\boldsymbol{x}^{(1)}, \boldsymbol{x}^{(2)}) \in \Xi$,
	\begin{align*}
		&\liminf_{n\to \infty} \left(\tilde{H}(\boldsymbol{\new}^n) + \tilde{H}(\boldsymbol{x}^{(2),n}) - \tilde{H}(\boldsymbol{x}^{(1),n}) + \tilde{K}(\boldsymbol{x}^{(1),n}, \boldsymbol{x}^{(2),n})\right)\\
		&\quad \ge \tilde{H}(\boldsymbol{\new}) + \tilde{H}(\boldsymbol{x}^{(2)}) - \tilde{H}(\boldsymbol{x}^{(1)}) + \tilde{K}(\boldsymbol{x}^{(1)}, \boldsymbol{x}^{(2)}),
	\end{align*}
	where $\boldsymbol{\new}^n = \boldsymbol{x}^{(1),n}-\boldsymbol{x}^{(2),n}$, $\boldsymbol{\new}=\boldsymbol{x}^{(1)}-\boldsymbol{x}^{(2)}$.
\end{Lemma}
Recall the functions $H$ and $K$ from \eqref{eq:hdefn} and \eqref{eq:kdefnn} respectively. We note that with $\tilde K$ and $\tilde H$ as introduced in the above lemma, for $\zero \le \boldsymbol{q} \le \bar{\boldsymbol{q}} \le \boldsymbol{p}$
\begin{equation}
	\label{eq:hktil}
	H(\boldsymbol{q}) = \tilde H(0, \boldsymbol{q}), \; K(\boldsymbol{q}) = \tilde{K}((0,\bar{\boldsymbol{q}}), (0,\bar{\boldsymbol{q}}-\boldsymbol{q})).
\end{equation}
The next lemma shows that $({\boldsymbol{\tilde\zeta}},\tilde{\psi})$  in Construction \ref{cons:cont} is a minimizer for the right side of \eqref{eq:I1I2L}.

\begin{Lemma}
	\label{lem:minimizer-verify-general}
	Suppose that $\ttau \ge 0$. Suppose that $\sum_{k=1}^\infty k \new_k + \new_0 > 2 \sum_{k=1}^\infty \new_k$.
	Fix $t_1\ge 0$ and let
	$(\boldsymbol{\zeta}^*,\psi^*)$ be as in Lemma \ref{lem:I1I2L} and $({\boldsymbol{\tilde\zeta}},\tilde{\psi})$ as introduced in Construction \ref{cons:cont}.
	Then 
	\begin{equation}
		I^2_{t_1,t_1+\ttau}(\boldsymbol{x}^{(1)}, \boldsymbol{x}^{(2)}) = \inf_{(\boldsymbol{\zeta},\psi) \in \mathcal{J}^2_{t_1,t_1+\ttau}(\boldsymbol{x}^{(1)}, \boldsymbol{x}^{(2)})} \int_{t_1}^{t_1+\ttau} L(\boldsymbol{\zeta}(s), \boldsymbol{\zeta}'(s))\,ds = \int_{t_1}^{t_1+\ttau} L({\boldsymbol{\tilde\zeta}}(s), {\boldsymbol{\tilde\zeta}}'(s))\,ds.
		\label{eq:i2ttau}
	\end{equation}
\end{Lemma}

Proofs of Lemmas \ref{lem:I1I2L}--\ref{lem:minimizer-verify-general} are given in Section \ref{sec:pfsect4}.
The following proposition summarizes an important consequence of the above lemmas  for the case when $x_0^{(1)}=x_0^{(2)}=0$.

\begin{Proposition}
	\label{prop:minimizer-summary}
	$\,$
	\phantomsection
	\begin{enumerate}[(a)]
	\item
		Suppose $\zero \le \boldsymbol{q} \le \bar{\boldsymbol{q}} \le  \boldsymbol{p}$ and that either $\sum_{k=1}^\infty kq_k > 2\sum_{k=1}^\infty q_k$
		or 
		$\sum_{k=1}^\infty kq_k = 2\sum_{k=1}^\infty q_k$ but $p_1=0$.
		Given $t_1 \ge 0$,  and with $\boldsymbol{x}^{(0)} \doteq (0, \boldsymbol{p})$, $\boldsymbol{x}^{(1)} \doteq (0, \bar{\boldsymbol{q}})$,
		suppose  there exists some $(\boldsymbol{\zeta}^*,\psi^*) \in \mathcal{J}^0_{0,t_1}(\boldsymbol{x}^{(0)}, \boldsymbol{x}^{(1)})$ such that $I_{0,t_1}(\boldsymbol{\zeta}^*,\psi^*) < \infty$.
		Then 
		\[\Scale[0.9]{\begin{aligned}
			\inf_{t_2 \ge t_1} I^1_{t_1,t_2}((0,\bar{\boldsymbol{q}}), (0,\bar{\boldsymbol{q}}-\boldsymbol{q})) = I^2_{t_1,t_1+\tau}((0,\bar{\boldsymbol{q}}), (0,\bar{\boldsymbol{q}}-\boldsymbol{q}))
			= H({\boldsymbol{q}}) + H(\bar{\boldsymbol{q}}-{\boldsymbol{q}}) - H(\bar{\boldsymbol{q}}) + K({\boldsymbol{q}}),
		\end{aligned}}\]
		where \blue{$\tau \doteq \ttau((0,\bar{\boldsymbol{q}}), (0,\bar{\boldsymbol{q}}-\boldsymbol{q})) = \frac{1}{2} \sum_{k=1}^\infty kq_k$}.
	\item
		Suppose $p_1=0$, $\boldsymbol{q} \ge 0$, $\bar{\boldsymbol{q}} \ge \zero$, $\boldsymbol{q} + \bar{\boldsymbol{q}} \le \boldsymbol{p}$, $\sum_{k=1}^\infty kq_k \ge 2\sum_{k=1}^\infty q_k$, and $\sum_{k=1}^\infty k\bar{q}_k \ge 2\sum_{k=1}^\infty \bar{q}_k$.
		Let \blue{$\tau \doteq \ttau((0,\boldsymbol{p}), (0,\boldsymbol{p}-\boldsymbol{q})) = \frac{1}{2} \sum_{k=1}^\infty kq_k$} and $\bar{\tau} \doteq \frac{1}{2} \sum_{k=1}^\infty k \bar{q}_k$.
		Then
		\[\begin{aligned}
			 &I^2_{0,\bar \tau}((0,\boldsymbol{p}), (0,\boldsymbol{p} - \bar{\boldsymbol{q}})) + I^2_{\bar \tau, \bar \tau + \tau}((0,\boldsymbol{p} - \bar{\boldsymbol{q}}), (0,\boldsymbol{p} - \bar{\boldsymbol{q}} - \boldsymbol{q}))\\
			 &\quad= I^2_{0, \tau}((0,\boldsymbol{p}), (0,\boldsymbol{p} -  \boldsymbol{q})) + I^2_{ \tau,  \tau +\bar \tau}((0,\boldsymbol{p} -  \boldsymbol{q}), (0,\boldsymbol{p} -  \boldsymbol{q} - \bar{\boldsymbol{q}})).
		\end{aligned}\]
	\end{enumerate}	
\end{Proposition}

\begin{proof}
	(a) The first equality in part (a) is a consequence of Lemma \ref{lem:I1I2L}. 
	For the second equality, consider first the case $\sum_{k=1}^\infty kq_k > 2\sum_{k=1}^\infty q_k$.
	From \eqref{eq:hktil} we have
	\[\Scale[0.9]{\begin{aligned}
		 H({\boldsymbol{q}}) + H(\bar{\boldsymbol{q}}-{\boldsymbol{q}}) - H(\bar{\boldsymbol{q}}) + K({\boldsymbol{q}}) 
		 = \tilde{H}(0,{\boldsymbol{q}}) + \tilde{H}(0,\bar{\boldsymbol{q}}-{\boldsymbol{q}}) - \tilde{H}(0,\bar{\boldsymbol{q}}) + \tilde{K}((0,\bar{\boldsymbol{q}}), (0,\bar{\boldsymbol{q}}-\boldsymbol{q})).
	\end{aligned}}\]
	Applying Lemma \ref{lem:minimizer-cost} with $\boldsymbol{x}^{(1)} = (0, \bar{\boldsymbol{q}})$, $\boldsymbol{x}^{(2)}
	= (0, \bar{\boldsymbol{q}}- \boldsymbol{q})$, the above expression equals
	$ \int_{t_1}^{t_1+\tau} L({\boldsymbol{\tilde\zeta}}(s), {\boldsymbol{\tilde\zeta}}'(s))\,ds$
	where ${\boldsymbol{\tilde\zeta}}$ is defined by \eqref{eq:def-minimizer} -- \eqref{eq:def-minimizer-psi}.
	Now from Lemma \ref{lem:minimizer-verify-general} 
	$$I^2_{t_1,t_1+\tau}((0, \bar{\boldsymbol{q}}), (0, \bar{\boldsymbol{q}}- \boldsymbol{q})) = H({\boldsymbol{q}}) + H(\bar{\boldsymbol{q}}-{\boldsymbol{q}}) - H(\bar{\boldsymbol{q}}) + K({\boldsymbol{q}})$$
	which proves the second equality in part (a) for the considered case.
	
	Now we consider the case $\sum_{k=1}^\infty kq_k = 2\sum_{k=1}^\infty q_k$ and $p_1=0$.
	Since $p_1=0$, we must have $q_k=0$ for each $k \ne 2$.
	Then for any $(\boldsymbol{\zeta},\psi) \in \mathcal{J}^2_{t_1,t_1+\tau}((0, \bar{\boldsymbol{q}}), (0, \bar{\boldsymbol{q}}- \boldsymbol{q}))$ with $I_{t_1,t_1+\tau}(\boldsymbol{\zeta},\psi)<\infty$, we must have (see \eqref{eq:psi} and the definition of $\mathcal{J}^2_{t_1,t_2}$) $\zeta_2'(t)=-1$ and $\zeta_k'(t)=\psi'(t)=0$, $k \ne 2$ for $t \in [t_1,t_1+\tau]$.
	Also, in this case $q_1=0$ and so we are in case (i) of Construction \ref{cons:cont} with $\boldsymbol{x}^{(1)} = (0, \bar{\boldsymbol{q}})$ and
	$\boldsymbol{x}^{(2)} = (0, \bar{\boldsymbol{q}}- \boldsymbol{q})$.
	It is easily checked that 
	any $(\boldsymbol{\zeta},\psi)$ with the above properties 
	is same as the minimizer $({\boldsymbol{\tilde\zeta}},\tilde{\psi})$  over
	$[t_1, t_1+\tau]$.
		%
		%
		%
	Thus using Lemma \ref{lem:I1I2L} and  Lemma \ref{lem:minimizer-cost} we get 
	\begin{align*}
		I^2_{t_1, t_1+\tau}((0,\bar{\boldsymbol{q}}), (0,\bar{\boldsymbol{q}} -  \boldsymbol{q})) & = 
		\inf_{(\boldsymbol{\zeta},\psi) \in \mathcal{J}^2_{t_1,t_1+\tau}(\boldsymbol{x}^{(1)}, \boldsymbol{x}^{(2)})} \int_{t_1}^{t_1+\tau} L(\boldsymbol{\zeta}(s), \boldsymbol{\zeta}'(s))\,ds\\
		&=\int_{t_1}^{t_1+\tau} L({\boldsymbol{\tilde\zeta}}(s),{\boldsymbol{\tilde\zeta}}'(s)) \, ds \\
		& = \tilde{H}(0,{\boldsymbol{q}}) + \tilde{H}(0,{\bar{\boldsymbol{q}}}-{\boldsymbol{q}}) - \tilde{H}(0,{\bar{\boldsymbol{q}}}) + \tilde{K}((0,{\bar{\boldsymbol{q}}}), (0,{\bar{\boldsymbol{q}}}-\boldsymbol{q})) \\
		& = H({\boldsymbol{q}}) + H({\bar{\boldsymbol{q}}}-{\boldsymbol{q}}) - H({\bar{\boldsymbol{q}}}) + K({\boldsymbol{q}}).
	\end{align*}
	This proves part (a) in this case.
	
	(b) 
	From part (a),
	$$I^2_{0,\bar \tau}((0,\boldsymbol{p}), (0,\boldsymbol{p} - \bar{\boldsymbol{q}})) = H({\bar{\boldsymbol{q}}}) + H({\boldsymbol{p}}-{\bar{\boldsymbol{q}}}) - H({\boldsymbol{p}}) + K(\bar{\boldsymbol{q}})$$
	and since the right side is finite, again from part (a),
	$$I^2_{\bar \tau, \bar \tau + \tau}((0,\boldsymbol{p} - \bar{\boldsymbol{q}}), (0,\boldsymbol{p} - \bar{\boldsymbol{q}} - \boldsymbol{q})) = H({\boldsymbol{q}}) + H({\boldsymbol{p}}-\bar{\boldsymbol{q}}-\boldsymbol{q}) - H({\boldsymbol{p}}-\bar{\boldsymbol{q}}) + K({\boldsymbol{q}}).$$
	Therefore,
	\begin{align}
		& I^2_{0,\bar \tau}((0,\boldsymbol{p}), (0,\boldsymbol{p} - \bar{\boldsymbol{q}})) + I^2_{\bar \tau, \bar \tau + \tau}((0,\boldsymbol{p} - \bar{\boldsymbol{q}}), (0,\boldsymbol{p} - \bar{\boldsymbol{q}} - \boldsymbol{q})) \notag \\
		& = \left[ H({\bar{\boldsymbol{q}}}) + H({\boldsymbol{p}}-{\bar{\boldsymbol{q}}}) - H({\boldsymbol{p}}) + K(\bar{\boldsymbol{q}}) \right] + \left[ H({\boldsymbol{q}}) + H({\boldsymbol{p}}-\bar{\boldsymbol{q}}-\boldsymbol{q}) - H({\boldsymbol{p}}-\bar{\boldsymbol{q}}) + K({\boldsymbol{q}}) \right] \notag \\
		& = \left[ H({\boldsymbol{q}}) + H({\boldsymbol{p}}-{\boldsymbol{q}}) - H({\boldsymbol{p}}) + K({\boldsymbol{q}}) \right] + \left[ H({\bar{\boldsymbol{q}}}) + H({\boldsymbol{p}}-\boldsymbol{q}-\bar{\boldsymbol{q}}) - H({\boldsymbol{p}}-\boldsymbol{q}) + K(\bar{\boldsymbol{q}}) \right] \notag \\
		& = I^2_{0, \tau}((0,\boldsymbol{p}), (0,\boldsymbol{p} -  \boldsymbol{q})) + I^2_{ \tau,  \tau +\bar \tau}((0,\boldsymbol{p} -  \boldsymbol{q}), (0,\boldsymbol{p} -  \boldsymbol{q} - \bar{\boldsymbol{q}})), \label{eq:minimizer-summary-pf}
	\end{align}
	where the last line follows, once more, from  (a). This proves  (b) and completes the proof.	
\end{proof}

\section{Proof of Theorem \ref{thm:ldg_degree_distribution}}
\label{sec:pf_LDP_degree}

In this section we will use Theorem \ref{thm:main-ldp} and results in Section \ref{sec:cal} to prove Theorem \ref{thm:ldg_degree_distribution}.
Let $0 \le \boldsymbol{q} \le \boldsymbol{p}$. Recall the \cg{(distributionally equivalent)} representation of the event $E^{n,\varepsilon}(\boldsymbol{q})$ given in \eqref{eq:enverps}, \blue{in terms of $\Xbd^n$}.
Define
\begin{align}
	E^{n,\varepsilon,T}(\boldsymbol{q}) & \doteq \{ \exists\, t_1, t_2 \in [0,T] \text{ such that } X^n_0(t_1-) = X^n_0(t_2) = - 1/n, X^n_0(t) > - 1/n \mbox{ for } t\in [t_1, t_2), \notag \\
	& \qquad |X_k^n(t_1-) - X_k^n(t_2) - q_k| \le \varepsilon \mbox{ for all } k \in \mathbb{N} \} \notag \\
	& = \{ \exists\, t_1, t_2 \in [0,T] \text{ such that } X^n_0(t_1-) = X^n_0(t_2) = - 1/n, \notag \\
	& \qquad Y^n(t) > Y^n(t_1-) - 2/n \mbox{ for } t\in [t_1, t_2), |X_k^n(t_1-) - X_k^n(t_2) - q_k| \le \varepsilon \mbox{ for all } k \in \mathbb{N} \}. \label{eq:app_E_n_eps}
\end{align}
Note that $E^{n,\varepsilon,T}(\boldsymbol{q}) \subset E^{n,\varepsilon}(\boldsymbol{q})$ but they are not equal, since the continuous-time EEA may not terminate by time $T$.
Consider the event that the continuous-time EEA terminates before time  $T$, namely the event $F^{n,T}$ defined as
\begin{equation}
	\label{eq:F-n-T}
	F^{n,T} \doteq \{X^n(T)=(-{1}/{n},\boldsymbol{0})\}.
\end{equation}
Then
\begin{equation}
	\label{eq:E-Etil}
	E^{n,\varepsilon}(\boldsymbol{q}) \cap F^{n,T} \subset E^{n,\varepsilon,T}(\boldsymbol{q}) \subset E^{n,\varepsilon}(\boldsymbol{q}).
\end{equation}
The following lemma guarantees that in order to study the exponential rate of decay of $P(E^{n,\varepsilon}(\boldsymbol{q}))$, it suffices to study that of $P(E^{n,\varepsilon,T}(\boldsymbol{q}))$.

\begin{Lemma}
	\label{lem:choosing-T}
	$\limsup_{n \to \infty} \frac{1}{n} \log P((F^{n,T})^c) \to -\infty$ as $T \to \infty$.
\end{Lemma}

\begin{proof}
	Recall from Section \ref{sec:eea} that  the discrete-time EEA terminates in at most $nN$ steps where $N \doteq \lfloor \sup_n \frac{1}{2} \sum_{k=1}^\infty k \frac{n_k}{n} \rfloor + 1 < \infty$.
	So since the discrete time EEA is the embedded chain associated with the continuous time EEA (see Section \ref{sec:model}), $\boldsymbol{X}^n$ will have at most $nN$ jumps before arriving at the absorbing state $(-\frac{1}{n},\boldsymbol{0})$. 
	Since the total jump rate for $\boldsymbol{X}^n(t)$ at any instant  before getting absorbed is $n\sum_{k=0}^\infty r_k(\boldsymbol{X}^n(t)) = n$, we have
	\begin{equation*}
		P(F^{n,T}) \ge P \left( \sum_{i=1}^{nN} \xi_i \le T \right) = P \left( \frac{1}{n} \sum_{i=1}^{nN} \tilde{\xi}_i \le T \right),
	\end{equation*}
	where $\xi_i$ are  i.i.d. $\exp(n)$  and $\tilde{\xi}_i$ are  i.i.d. $\exp(1)$.
	Therefore
	\begin{align*}
		\limsup_{n \to \infty} \frac{1}{n} \log P((F^{n,T})^c) & \le 
		 N \limsup_{n \to \infty} \frac{1}{nN} \log \mathbb{P} \left( \frac{1}{nN} \sum_{i=1}^{nN} \tilde{\xi}_i > \frac{T}{N} \right) \\
		& = -N L_1 \left( \frac{T}{N} \right) \to -\infty
	\end{align*}
	as $T \to \infty$, where the second equality is  from Cram\'{e}r's theorem and where
	 $L_1(x) \doteq x - 1 - \log x$ for $x \ge 0$ is the Legendre transform of the log-moment generating function of $\tilde{\xi}_1$.
\end{proof}

The following lemma gives an upper bound for the exponential rate of decay of 
$P( E^{n,\varepsilon}(\boldsymbol{q}) )$.

\begin{Lemma}
	\label{lem:puhalskii-upper-bound}
	Suppose  $\sum_{k=1}^\infty k q_k  > 2\sum_{k=1}^\infty q_k$.
	Then the upper bound in \eqref{eq:upperbd_to_improve*} holds, namely
		\begin{equation}
			\label{eq:upperbd_to_improve}
			\cg{\limsup_{\varepsilon \to 0}} \limsup_{n \to \infty } \frac{1}{n} \log P( E^{n,\varepsilon}(\boldsymbol{q})) \le - \inf_{\boldsymbol{q} \le \bar{\boldsymbol{p}} \le \boldsymbol{p}, t_1 \ge 0}[ I^0_{0,t_1}((0,\boldsymbol{p}), (0,\bar{\boldsymbol{p}})) + I^2_{t_1,t_1+\tau}((0,\bar{\boldsymbol{p}}), (0,\bar{\boldsymbol{p}} - \boldsymbol{q}))],
		\end{equation}
		\blue{where $\tau \doteq \ttau((0,\bar{\boldsymbol{p}}), (0,\bar{\boldsymbol{p}} - \boldsymbol{q})) = \frac{1}{2} \sum_{k=1}^\infty kq_k$ for each $\boldsymbol{q} \le \bar{\boldsymbol{p}} \le \boldsymbol{p}$.}
\end{Lemma}

\begin{proof} 
	From \eqref{eq:E-Etil} we have
	\begin{equation*}
		P( E^{n,\varepsilon}(\boldsymbol{q})) \le P (E^{n,\varepsilon,T}(\boldsymbol{q})) + P ( (F^{n,T})^c)
	\end{equation*}	
	and hence
	\begin{align*}
		 \limsup_{n \to \infty } \frac{1}{n} \log P( E^{n,\varepsilon}(\boldsymbol{q})) 
		 \le \max\left\{ \limsup_{n \to \infty } \frac{1}{n} \log P( E^{n,\varepsilon,T}(\boldsymbol{q})), \limsup_{n \to \infty} \frac{1}{n} \log P ( (F^{n,T})^c)\right\}.
	\end{align*}
	In view of Lemma \ref{lem:choosing-T}, it suffices to show that for all sufficiently large $T$
	\begin{equation*}
		\cg{\limsup_{\varepsilon \to 0}} \limsup_{n \to \infty } \frac{1}{n} \log P( E^{n,\varepsilon,T}(\boldsymbol{q})) \le - \inf_{\boldsymbol{q}\le \bar{\boldsymbol{p}} \le \boldsymbol{p}, t_1 \ge 0}[ I^0_{0,t_1}((0,\boldsymbol{p}), (0,\bar{\boldsymbol{p}})) + I^2_{t_1,t_1+\tau}((0,\bar{\boldsymbol{p}}), (0,\bar{\boldsymbol{p}} - \boldsymbol{q}))].
	\end{equation*}
Let $\mathfrak{P}_T \doteq \mathbb{D}([0,T]:\mathbb{R} \times \mathbb{R}_+^\infty \times \mathbb{R})$ and consider
	\begin{align*}
		\tilde{E}^{\varepsilon,T}(\boldsymbol{q}) &
		\doteq \{ (\boldsymbol{\zeta},\psi) \in \mathfrak{P}_T : \exists\, t_1, t_2 \in [0,T] \text{ such that } \zeta_0(t_1-) = \zeta_0(t_2) \le 0, \\
		& \qquad \psi(t) \ge \psi(t_1-) - \varepsilon \mbox{ for } t \in [t_1, t_2), |\zeta_k(t_1-) - \zeta_k(t_2) -q_k| \le \varepsilon \mbox{ for all } k \in \mathbb{N}\}.
	\end{align*}
	Denote the closure of $\tilde{E}^{\varepsilon,T}(\boldsymbol{q})$ by $cl\tilde{E}^{\varepsilon,T}(\boldsymbol{q})$.
	From the definition  in \eqref{eq:app_E_n_eps}, when $n > 2 \varepsilon^{-1}$
	\begin{equation*}
		E^{n,\varepsilon,T}(\boldsymbol{q}) \subset \{ (\boldsymbol{X}^n,Y^n) \in \tilde{E}^{\varepsilon,T}(\boldsymbol{q})\} \subset \{ (\boldsymbol{X}^n,Y^n) \in cl\tilde{E}^{\varepsilon,T}(\boldsymbol{q})\}.
	\end{equation*}
	From this and Theorem \ref{thm:main-ldp} we have
	\begin{equation*}
		\limsup_{n \to \infty } \frac{1}{n} \log P( E^{n,\varepsilon,T}(\boldsymbol{q})) \le \limsup_{n \to \infty } \frac{1}{n} \log P( (\boldsymbol{X}^n,Y^n) \in cl\tilde{E}^{\varepsilon,T}(\boldsymbol{q})) \le - \inf_{(\boldsymbol{\zeta},\psi) \in cl\tilde{E}^{\varepsilon,T}(\boldsymbol{q})} I_T(\boldsymbol{\zeta},\psi).
	\end{equation*}
	Since $I_T(\boldsymbol{\zeta},\psi) < \infty$ only when $(\boldsymbol{\zeta},\psi) \in \mathcal{C}_T$, we have
	\begin{align*}
		\limsup_{n \to \infty } \frac{1}{n} \log P( E^{n,\varepsilon,T}(\boldsymbol{q})) \le - \inf_{(\boldsymbol{\zeta},\psi) \in cl\tilde{E}^{\varepsilon,T}(\boldsymbol{q}) \cap \mathcal{C}_T} I_T(\boldsymbol{\zeta},\psi).
	\end{align*}
	
	It is easy to see that $cl\tilde{E}^{\varepsilon,T}(\boldsymbol{q}) \cap \mathcal{C}_T = \tilde{E}^{\varepsilon,T}(\boldsymbol{q}) \cap \mathcal{C}_T$.
	Thus we have
	\begin{align*}
		\limsup_{n \to \infty } \frac{1}{n} \log P( E^{n,\varepsilon,T}(\boldsymbol{q})) \le - \inf_{(\boldsymbol{\zeta},\psi) \in \tilde{E}^{\varepsilon,T}(\boldsymbol{q}) \cap \mathcal{C}_T} I_T(\boldsymbol{\zeta},\psi).
	\end{align*}
	Letting
	\begin{align*}
		\tilde{E}^T(\boldsymbol{q}) & 
		\doteq\{ (\boldsymbol{\zeta},\psi) \in \mathcal{C}_T : \exists\, t_1, t_2 \in [0,T] \text{ such that } \zeta_0(t_1) = \zeta_0(t_2) \le 0, \\
		& \qquad \psi(t) \ge \psi(t_1) \mbox{ for } t \in [t_1, t_2), \zeta_k(t_1) - \zeta_k(t_2) = q_k \mbox{ for all } k \in \mathbb{N}\},
	\end{align*}
	we have $\tilde{E}^T(\boldsymbol{q}) = \bigcap_{\varepsilon > 0} \left( \tilde{E}^{\varepsilon,T}(\boldsymbol{q}) \cap \mathcal{C}_T \right)$. 
	From this, the lower semi-continuity and compactness of level sets of $I_T(\boldsymbol{\zeta},\psi)$ (since $I_T$ is a rate function; see Theorem \ref{thm:main-ldp}), it follows
	\begin{equation*}
		\cg{\limsup_{\varepsilon \to 0}}	\limsup_{n \to \infty } \frac{1}{n} \log P( E^{n,\varepsilon,T}(\boldsymbol{q})) \le - \cg{\liminf_{\varepsilon \to 0}} \inf_{(\boldsymbol{\zeta},\psi) \in \tilde{E}^{\varepsilon,T}(\boldsymbol{q}) \cap \mathcal{C}_T} I_T(\boldsymbol{\zeta},\psi) =  - \inf_{(\boldsymbol{\zeta},\psi) \in \tilde{E}^T(\boldsymbol{q})} I_T(\boldsymbol{\zeta},\psi).
	\end{equation*}
	Breaking $(\boldsymbol{\zeta},\psi) \in \tilde{E}^T(\boldsymbol{q})$ for $t \in [0,T]$  according to $t \le t_1$, $t_1 \le t \le t_2$ and $t \ge t_2$, 
	\begin{align*}
		\inf_{(\boldsymbol{\zeta},\psi) \in \tilde{E}^T(\boldsymbol{q})} I_T(\boldsymbol{\zeta},\psi) 
		& = \inf_{\boldsymbol{q} \le \bar{\boldsymbol{p}} \le \boldsymbol{p}, 0 \le t_1 < t_2 \le T} [I^0_{0,t_1}((0,\boldsymbol{p}), (0,\bar{\boldsymbol{p}})) + I^1_{t_1,t_2}((0,\bar{\boldsymbol{p}}), (0,\bar{\boldsymbol{p}} - \boldsymbol{q}))] \\
		& \ge \inf_{\boldsymbol{q} \le \bar{\boldsymbol{p}} \le \boldsymbol{p}, 0 \le t_1 < t_2<\infty} [I^0_{0,t_1}((0,\boldsymbol{p}), (0,\bar{\boldsymbol{p}})) + I^1_{t_1,t_2}((0,\bar{\boldsymbol{p}}), (0,\bar{\boldsymbol{p}} - \boldsymbol{q}))] \\
		& = \inf_{\boldsymbol{q} \le \bar{\boldsymbol{p}} \le \boldsymbol{p}, t_1 \ge 0} [I^0_{0,t_1}((0,\boldsymbol{p}), (0,\bar{\boldsymbol{p}})) + I^2_{t_1, t_1 + \tau}((0,\bar{\boldsymbol{p}}), (0,\bar{\boldsymbol{p}} - \boldsymbol{q}))],
	\end{align*}
	where the last line follows from Lemma \ref{lem:I1I2L}.
	The result follows.
\end{proof}

The following lemma improves the upper bound \eqref{eq:upperbd_to_improve} in Lemma \ref{lem:puhalskii-upper-bound} when $p_1=0$.

\begin{Lemma}
	\label{lem:puhalskii-upper-bound-improvement}
	Suppose $p_1=0$ and $\zero \le \boldsymbol{q} \le \boldsymbol{p}$.
	Then
	\begin{enumerate}[\upshape(a)]
	\item
		$I^0_{0,t_1}((0,\boldsymbol{p}),(0,{\boldsymbol{q}})) = I^1_{0,t_1}((0,\boldsymbol{p}),(0,{\boldsymbol{q}}))$ for each $t_1 \ge 0$.
	\item
		Let $\tau \doteq \frac{1}{2} \sum_{k=1}^\infty kq_k$ as in Lemma \ref{lem:puhalskii-upper-bound}.
		The infimum on the right side of \eqref{eq:upperbd_to_improve} is achieved at $t_1=0$:
		\begin{equation*}
			\inf_{\boldsymbol{q} \le \bar{\boldsymbol{p}} \le \boldsymbol{p}, t_1 \ge 0}[ I^0_{0,t_1}((0,\boldsymbol{p}), (0,\bar{\boldsymbol{p}})) + I^2_{t_1,t_1+\tau}((0,\bar{\boldsymbol{p}}), (0,\bar{\boldsymbol{p}} - \boldsymbol{q}))] = I^2_{0,\tau}((0,\boldsymbol{p}), (0,\boldsymbol{p} - \boldsymbol{q})).
		\end{equation*}
	\end{enumerate}
\end{Lemma}

\begin{proof}
	(a)
	Fix $t_1 \ge 0$.
	It suffices to show that if $(\boldsymbol{\zeta},\psi) \in \mathcal{J}^0_{0,t_1}((0,\boldsymbol{p}),(0,{\boldsymbol{q}}))$ satisfies
	$I_{0,t_1}(\boldsymbol{\zeta},\psi)<\infty$ then $(\boldsymbol{\zeta},\psi) \in \mathcal{J}^1_{0,t_1}((0,\boldsymbol{p}),(0,{\boldsymbol{q}}))$.
	%
	%
	%
	For such a pair of $(\boldsymbol{\zeta},\psi)$, let $\boldsymbol{\varphi} \in \mathcal{S}_{t_1}(\boldsymbol{\zeta},\psi)$ be such that the associated cost is finite.
	In particular $\psi$, and consequently $\zeta_0$, is absolutely continuous.
	
	Since $\zeta_0(t) = \Gamma(\psi)(t) \ge 0$ for $t \in [0,t_1]$, we have
	\begin{equation}
		\label{eq:puhalskii-upper-bound-improvement-1}
		1_{\{\zeta_0(t)>0\}}\zeta_0'(t) = 1_{\{\zeta_0(t)>0\}}\psi'(t), \quad 1_{\{\zeta_0(t)=0\}}\zeta_0'(t) = 0, \mbox{ a.e. } t \in [0,t_1].
	\end{equation}
	Since $p_1=0$, we see from \eqref{eq:psi} that $1_{\{\zeta_0(t)=0\}}\psi'(t) \ge 0$ for a.e.\ $t \in [0,t_1]$. Indeed, when $p_1=0$ the term for $k=1$ in the sum on the right side of \eqref{eq:psi} is zero. Also, the term for $k=2$ is always zero and the integrand for $k=0$ is zero on the set $\{\zeta_0(t)=0\}$. This shows that,
	on this set, the derivative of the sum on the right side of \eqref{eq:psi} is nonnegative.
Combining this with \eqref{eq:puhalskii-upper-bound-improvement-1}, we have for $t \in [0,t_1]$,
	\begin{align*}
		\psi(t) & = \int_0^t 1_{\{\zeta_0(s)>0\}}\psi'(s)\,ds + \int_0^t 1_{\{\zeta_0(s)=0\}}\psi'(s)\,ds \ge \int_0^t 1_{\{\zeta_0(s)>0\}}\zeta_0'(s)\,ds \\
		& = \int_0^t 1_{\{\zeta_0(s)>0\}}\zeta_0'(s)\,ds + \int_0^t 1_{\{\zeta_0(s)=0\}}\zeta_0'(s)\,ds = \zeta_0(t) \ge 0.
	\end{align*}
	This implies $(\boldsymbol{\zeta},\psi) \in \mathcal{J}^1_{0,t_1}((0,\boldsymbol{p}),(0,\boldsymbol{q}))$
	and part (a) follows.
	
	(b)
	For $\boldsymbol{0} \le \bar{\boldsymbol{p}} \le \boldsymbol{p}$, let $\bar{\boldsymbol{q}} \doteq \boldsymbol{p} - \bar{\boldsymbol{p}}$ and  $\bar{\tau} \doteq \frac{1}{2} \sum_{k=1}^\infty k \bar{q}_k$.
	Since $p_1=0$, we always have $\sum_{k=1}^\infty k q_k \ge 2\sum_{k=1}^\infty q_k$ and $\sum_{k=1}^\infty k \bar{q}_k \ge 2\sum_{k=1}^\infty \bar{q}_k$.
	Therefore
	\begin{align*}
		& \inf_{\boldsymbol{q} \le \bar{\boldsymbol{p}} \le \boldsymbol{p}, t_1 \ge 0}[ I^0_{0,t_1}((0,\boldsymbol{p}), (0,\bar{\boldsymbol{p}})) + I^2_{t_1,t_1+\tau}((0,\bar{\boldsymbol{p}}), (0,\bar{\boldsymbol{p}} - \boldsymbol{q}))] \\
		& = \inf_{\boldsymbol{q} \le \bar{\boldsymbol{p}} \le \boldsymbol{p}, t_1 \ge 0}[ I^1_{0,t_1}((0,\boldsymbol{p}), (0,\bar{\boldsymbol{p}})) + I^2_{t_1,t_1+\tau}((0,\bar{\boldsymbol{p}}), (0,\bar{\boldsymbol{p}} - \boldsymbol{q}))] \\
		& = \inf_{\boldsymbol{q} \le \bar{\boldsymbol{p}} \le \boldsymbol{p}}[ I^2_{0,\taubar}((0,\boldsymbol{p}), (0,\bar{\boldsymbol{p}})) + I^2_{\taubar,\taubar+\tau}((0,\bar{\boldsymbol{p}}), (0,\bar{\boldsymbol{p}} - \boldsymbol{q}))] \\
		& = \inf_{\boldsymbol{0} \le \bar{\boldsymbol{q}} \le \boldsymbol{p} - \boldsymbol{q}} [I^2_{0,\bar \tau}((0,\boldsymbol{p}), (0,\boldsymbol{p} - \bar{\boldsymbol{q}})) + I^2_{\bar \tau, \bar \tau + \tau}((0,\boldsymbol{p} - \bar{\boldsymbol{q}}), (0,\boldsymbol{p} - \bar{\boldsymbol{q}} - \boldsymbol{q}))]
	\end{align*}	
	where the first equality uses part (a) with ${\boldsymbol{q}} = \bar{\boldsymbol{p}}$ and the second equality follows from Lemma \ref{lem:I1I2L} and the observation that $I^2_{t,t+\tau}((0,\bar{\boldsymbol{p}}), (0,\bar{\boldsymbol{p}} - \boldsymbol{q})) = I^2_{t', t' + \tau}((0,\bar{\boldsymbol{p}}), (0,\bar{\boldsymbol{p}} - \boldsymbol{q}))$ for all $t, t'$ as long as 
	$I^0_{0, s}((0,{\boldsymbol{p}}), (0,\bar{\boldsymbol{p}})<\infty$ for $s= t, t'$.
	Using Proposition \ref{prop:minimizer-summary}(b), the right side on the last line equals	
	\begin{align*}	
		&  \inf_{\boldsymbol{0} \le \bar{\boldsymbol{q}} \le \boldsymbol{p} - \boldsymbol{q}} [I^2_{0, \tau}((0,\boldsymbol{p}), (0,\boldsymbol{p} -  \boldsymbol{q})) + I^2_{ \tau,  \tau +\bar \tau}((0,\boldsymbol{p} -  \boldsymbol{q}), (0,\boldsymbol{p} -  \boldsymbol{q} - \bar{\boldsymbol{q}}))]\\
		& = I^2_{0, \tau}((0,\boldsymbol{p}), (0,\boldsymbol{p} -  \boldsymbol{q})) + \inf_{\boldsymbol{0} \le \bar{\boldsymbol{q}} \le \boldsymbol{p} - \boldsymbol{q}} I^2_{ \tau,  \tau +\bar \tau}((0,\boldsymbol{p} -  \boldsymbol{q}), (0,\boldsymbol{p} -  \boldsymbol{q} - \bar{\boldsymbol{q}}))\\
		& = I^2_{0, \tau}((0,\boldsymbol{p}), (0,\boldsymbol{p} -  \boldsymbol{q})),
	\end{align*}
	where the last equality follows by considering $\bar\qbd=0$. 
	This completes the proof.
\end{proof}

Next we will prove the lower bound.

\begin{Lemma}
	\label{lem:puhalskii-lower-bound}
	Suppose $\zero \le \boldsymbol{q} \le \boldsymbol{p}$ and $\sum_{k=1}^\infty k q_k  > 2\sum_{k=1}^\infty q_k$.
	Let $\tau \doteq \frac{1}{2} \sum_{k=1}^\infty kq_k$.
	Then the lower bound in \eqref{star1227} holds.
\end{Lemma}

\begin{proof}
	\blue{Let $({\boldsymbol{\tilde\zeta}}(t),\tilde{\psi}(t))$ be as introduced in Construction \ref{cons:cont} for $t \le \tau$}, with $t_1=0$, $\boldsymbol{x}^{(1)} = (0,\boldsymbol{p})$ and $\boldsymbol{x}^{(2)} = (0,\boldsymbol{p}-\boldsymbol{q})$.
	We define \blue{$(\tilde{\boldsymbol{\boldsymbol{\zeta}}}(t),\tilde{\psi}(t))$} for $t > \tau$ through \eqref{eq:psi}-\eqref{eq:phi_k} by setting $\varphi_k(t,y)=1$ for all $k,y$ and $t>\tau$. Then $I_t({\boldsymbol{\tilde\zeta}},\tilde{\psi}) = I_\tau({\boldsymbol{\tilde\zeta}},\tilde{\psi})$ for all $t > \tau$.
	So by Lemmas \ref{lem:I1I2L}  and  \ref{lem:minimizer-verify-general},  for $t\ge \tau$
	\begin{equation}
		\label{eq:puhalskii-lower-bound-minimizer}
		I_t({\boldsymbol{\tilde\zeta}},\tilde{\psi}) = I_\tau({\boldsymbol{\tilde\zeta}},\tilde{\psi}) = \int_0^\tau L({\boldsymbol{\tilde\zeta}}(s), {\boldsymbol{\tilde\zeta}}'(s))\,ds =  I^2_{0,\tau}((0,\boldsymbol{p}), (0,\boldsymbol{p} - \boldsymbol{q})).
	\end{equation}		 
	For $\delta \in (0,1)$ consider the set
	\begin{equation}
		\label{eq:Gtil-delta}
		\Scale[0.9]{\tilde{G}_{\delta}({\boldsymbol{\tilde\zeta}},\tilde{\psi}) \doteq \{ (\boldsymbol{\zeta},\psi) \in \mathbb{D}([0,\infty):\mathbb{R} \times \mathbb{R}_+^\infty \times \mathbb{R}) : \sup_{t \in [0,\tau]}|\zeta_k(t) - \tilde{\zeta}_k(t)| < \delta, \mbox{ for all } k = 0, 1 ,2,\dotsc, \lfloor\delta^{-1}\rfloor\}.}
	\end{equation}
	Let $\tau^n \doteq \inf\{t \ge \tau : X^n_0(t) = -\frac{1}{n}\}$.
	Then $\tau^n < \infty$ a.s.
	Define for odd integer $j \ge -1$,
	\begin{equation*}
		G^n_j \doteq \left\{ X^n_0(\tau) = \frac{j}{n}, X^n_k(\tau^n) = X^n_k(\tau), k \in \mathbb{N} \right\},
	\end{equation*}
	and for even integer $j \ge -1$,
	\begin{equation*}
		G^n_j \doteq \left\{ X^n_0(\tau) = \frac{j}{n}, \sum_{k=1}^\infty  (X^n_k(\tau^n) - X^n_k(\tau)) = - \frac{1}{n} \right\}.
	\end{equation*}
	Intuitively, $G^n_j$ describes the event that from time instant $\tau$ to the time $\tau^n$ at which the current component is fully explored, the continuous-time EEA does not wake up any sleeping vertices, with the exception that if the number of active half-edges at time $\tau$ is odd (namely $X^n_0(\tau) = \frac{j}{n}$ for some even integer $j \ge -1$), in which case exactly one sleeping vertex (necessarily with odd degree) will be woken up.
	Consider the event
	\begin{equation*}
		A^n_\delta({\boldsymbol{\tilde\zeta}},\tilde{\psi}) \doteq \left\{(\boldsymbol{X}^n,Y^n) \in \tilde{G}_{\delta}({\boldsymbol{\tilde\zeta}},\tilde{\psi})\right\} \bigcap \left( \bigcup_{j=-1}^\infty G^n_j \right).
	\end{equation*}
	Fix $\varepsilon \in (0,1)$.
	We claim that there exist $\delta_0>0$ and $n_0 >0$ such that
	\begin{equation}
		\label{eq:puhalskii-lower-bound-claim}
		A^n_\delta({\boldsymbol{\tilde\zeta}},\tilde{\psi}) \subset E^{n,\varepsilon}(\boldsymbol{q}) \mbox{ for all } \delta < \delta_0 \mbox{ and } n > n_0.
	\end{equation}
	To see this, first note that by Assumption \ref{asp:exponential-boundN}, there exists $M \in \mathbb{N}$ such that
	\begin{equation}
		\label{eq:puhalskii-lower-bound-tail}
		\sup_{n \in \mathbb{N}} \sum_{k=M}^\infty k \frac{n_k}{n} < \frac{\varepsilon}{2},\quad \sum_{k=M}^\infty k p_k < \frac{\varepsilon}{2}.
	\end{equation}
	By continuity of ${\boldsymbol{\tilde\zeta}}$, there exists $\varepsilon_0 > 0$ such that 
	\begin{align}
		&|\tilde{\zeta}_k(t) - \tilde{\zeta}_k(0)| < \frac{\varepsilon}{4} \mbox{ for all } t \in [0,\varepsilon_0], k=0,1,\dotsc,M, \label{eq:puhalskii-lower-bound-0} \\
		&|\tilde{\zeta}_k(t) - \tilde{\zeta}_k(\tau)| < \frac{\varepsilon}{4} \mbox{ for all } t \in [\tau - \varepsilon_0, \tau], k=0,1,\dotsc,M. \label{eq:puhalskii-lower-bound-tau}
	\end{align}
	From Lemma \ref{lem:minimizer-def}(c) we have $\tilde{\zeta}_0(t)>0$ for all $t \in (0,\tau)$. 
	Since $\tilde{\zeta}_0(t)$ is continuous,
	\begin{equation*}
		\delta_0 \doteq \left(\inf_{t \in [\varepsilon_0, \tau - \varepsilon_0]}\tilde{\zeta}_0(t)\right) \wedge \frac{\varepsilon}{4} \wedge \frac{1}{M} > 0.
	\end{equation*}
	Take $n_0 > \frac{4}{\varepsilon}$. We now show \eqref{eq:puhalskii-lower-bound-claim} with this choice of $n_0$ and $\delta_0$.
	Fix $\delta<\delta_0$ and $n>n_0$ and
 consider $\omega \in A^n_\delta({\boldsymbol{\tilde\zeta}},\tilde{\psi})$.
	For $t \in [\varepsilon_0, \tau - \varepsilon_0]$, since $|X^n_0(t)-\tilde{\zeta}_0(t)| < \delta < \delta_0 \le \tilde{\zeta}_0(t)$, we have $\inf_{t \in [\varepsilon_0, \tau - \varepsilon_0]}X^n_0(t) > 0$. 
	So there exist $t_1^n \in [0,\varepsilon_0]$ and $t_2^n \in [\tau-\varepsilon_0, \tau^n]$ such that 
	\begin{equation}
		\label{eq:puhalskii-lower-bound-degree-0}
		X^n_0(t_1^n-)=X^n_0(t_2^n)=-\frac{1}{n}, \quad X^n_0(t) > -\frac{1}{n} \mbox{ for } t \in [t_1^n,t_2^n),
	\end{equation}
	where by convention $X^n_0(0-)=X^n_0(0) = -1/n$.
	For $k \ge M$, it follows from \eqref{eq:puhalskii-lower-bound-tail} that
	\begin{equation}
		\label{eq:puhalskii-lower-bound-degree-1}
		|X_k^n(t_1^n-) - X_k^n(t_2^n) - q_k| \le |X_k^n(t_1^n-) - X_k^n(t_2^n)| + q_k \le \frac{n_k}{n} + p_k < \varepsilon.
	\end{equation}
	For $1 \le k \le M \le \lfloor \delta^{-1} \rfloor$,
	\begin{align}
		| X_k^n(t_1^n-) - X_k^n(t_2^n) - q_k|
		& =  |(X_k^n(t_1^n-) - X_k^n(t_2^n)) - ( \tilde{\zeta}_k(0) - \tilde{\zeta}_k(\tau))| \notag \\
		& \le |X_k^n(t_1^n-)-\tilde{\zeta}_k(0)| + |X_k^n(t_2^n)-\tilde{\zeta}_k(\tau)|. \label{eq:puhalskii-lower-bound-degree-2}
	\end{align}
	From \eqref{eq:Gtil-delta} and \eqref{eq:puhalskii-lower-bound-0} we have the following bound for the first term
	in \eqref{eq:puhalskii-lower-bound-degree-2}.
	\begin{equation*}
		|X_k^n(t_1^n-)-\tilde{\zeta}_k(0)| \le |X_k^n(t_1^n-)-\tilde{\zeta}_k(t_1^n)| + |\tilde{\zeta}_k(t_1^n)-\tilde{\zeta}_k(0)| < \delta + \frac{\varepsilon}{4}.
	\end{equation*}
	For the second term in \eqref{eq:puhalskii-lower-bound-degree-2}, if $t_2^n \le \tau$, then using \eqref{eq:Gtil-delta} and \eqref{eq:puhalskii-lower-bound-tau} we have
	\begin{equation*}
		|X^n_k(t_2^n) - \tilde{\zeta}_k(\tau)| \le |X^n_k(t_2^n) - \tilde{\zeta}_k(t_2^n)| + |\tilde{\zeta}_k(t_2^n) - \tilde{\zeta}_k(\tau)| < \delta + \frac{\varepsilon}{4}.
	\end{equation*}
	If $t_2^n > \tau$, then $t_2^n=\tau^n$ and from the definition of $G_j^n$ and \eqref{eq:Gtil-delta} we have
	\begin{equation*}
		|X^n_k(t_2^n) - \tilde{\zeta}_k(\tau)| \le |X^n_k(\tau^n) - X^n_k(\tau)| + |X^n_k(\tau) - \tilde{\zeta}_k(\tau)| \le \frac{1}{n} + \delta \le \frac{\varepsilon}{4} + \delta.
	\end{equation*}
	Combining these three displays with \eqref{eq:puhalskii-lower-bound-degree-2} gives
	\begin{equation*}
		| X_k^n(t_1^n-) - X_k^n(t_2^n) - q_k| < 2 \left( \delta + \frac{\varepsilon}{4} \right) < \varepsilon, \quad k \in \mathbb{N}.
	\end{equation*}
	From this, and \eqref{eq:enverps}, \eqref{eq:puhalskii-lower-bound-degree-0}, \eqref{eq:puhalskii-lower-bound-degree-1} we have $\omega \in E^{n,\varepsilon}(\boldsymbol{q})$. Since  $\delta < \delta_0$ and $n > n_0$ is arbitrary, the
	claim \eqref{eq:puhalskii-lower-bound-claim} holds.
	
	For fixed $\delta < \delta_0$ and $n > n_0$ consider the following two probabilities
	\begin{equation*}
		P(A^n_\delta({\boldsymbol{\tilde\zeta}},\tilde{\psi})), \quad P((\boldsymbol{X}^n,Y^n) \in \tilde{G}_{\delta}({\boldsymbol{\tilde\zeta}},\tilde{\psi})).
	\end{equation*}
	Write
	\begin{align}
		P(A^n_\delta({\boldsymbol{\tilde\zeta}},\tilde{\psi})) & = \sum_{j=-1}^\infty P\left(\{(\boldsymbol{X}^n,Y^n) \in \tilde{G}_{\delta}({\boldsymbol{\tilde\zeta}},\tilde{\psi})\} \cap G^n_j\right) 
		 = \sum_{j=-1}^{\lfloor \delta n \rfloor} \boldsymbol{E}\left[ 1_{\{(\boldsymbol{X}^n,Y^n) \in \tilde{G}_{\delta}({\boldsymbol{\tilde\zeta}},\tilde{\psi})\}} P\left( G^n_j | \mathcal{F}_\tau \right) \right] \label{eq:puhalskii-lower-bound-key}
	\end{align}
	where we only have to sum up to $\lfloor \delta n \rfloor$ in the last line when $(\boldsymbol{X}^n,Y^n) \in \tilde{G}_{\delta}({\boldsymbol{\tilde\zeta}},\tilde{\psi})$ since $\tilde{\zeta}_0(\tau)=0$.
	Since $\tau^n = \tau$ on $\{X_0^n(\tau)=-\frac{1}{n}\}$, we have
	\begin{equation}
		\label{eq:puhalskii-lower-bound-G1}
		P\left( G^n_{-1} | \mathcal{F}_\tau \right) = 1_{\{X^n_0(\tau) = -\frac{1}{n}\}}.
	\end{equation}
	From Assumption \ref{asp:exponential-boundN} and the fact that $r(\boldsymbol{X}^n(t))$ is non-increasing it follows
	\begin{equation*}
		\sup_{n \in \mathbb{N}} \sup_{t \ge 0} r(\boldsymbol{X}^n(t)) \le \sup_{n \in \mathbb{N}} \sum_{k=1}^\infty \frac{kn_k}{n} \doteq C_0 < \infty.
	\end{equation*}
	Hence for odd integer $1 \le j \le \lfloor \delta n \rfloor$ and $\delta < \frac{C_0}{2}$,
	\begin{align}
		& 1_{\{(\boldsymbol{X}^n,Y^n) \in \tilde{G}_{\delta}({\boldsymbol{\tilde\zeta}},\tilde{\psi})\}} P\left( G^n_j | \mathcal{F}_\tau \right) \notag \\
		& = 1_{\{(\boldsymbol{X}^n,Y^n) \in \tilde{G}_{\delta}({\boldsymbol{\tilde\zeta}},\tilde{\psi})\}} \frac{j}{nr(\boldsymbol{X}^n(\tau))} \cdot \frac{j-2}{nr(\boldsymbol{X}^n(\tau))-2} \dotsm \frac{1}{nr(\boldsymbol{X}^n(\tau))-(j-1)} 1_{\{X^n_0(\tau) = \frac{j}{n}\}} \notag \\
		& \ge 1_{\{(\boldsymbol{X}^n,Y^n) \in \tilde{G}_{\delta}({\boldsymbol{\tilde\zeta}},\tilde{\psi})\}} \frac{j}{C_0n} \cdot \frac{j-2}{C_0n} \dotsm \frac{1}{C_0n} 1_{\{X^n_0(\tau) = \frac{j}{n}\}} \notag \\
		& \ge 1_{\{(\boldsymbol{X}^n,Y^n) \in \tilde{G}_{\delta}({\boldsymbol{\tilde\zeta}},\tilde{\psi})\}} \frac{\lfloor 2\delta n \rfloor!}{(C_0n)^{\lfloor 2\delta n \rfloor}} 1_{\{X^n_0(\tau) = \frac{j}{n}\}}, \label{eq:puhalskii-lower-bound-G2}
	\end{align}
	where the last inequality follows since the term on the last line includes  more fractions that are less than $1$ than the one on the previous line. For the first equality we have used the fact that on the event  $G^n_j$ all the active $j+1$ half edges (an even number) at time instant $\tau$ should merge among themselves (without waking any sleeping vertices) by the time instant $\tau^n$, whereas the total number of available half edges (either awake or sleeping) at time instant $\tau$ equals $nr(\boldsymbol{X}^n(\tau))+1$.
	
	For even integer $0\le j \le \lfloor \delta n \rfloor$, we consider three different cases for values of $\boldsymbol{p}$ and $\boldsymbol{q}$.
		
	{\bf Case 1:} There exists some odd $m \in \mathbb{N}$ such that $p_m > q_m \ge 0$.	
	Let $C_m \doteq \frac{1}{2}(p_m - q_m) > 0$. 
	For $\delta < \frac{1}{m} \wedge \delta_0 \wedge C_m$ and $(\boldsymbol{X}^n,Y^n) \in \tilde{G}_{\delta}({\boldsymbol{\tilde\zeta}},\tilde{\psi})$,  we  have from \eqref{eq:Gtil-delta},
	\begin{equation*}
		X^n_{m}(\tau) = \tilde{\zeta}^n_{m}(\tau) - (\tilde{\zeta}^n_{m}(\tau) - X^n_{m}(\tau)) > (p_{m} - q_{m}) - \delta > C_m,
	\end{equation*}
	which implies $X^n_{m}(\tau) \ge 1/n$ for $n \ge \delta^{-1}$.
	So for even integer $0 \le j \le \lfloor \delta n \rfloor$ and $n > \frac{m}{\delta} \vee n_0$,
	
	\begin{align}
		& 1_{\{(\boldsymbol{X}^n,Y^n) \in \tilde{G}_{\delta}({\boldsymbol{\tilde\zeta}},\tilde{\psi})\}} P\left( G^n_j | \mathcal{F}_\tau \right) \notag \\
		& \ge 1_{\{(\boldsymbol{X}^n,Y^n) \in \tilde{G}_{\delta}({\boldsymbol{\tilde\zeta}},\tilde{\psi})\}} \frac{mnX^n_{m}(\tau)}{nr(\boldsymbol{X}^n(\tau))} \cdot \frac{j+m-2}{nr(\boldsymbol{X}^n(\tau))-2} \cdot \frac{j+m-4}{nr(\boldsymbol{X}^n(\tau))-4} \notag \\
		& \quad \dotsm \frac{1}{nr(\boldsymbol{X}^n(\tau))-(j+m-1)} 1_{\{X^n_0(\tau) = \frac{j}{n}\}} \notag \\
		& \ge 1_{\{(\boldsymbol{X}^n,Y^n) \in \tilde{G}_{\delta}({\boldsymbol{\tilde\zeta}},\tilde{\psi})\}} \frac{m}{C_0n} \cdot \frac{j+m-2}{C_0n} \cdot \frac{j+m-4}{C_0n} \dotsm \frac{1}{C_0n} 1_{\{X^n_0(\tau) = \frac{j}{n}\}} \notag \\
		& \ge 1_{\{(\boldsymbol{X}^n,Y^n) \in \tilde{G}_{\delta}({\boldsymbol{\tilde\zeta}},\tilde{\psi})\}} \frac{\lfloor 2\delta n \rfloor!}{(C_0n)^{\lfloor 2\delta n \rfloor}} 1_{\{X^n_0(\tau) = \frac{j}{n}\}}, \label{eq:puhalskii-lower-bound-G3}
	\end{align}
	where the last inequality follows once again as in \eqref{eq:puhalskii-lower-bound-G2}.
	Combining \eqref{eq:puhalskii-lower-bound-key}--\eqref{eq:puhalskii-lower-bound-G3} implies that, for $\delta < \frac{C_0}{2} \wedge \frac{1}{m} \wedge \delta_0 \wedge C_m$ and $n > \frac{m}{\delta} \vee n_0$,
	\begin{align*}
		P(A^n_\delta({\boldsymbol{\tilde\zeta}},\tilde{\psi})) & \ge \sum_{j=-1}^{\lfloor \delta n \rfloor} \boldsymbol{E}\left[ 1_{\{(\boldsymbol{X}^n,Y^n) \in \tilde{G}_{\delta}({\boldsymbol{\tilde\zeta}},\tilde{\psi})\}} \frac{\lfloor 2\delta n \rfloor!}{(C_0n)^{\lfloor 2\delta n \rfloor}} 1_{\{X^n_0(\tau) = \frac{j}{n}\}} \right] \\
		& = \frac{\lfloor 2\delta n \rfloor!}{(C_0n)^{\lfloor 2\delta n \rfloor}} P((\boldsymbol{X}^n,Y^n) \in \tilde{G}_{\delta}({\boldsymbol{\tilde\zeta}},\tilde{\psi})) \\
		& \ge \sqrt{2\pi\lfloor 2\delta n \rfloor} \left(\frac{\lfloor 2\delta n \rfloor}{C_0en}\right)^{\lfloor 2\delta n \rfloor} P((\boldsymbol{X}^n,Y^n) \in \tilde{G}_{\delta}({\boldsymbol{\tilde\zeta}},\tilde{\psi})),
	\end{align*}
	where the last line uses Stirling's approximation $n! \ge \sqrt{2\pi n} (\frac{n}{e})^n$.
	From this and \eqref{eq:puhalskii-lower-bound-claim} we have
	\begin{align}
		& \liminf_{n \to \infty} \frac{1}{n}\log P(E^{n,\varepsilon}(\boldsymbol{q})) \notag \\
		& \ge \liminf_{n \to \infty}  \frac{1}{n}\log P( A^n_\delta({\boldsymbol{\tilde\zeta}},\tilde{\psi}) ) \notag \\
		& \ge \liminf_{n \to \infty} \left[ \frac{1}{2n} \log \left( 2\pi\lfloor 2\delta n \rfloor \right) + \frac{\lfloor 2\delta n \rfloor}{n} \log \left(\frac{\lfloor 2\delta n \rfloor}{C_0en}\right) + \frac{1}{n} \log P((\boldsymbol{X}^n,Y^n) \in \tilde{G}_{\delta}({\boldsymbol{\tilde\zeta}},\tilde{\psi})) \right] \notag \\
		& = 2\delta \log\left(\frac{ 2\delta }{C_0e}\right) + \liminf_{n \to \infty} \frac{1}{n} \log P((\boldsymbol{X}^n,Y^n) \in \tilde{G}_{\delta}({\boldsymbol{\tilde\zeta}},\tilde{\psi})). \label{eq:puhalskii-lower-bound-last}
	\end{align}
	Define the open set
	\begin{equation*}
		\Scale[0.9]{G_{\delta,\tau}({\boldsymbol{\tilde\zeta}},\tilde{\psi}) \doteq \{ (\boldsymbol{\zeta},\psi) \in \mathbb{D}([0,\tau]:\mathbb{R} \times \mathbb{R}_+^\infty \times \mathbb{R}): \sup_{t \in [0,\tau]}|\zeta_k(t) - \tilde{\zeta}_k(t)| < \delta, \mbox{ for all } k = 0, 1 ,2,\dotsc, \lfloor\delta^{-1}\rfloor\}.}
	\end{equation*}
	It follows from Theorem \ref{thm:main-ldp} that
	\begin{align*}
		\liminf_{n \to \infty} \frac{1}{n} \log P((\boldsymbol{X}^n,Y^n) \in \tilde{G}_{\delta}({\boldsymbol{\tilde\zeta}},\tilde{\psi})) & = \liminf_{n \to \infty} \frac{1}{n} \log P((\boldsymbol{X}^n,Y^n) \in G_{\delta,\tau}({\boldsymbol{\tilde\zeta}},\tilde{\psi})) \\
		& \ge - \inf_{(\boldsymbol{\zeta},\psi) \in G_{\delta,\tau}({\boldsymbol{\tilde\zeta}},\tilde{\psi})} I_\tau((\boldsymbol{\zeta},\psi))  \ge - I_\tau({\boldsymbol{\tilde\zeta}},\tilde{\psi}).
	\end{align*}
	Combining this with \eqref{eq:puhalskii-lower-bound-last},
	\eqref{eq:puhalskii-lower-bound-minimizer} and sending $\delta \to 0$ gives
	\begin{equation}
		\label{eq:puhalskii-lower-bound-pre-key}
		\liminf_{n \to \infty} \frac{1}{n}\log P(E^{n,\varepsilon}(\boldsymbol{q})) \ge - I_\tau({\boldsymbol{\tilde\zeta}},\tilde{\psi}) = -I^2_{0,\tau}((0,\boldsymbol{p}), (0,\boldsymbol{p} - \boldsymbol{q})). 
	\end{equation}
	The lower bound in Case 1 now follows on sending $\varepsilon \to 0$.
	
	{\bf Case 2:} $p_m = 0$ for all odd $m \in \mathbb{N}$.
	It suffices to establish a similar estimate as in \eqref{eq:puhalskii-lower-bound-G3}, the lower bound in Case 2 will then follow as in   Case 1.
	From Assumptions \ref{asp:convgN} and \ref{asp:exponential-boundN}  $$\sum_{k=0}^\infty (2k+1) \frac{n_{2k+1}}{n} \to \sum_{k = 0}^\infty (2k+1) p_{2k+1} = 0.$$
	Therefore for each $\kappa \in (0,1)$, there exists some $\ntil_\kappa \in \mathbb{N}$ such that $0 \le \sum_{k=0}^\infty (2k+1) \frac{n_{2k+1}}{n} < \kappa$ for $n > \ntil_\kappa$, which implies $n_m=0$ for all odd $m \ge \kappa n$.
Consider now an even integer $0 \le j \le \lfloor \delta n \rfloor$ and $n > \ntil_\kappa$.
Denote by $M^n$ the largest odd degree for which there is a sleeping vertex at time instant $\tau$ in the continuous time EEA. Note that $M^n \le \kappa n$ a.s.
	Therefore for $\kappa < \delta<\frac{C_0}{2} \wedge \delta_0$ and $n > n_0 \vee \ntil_\kappa$,
	\begin{align*}
		& 1_{\{(\boldsymbol{X}^n,Y^n) \in \tilde{G}_{\delta}({\boldsymbol{\tilde\zeta}},\tilde{\psi})\}} P\left( G^n_j | \mathcal{F}_\tau \right) \notag \\
		& \ge 1_{\{(\boldsymbol{X}^n,Y^n) \in \tilde{G}_{\delta}({\boldsymbol{\tilde\zeta}},\tilde{\psi})\}} \sum_{1 \le m \le \lfloor \kappa n \rfloor, m \mbox{ is odd}} \left[ 1_{\{M^n=m\}}   \frac{mnX^n_{m}(\tau)}{nr(\boldsymbol{X}^n(\tau))} \cdot \frac{j+m-2}{nr(\boldsymbol{X}^n(\tau))-2} \right. \notag \\
		& \quad \left. \cdot \frac{j+m-4}{nr(\boldsymbol{X}^n(\tau))-4} \dotsm \frac{1}{nr(\boldsymbol{X}^n(\tau))-(j+m-1)} \right] 1_{\{X^n_0(\tau) = \frac{j}{n}\}} 
		\end{align*}
Since $nX_m^n(\tau) \ge 1$ on the set $\{M^n=m\}$	
the right side can be bounded below by
\begin{align*}
 &1_{\{(\boldsymbol{X}^n,Y^n) \in \tilde{G}_{\delta}({\boldsymbol{\tilde\zeta}},\tilde{\psi})\}} \sum_{1 \le m \le \lfloor \kappa n \rfloor, m \mbox{ is odd}} \left[ 1_{\{M^n=m\}} \frac{(j+m-2)!!m}{(C_0n)^{(j+m+1)/2}} \right] 1_{\{X^n_0(\tau) = \frac{j}{n}\}} \notag \\
		& \ge 1_{\{(\boldsymbol{X}^n,Y^n) \in \tilde{G}_{\delta}({\boldsymbol{\tilde\zeta}},\tilde{\psi})\}} \sum_{1 \le m \le \lfloor \kappa n \rfloor, m \mbox{ is odd}} \left[ 1_{\{M^n=m\}} \frac{\lfloor 2\delta n \rfloor!}{(C_0n)^{\lfloor 2\delta n \rfloor}} \right] 1_{\{X^n_0(\tau) = \frac{j}{n}\}} \notag \\
		& = 1_{\{(\boldsymbol{X}^n,Y^n) \in \tilde{G}_{\delta}({\boldsymbol{\tilde\zeta}},\tilde{\psi})\}} \frac{\lfloor 2\delta n \rfloor!}{(C_0n)^{\lfloor 2\delta n \rfloor}} 1_{\{X^n_0(\tau) = \frac{j}{n}\}}, \notag
	\end{align*}
	 where the last inequality follows once again
	as  in \eqref{eq:puhalskii-lower-bound-G2}.
	Therefore we have  the same inequality as in \eqref{eq:puhalskii-lower-bound-G3} for $n > n_0 \vee \ntil_\kappa$ and $\delta<\frac{C_0}{2} \wedge \delta_0$, and so the lower bound in Case 2 follows.
	
{\bf Case 3:} There exists an odd $m \in \mathbb{N}$ such that $p_m > 0$ but $p_m=q_m$.
	For $i \in \mathbb{N}$, consider the vector $\boldsymbol{q}^i \doteq (q_k^i)_{k \in \mathbb{N}}$, where $q_k^i \doteq q_k$ for $k \ne m$ and $q_m^i \doteq q_m - \frac{1}{i}$.
	Fix $\varepsilon \in (0,1)$.
	Choose $i$  so that $p_m > q_m^i > 0$, $\varepsilon > \frac{1}{i} = q_m - q_m^i$ and $\sum_{k=1}^\infty kq_k^i > 2\sum_{k=1}^\infty q_k^i$.
	When $\boldsymbol{q}$ is replaced by $\boldsymbol{q}^i$ we are in Case 1 and
	thus for $\varepsilon^i < (\varepsilon - \frac{1}{i})$, from the lower bound \eqref{eq:puhalskii-lower-bound-pre-key} for Case 1, 
	\begin{equation}
		\label{eq:puhalskii-lower-bound-case3}
		\liminf_{n \to \infty} \frac{1}{n}\log P(E^{n,\varepsilon}(\boldsymbol{q})) \ge \liminf_{n \to \infty} \frac{1}{n}\log P(E^{n,\varepsilon^i}(\boldsymbol{q}^i)) \ge  -I^2_{0,\tau}((0,\boldsymbol{p}), (0,\boldsymbol{p} - \boldsymbol{q}^i)).
	\end{equation}
	From Proposition \ref{prop:minimizer-summary}(a) we have
	\begin{equation}\label{eq:i20t}
		I^2_{0,\tau}((0,\boldsymbol{p}), (0,\boldsymbol{p} - \boldsymbol{q}^i)) = H({\boldsymbol{q}^i})+H({	\boldsymbol{p}}-{\boldsymbol{q}^i})-H({\boldsymbol{p}})+K({\boldsymbol{q}^i}).
	\end{equation}
	Since $((0,\boldsymbol{p}), (0,\boldsymbol{p}-\boldsymbol{q}^i))$ and $((0,\boldsymbol{p}), (0,\boldsymbol{p}-\boldsymbol{q}))$
	are in $\Xi$ and $\boldsymbol{q}^i\to \boldsymbol{q}$, from Lemma \ref{lem:lemctybeta} we have $K(\boldsymbol{q}^i)\to K(\boldsymbol{q})$. Also, clearly
	$$
	H({\boldsymbol{q}^i})+H({	\boldsymbol{p}}-{\boldsymbol{q}^i})-H({\boldsymbol{p}})
	\to H({\boldsymbol{q}})+H({	\boldsymbol{p}}-{\boldsymbol{q}})-H({\boldsymbol{p}}).$$
	Thus, as $i\to \infty$, the right side of \eqref{eq:i20t}  converges to
	$$H({\boldsymbol{q}})+H({	\boldsymbol{p}}-{\boldsymbol{q}})-H({\boldsymbol{p}})+K({\boldsymbol{q}})
	= I^2_{0,\tau}((0,\boldsymbol{p}), (0,\boldsymbol{p} - \boldsymbol{q})),$$
	where the equality follows again by Proposition \ref{prop:minimizer-summary}(a).
	The desired result now follows on sending $i \to \infty$ and then $\varepsilon \to 0$
	in \eqref{eq:puhalskii-lower-bound-case3}.

	The above three cases cover all possible values of $\boldsymbol{p}$ and $\boldsymbol{q}$.
	This completes the proof.	
\end{proof}

\subsection{Completing the proof of Theorem \ref{thm:ldg_degree_distribution}}

The upper bound of Theorem \ref{thm:ldg_degree_distribution} follows from Lemma \ref{lem:puhalskii-upper-bound}, Lemma \ref{lem:puhalskii-upper-bound-improvement}(b) and Proposition \ref{prop:minimizer-summary}(a).
The lower bound of Theorem \ref{thm:ldg_degree_distribution} follows from Lemma \ref{lem:puhalskii-lower-bound} and Proposition \ref{prop:minimizer-summary}(a).
\qed

\section{Proofs of Auxiliary Lemmas}
\label{sec:pfsect4}
In this section we prove the lemmas in Section \ref{sec:cal}. 
Specifically, in Section \ref{subsec:6.1} we prove Lemma \ref{lem:I1I2L}, 
in Section \ref{subsec:6.2} we prove Lemmas \ref{lem:uniqbeta} and \ref{lem:lemctybeta},
in Section \ref{subsec:6.3} we prove Lemma \ref{lem:minimizer-def},
in Section \ref{subsec:6.4} we prove Lemma \ref{lem:minimizer-cost},
and finally in Section \ref{subsec:6.5} we prove Lemma \ref{lem:minimizer-verify-general}.

We start with the following remark.

\begin{Remark}
	\label{rmk:prep_evolution}
	Fix $0 \le t_1 < t_2 < \infty$, $\boldsymbol{x}^{(1)},\boldsymbol{x}^{(2)} \in \Rmb_+^\infty$ and $(\boldsymbol{\zeta},\psi) \in \mathcal{J}^1_{t_1,t_2}(\boldsymbol{x}^{(1)}, \boldsymbol{x}^{(2)})$ such that $I_{t_1,t_2}(\boldsymbol{\zeta},\psi) < \infty$.
	Fix $\varepsilon \in (0,1)$.
	Then there exists $\boldsymbol{\varphi} \in \mathcal{S}_{t_2}(\boldsymbol{\zeta},\psi)$ such that
	\begin{equation*}
		\sum_{k=0}^\infty \int_{[t_1,t_2] \times [0,1]} \ell(\varphi_k(s,y)) \,ds\,dy \le I_{t_1,t_2}(\boldsymbol{\zeta},\psi) + \varepsilon.
	\end{equation*}
	Using convexity of $\ell$, we can assume without loss of generality that $\varphi_k(t,y) = \rho_k(t) 1_{[0,r_k(\boldsymbol{\zeta}(t)))}(y) + 1_{[r_k(\boldsymbol{\zeta}(t)),1]}(y)$ for $t \in [t_1,t_2]$, where
	$\rho_k$ is some nonnegative function.
	From \eqref{eq:psi} and \eqref{eq:phi_k} we see that for a.e.\ $t \in [t_1,t_2]$, 
	\begin{align}
		\zeta_k'(t) & = -\rho_k(t)r_k(\boldsymbol{\zeta}(t)), k \in \mathbb{N}, \label{eq:I1I2L_phi_k} \\
		\psi'(t) & = \sum_{k=0}^\infty (k-2)\rho_k(t)r_k(\boldsymbol{\zeta}(t)). \label{eq:I1I2L_psi}
	\end{align}
	Since $(\boldsymbol{\zeta},\psi) \in \mathcal{J}_{t_1,t_2}^1(\boldsymbol{x}^{(1)}, \boldsymbol{x}^{(2)})$, 
	$\zeta_0(t) = \zeta_0(t_1) + \psi(t) - \psi(t_1)$ over $(t_1,t_2)$, namely 
	there is no reflection over this interval.
	Therefore for a.e.\ $t \in [t_1,t_2]$, we have $\zeta_0'(t) = \psi'(t)$ and
	\begin{equation}
		\label{eq:rmk_prep}
		- \frac{1}{2} \frac{d}{dt} r(\boldsymbol{\zeta}(t)) = - \frac{1}{2} \left( \psi'(t) + \sum_{k=1}^\infty k \zeta'_k(t) \right) = \sum_{k=0}^\infty \rho_k(t) r_k(\boldsymbol{\zeta}(t)), t \in [t_1,t_2].
	\end{equation}
\end{Remark}

Now we  prove the lemmas in Section \ref{sec:cal}.

\subsection{Proof of Lemma \ref{lem:I1I2L}}
\label{subsec:6.1}
 We first prove \eqref{eq:I1I2}.
	Since $$\inf_{t_2 \ge t_1} I^1_{t_1,t_2}(\boldsymbol{x}^{(1)}, \boldsymbol{x}^{(2)}) \le I^1_{t_1,t_1+\ttau}(\boldsymbol{x}^{(1)}, \boldsymbol{x}^{(2)}) \le I^2_{t_1,t_1+\ttau}(\boldsymbol{x}^{(1)}, \boldsymbol{x}^{(2)}),$$
	it suffices to show 
	\begin{equation}
		\inf_{t_2 \ge t_1} I^1_{t_1,t_2}(\boldsymbol{x}^{(1)}, \boldsymbol{x}^{(2)}) \ge I^2_{t_1,t_1+\ttau}(\boldsymbol{x}^{(1)}, \boldsymbol{x}^{(2)}) \label{eq:25.2}
	\end{equation}
	when $\inf_{t_2 \ge t_1} I^1_{t_1,t_2}(\boldsymbol{x}^{(1)}, \boldsymbol{x}^{(2)}) < \infty$.
	Fix $\varepsilon \in (0,1)$.
	There exist $t_2^\varepsilon \ge t_1$, $(\boldsymbol{\zeta},\psi) \in \mathcal{J}^1_{t_1,t_2^\varepsilon}(\boldsymbol{x}^{(1)}, \boldsymbol{x}^{(2)})$ and $\boldsymbol{\varphi} \in \mathcal{S}_{t_2^\varepsilon}(\boldsymbol{\zeta},\psi)$ such that
	\begin{equation}
		\label{eq:I1I2L_first}
		\sum_{k=0}^\infty \int_{[t_1,t_2^\varepsilon] \times [0,1]} \ell(\varphi_k(s,y)) \,ds\,dy \le I_{t_1,t_2^\varepsilon}(\boldsymbol{\zeta},\psi) + \varepsilon \le \inf_{t_2 \ge t_1} I^1_{t_1,t_2}(\boldsymbol{x}^{(1)}, \boldsymbol{x}^{(2)}) + 2\varepsilon.
	\end{equation}
	
Recall that $t \mapsto r(\boldsymbol{\zeta}(t))$ is a non-increasing function (see \eqref{eq:rmk_prep}).	We claim that in fact we can assume without loss of generality that $t \mapsto r(\boldsymbol{\zeta}(t))$ is strictly decreasing for $t \in [t_1,t_2^\varepsilon]$.	
Indeed, if this function is not strictly decreasing, we can modify $(\boldsymbol{\zeta},\psi)$ such that for the modified trajectory strict monotonicity holds and the associated cost is not any higher. Such a modification can be constructed via a limiting argument as follows.
Consider $(\boldsymbol{\zeta}^n,\psi^n)$ defined recursively as:
	$(\boldsymbol{\zeta}^0,\psi^0,\boldsymbol{\varphi}^0) \doteq (\boldsymbol{\zeta},\psi,\boldsymbol{\varphi})$ on $[0,t_2^0]$ where $t_2^0 \doteq t_2^\varepsilon$.
	For $n \in \mathbb{N}_0$, having defined $(\boldsymbol{\zeta}^n,\psi^n)$  and $t_2^n \le t_2^\varepsilon$, where
	$(\boldsymbol{\zeta}^n,\psi^n) \in \mathcal{J}^1_{t_1,t_2^n}(\boldsymbol{x}^{(1)}, \boldsymbol{x}^{(2)})$, and
	$\boldsymbol{\varphi}^n \in \mathcal{S}_{t_2^n}(\boldsymbol{\zeta}^n,\psi^n)$ such that 
	\begin{equation}\sum_{k=0}^\infty \int_{[t_1,t_2^n] \times [0,1]} \ell(\varphi_k^n(s,y)) \,ds\,dy
	\le \sum_{k=0}^\infty \int_{[t_1,t_2^\varepsilon] \times [0,1]} \ell(\varphi_k(s,y)) \,ds\,dy,
	\label{eq:cosdec}
\end{equation}
	we modify $(\boldsymbol{\zeta}^n,\psi^n)$, in case $r(\boldsymbol{\zeta}(\cdot))$ is not strictly decreasing on
	$[t_1,t_2^n]$ as follows. Let $[s_1^n,s_2^n] \subset [t_1,t_2^n]$  be the largest constant piece of $r(\boldsymbol{\zeta}^n(\cdot))$, namely $r(\boldsymbol{\zeta}^n(t))$ is constant on $[s_1^n,s_2^n]$ and $s_2^n-s_1^n$ is maximized among all such possible pieces.
	Let $t_2^{n+1} \doteq t_2^n-(s_2^n-s_1^n)$ and define $(\boldsymbol{\zeta}^{n+1},\psi^{n+1})$ by shrinking $(\boldsymbol{\zeta}^n,\psi^n)$ over $[s_1^n,s_2^n]$, namely let $(\boldsymbol{\zeta}^{n+1}(t),\psi^{n+1}(t)) \doteq (\boldsymbol{\zeta}^n(t),\psi^n(t))$ for $t \le s_1^n$ and $(\boldsymbol{\zeta}^{n+1}(t),\psi^{n+1}(t)) \doteq (\boldsymbol{\zeta}^n(t+s_2^n-s_1^n),\psi^n(t+s_2^n-s_1^n))$ for $s_1^n < t \le t_2^{n+1}$.
	Clearly $(\boldsymbol{\zeta}^{n+1},\psi^{n+1}) \in \mathcal{J}^1_{t_1,t_2^{n+1}}(\boldsymbol{x}^{(1)}, \boldsymbol{x}^{(2)})$ and the associated control $\boldsymbol{\varphi}^{n+1}$   satisfies \eqref{eq:cosdec}
	with $n$ replaced with $n+1$. 
		%
		%
	If $r(\boldsymbol{\zeta}(t))$ only has $N$ constant pieces over $[t_1,t_2^\varepsilon]$, then $(\boldsymbol{\zeta}^N,\psi^N)$ is the desired modification of $(\boldsymbol{\zeta},\psi)$.
	If $r(\boldsymbol{\zeta}(t))$ has countably many constant pieces over $[t_1,t_2^\varepsilon]$, then the sequence $(\boldsymbol{\zeta}^n,\psi^n)$ is well defined and \eqref{eq:cosdec} holds for every $n$. Since the sequence $t_2^n$
	is non-increasing, it converges to some point $\bar{t}_2$.
	Since $I_{\bar{t}_2}$ has compact sub-level sets, this sequence (of paths over the time interval $[t_1,\bar t_2]$) has a limit point $(\bar{\boldsymbol{\zeta}},\bar{\psi})$. It is easy to check that this limit point must belong to
	$\mathcal{J}^1_{t_1,\bar{t}_2}(\boldsymbol{x}^{(1)}, \boldsymbol{x}^{(2)})$ and 
	$I_{t_1,\bar{t}_2}(\bar{\boldsymbol{\zeta}},\bar{\psi}) \le \liminf_{n \to \infty} I_{t_1,\bar{t}_2}(\boldsymbol{\zeta}^n,\psi^n) \le I_{t_1,t_2^\varepsilon}(\boldsymbol{\zeta},\psi)$.
	From the construction one can show that for fixed $\delta \in (0,1)$, $\inf_{s \in [t_1,t_2^n-\delta]} |r(\boldsymbol{\zeta}^n(s)) - r(\boldsymbol{\zeta}^n(s+\delta))|$ is nondecreasing in $n$ and eventually positive.
	Therefore $|r(\bar{\boldsymbol{\zeta}}(s)) - r(\bar{\boldsymbol{\zeta}}(s+\delta))|>0$ for each $s \in [t_1,\tbar_2-\delta]$. 
	As $\delta$ is arbitrary, $(\bar{\boldsymbol{\zeta}},\bar{\psi})$ is the desired modification of $(\boldsymbol{\zeta},\psi)$ verifying
	the claim.

	From Remark \ref{rmk:prep_evolution} we can further assume without loss of generality that $\varphi_k(t,y) = \rho_k(t) 1_{[0,r_k(\boldsymbol{\zeta}(t)))}(y) + 1_{[r_k(\boldsymbol{\zeta}(t)),1]}(y)$ for $t \in [t_1,t_2^\varepsilon]$ and \eqref{eq:I1I2L_phi_k}--\eqref{eq:rmk_prep} hold for $t \in [t_1,t_2^\varepsilon]$.
	We now introduce a time transformation. Consider the non-decreasing function $f$ defined as: $f(0)=0$,
	\begin{equation*}
		f'(t) \doteq \left\{ 
		\begin{array}{ll}
			1, & t \in [0,t_1), \\
			- \frac{1}{2} \frac{d}{dt} r(\boldsymbol{\zeta}(t)) = \sum_{k=0}^\infty \rho_k(t) r_k(\boldsymbol{\zeta}(t)), & t \in [t_1,t_2^\varepsilon],
		\end{array} \right.
	\end{equation*}
	where the equality in the second line follows from \eqref{eq:rmk_prep}.
Since $r(\boldsymbol{\zeta}(t))$ is strictly decreasing for $t \in [t_1,t_2^\varepsilon]$, $f(t)$ must be strictly increasing for $t \in [t_1,t_2^\varepsilon]$.
	So $g \doteq f^{-1}$ is well-defined and absolutely continuous on $[0,f(t_2^\varepsilon)]$.
	Note that $f(t_2^\varepsilon) = f(t_1) + \int_{t_1}^{t_2^\varepsilon} f'(t)\,dt = t_1 - \frac{1}{2} (r(\boldsymbol{\zeta}(t_2^\varepsilon)) - r(\boldsymbol{\zeta}(t_1))) = t_1 + \ttau$, where the last equality is from
	\eqref{eq:tau}.
	Define $({\boldsymbol{\tilde\zeta}}(t),\tilde{\psi}(t)) \doteq (\boldsymbol{\zeta}(g(t)),\psi(g(t))$ for $t \in [0,t_1 + \ttau]$.
	Then it is easy to see that $({\boldsymbol{\tilde\zeta}},\tilde{\psi}) \in \mathcal{J}_{t_1,t_1+\ttau}^1(\boldsymbol{x}^{(1)}, \boldsymbol{x}^{(2)})$.
	Since $f(g(t))=t$,  $f'(g(t)) g'(t)=1$ and so
	\begin{equation*}
		\frac{d}{dt} r({\boldsymbol{\tilde\zeta}}(t)) = -2f'(g(t))g'(t) = -2 \mbox{ for a.e. } t \in [t_1,t_1+\ttau].
	\end{equation*}	
	Therefore $({\boldsymbol{\tilde\zeta}},\tilde{\psi}) \in \mathcal{J}_{t_1,t_1+\ttau}^2(\boldsymbol{x}^{(1)}, \boldsymbol{x}^{(2)})$.
	Define 
	\begin{equation*}
		\tilde{\varphi}_k(t,y) \doteq \left\{ 
		\begin{array}{ll}
			\varphi_k(t,y), & t \in [0,t_1), \\
			\tilde{\rho}_k(t) 1_{[0,r_k({\boldsymbol{\tilde\zeta}}(t)))}(y) + 1_{[r_k({\boldsymbol{\tilde\zeta}}(t)),1]}(y), & t \in [t_1, t_1+\ttau],
		\end{array} \right.
	\end{equation*}
	where $\tilde{\rho}_k(t) \doteq \rho_k(g(t)) g'(t)$ for $t \in [t_1, t_1+\ttau]$.
	From \eqref{eq:I1I2L_phi_k} and \eqref{eq:I1I2L_psi}, for $t \in [t_1, t_1+\ttau]$,
	\begin{align*}
		\tilde{\zeta}_k'(t) & = \zeta_k'(g(t)) g'(t) = -\rho_k(g(t)) r_k(\boldsymbol{\zeta}(g(t))) g'(t) = -\tilde{\rho}_k(t) r_k({\boldsymbol{\tilde\zeta}}(t)), k \in \mathbb{N}, \\
		\tilde{\psi}'(t) & = \psi'(g(t)) g'(t) = \sum_{k=0}^\infty (k-2)\rho_k(g(t))r_k(\boldsymbol{\zeta}(g(t))) g'(t) = \sum_{k=0}^\infty (k-2)\tilde{\rho}_k(t)r_k({\boldsymbol{\tilde\zeta}}(t)). 
	\end{align*}
	So $\tilde{\boldsymbol{\varphi}} \in \mathcal{S}_{t_1+\ttau}({\boldsymbol{\tilde\zeta}},\tilde{\psi})$.
	We claim that
	\begin{equation}
		\label{eq:I1I2L_cost}
		\sum_{k=0}^\infty \int_{[t_1,t_2^\varepsilon] \times [0,1]} \ell(\varphi_k(t,y)) \, dt\,dy \ge \sum_{k=0}^\infty \int_{[t_1,t_1+\ttau] \times [0,1]} \ell(\tilde{\varphi}_k(t,y)) \, dt\,dy.
	\end{equation}
	To see the claim, first note that the left hand side of \eqref{eq:I1I2L_cost} equals
	\begin{equation*}	
		\sum_{k=0}^\infty \int_{t_1}^{t_2^\varepsilon} r_k(\boldsymbol{\zeta}(t)) \ell(\rho_k(t)) \, dt.
	\end{equation*}
	Since $g(f(t))=t$, we have $g'(f(t)) f'(t)=1$ and hence the right hand side of \eqref{eq:I1I2L_cost} is
	\begin{align*}
		\sum_{k=0}^\infty \int_{t_1}^{t_1+\ttau} r_k({\boldsymbol{\tilde\zeta}}(t)) \ell(\tilde{\rho}_k(t)) \, dt
		& = \sum_{k=0}^\infty \int_{t_1}^{t_2^\varepsilon} r_k({\boldsymbol{\tilde\zeta}}(f(t))) \ell(\tilde{\rho}_k(f(t))) f'(t) \, dt \\
		& = \sum_{k=0}^\infty \int_{t_1}^{t_2^\varepsilon} r_k(\boldsymbol{\zeta}(t)) \ell\left(\frac{\rho_k(t)}{f'(t)}\right) f'(t) \, dt.
	\end{align*}	
	Combining the above two facts, we have
	\begin{align*}
		& \sum_{k=0}^\infty \int_{[t_1,t_2^\varepsilon] \times [0,1]} \ell(\varphi_k(t,y)) \, dt\,dy - \sum_{k=0}^\infty \int_{[t_1,t_1+\ttau] \times [0,1]} \ell(\tilde{\varphi}_k(t,y)) \, dt\,dy \\
		& = \sum_{k=0}^\infty \int_{t_1}^{t_2^\varepsilon} r_k(\boldsymbol{\zeta}(t)) \left[ \rho_k(t) \log f'(t) - f'(t) + 1 \right] dt \\
		& = \int_{t_1}^{t_2^\varepsilon} \left[ \left( \sum_{k=0}^\infty r_k(\boldsymbol{\zeta}(t)) \rho_k(t) \right) \log f'(t) - f'(t) + 1 \right] dt, \\
		& = \int_{t_1}^{t_2^\varepsilon} \ell(f'(t)) \, dt \; \ge 0,
	\end{align*}
	where the next to last equality uses the fact that $\sum_{k=0}^\infty r_k(\boldsymbol{\zeta}(t)) = 1$
	for all $t \in [t_1,t_2^\varepsilon)$ and the last equality uses the definition of $f'(t)$.
	This proves the claim in \eqref{eq:I1I2L_cost}.
	Combining \eqref{eq:I1I2L_first} and \eqref{eq:I1I2L_cost} with the fact that $\tilde{\boldsymbol{\varphi}} \in \mathcal{S}_{t_1+\ttau}({\boldsymbol{\tilde\zeta}},\tilde{\psi})$ and $({\boldsymbol{\tilde\zeta}},\tilde{\psi}) \in \mathcal{J}_{t_1,t_1+\ttau}^2(\boldsymbol{x}^{(1)}, \boldsymbol{x}^{(2)})$ gives
	\begin{equation*}
		I^2_{t_1,t_1+\ttau}(\boldsymbol{x}^{(1)}, \boldsymbol{x}^{(2)}) \le \sum_{k=0}^\infty \int_{[t_1,t_1+\ttau] \times [0,1]} \ell(\tilde{\varphi}_k(t,y)) \, dt\,dy \le \inf_{t_2 \ge t_1} I^1_{t_1,t_2}(\boldsymbol{x}^{(1)}, \boldsymbol{x}^{(2)}) + 2\varepsilon.
	\end{equation*}
	Since $\varepsilon \in (0,1)$ is arbitrary, \eqref{eq:25.2} follows, which, as argued previously, gives 
	\eqref{eq:I1I2}.

	Next consider \eqref{eq:IL} and the third statement in the lemma for fixed $(\boldsymbol{\zeta},\psi) \in \mathcal{J}^2_{t_1,t_1+\ttau}(\boldsymbol{x}^{(1)}, \boldsymbol{x}^{(2)})$.
	We first show that
	\begin{equation}
		\label{eq:I1I2L_second_1}
		I_{t_1,t_1+\ttau}(\boldsymbol{\zeta},\psi) \ge \int_{t_1}^{t_1+\ttau} L(\boldsymbol{\zeta}(s), \boldsymbol{\zeta}'(s))\,ds.
	\end{equation}	
	Assume without loss of generality that $I_{t_1,t_1+\ttau}(\boldsymbol{\zeta},\psi) < \infty$.
	Fix $\varepsilon \in (0,1)$.
	From Remark \ref{rmk:prep_evolution} we can find some $\boldsymbol{\varphi} \in \mathcal{S}_{t_1+\ttau}(\boldsymbol{\zeta},\psi)$ such that
	\begin{align*}
		& \sum_{k=0}^\infty \int_{[t_1,t_1+\ttau] \times [0,1]} \ell(\varphi_k(s,y)) \,ds\,dy \le I_{t_1,t_1+\ttau}(\boldsymbol{\zeta},\psi) + \varepsilon, \\
		& \varphi_k(t,y) = \rho_k(t) 1_{[0,r_k(\boldsymbol{\zeta}(t)))}(y) + 1_{[r_k(\boldsymbol{\zeta}(t)),1]}(y), \quad t \in [t_1,t_1+\ttau], k \in \mathbb{N}_0,
	\end{align*}
	for a suitable sequence of non-negative functions $\rho_k$, 
	and \eqref{eq:I1I2L_phi_k}--\eqref{eq:rmk_prep} hold for $t \in [t_1,t_1+\ttau]$.
	Using \eqref{eq:I1I2L_phi_k}, \eqref{eq:I1I2L_psi} and the fact that $\frac{d}{dt} r(\boldsymbol{\zeta}(t))=-2$ for a.e.\ $t \in [t_1,t_1+\ttau]$ we have
	\begin{align*}
		\rho_0(t)r_0(\boldsymbol{\zeta}(t)) & = -\frac{\psi'(t)+\sum_{k=1}^\infty (k-2) \zeta'_k(t)}{2} = 1+\sum_{k=1}^\infty \zeta'_k(t),
	\end{align*}
	which also implies $\sum_{k=1}^\infty \zeta_k'(t) \ge -1$ for a.e.\ $t \in [t_1,t_1+\ttau]$, proving the third statement in the lemma.
	Furthermore we have
	\begin{align}
		& \sum_{k=0}^\infty \int_{[t_1,t_1+\ttau] \times [0,1]} \ell(\varphi_k(s,y)) \,ds\,dy \notag \\
		& = \sum_{k=0}^\infty \int_{t_1}^{t_1+\ttau} r_k(\boldsymbol{\zeta}(t)) \ell(\rho_k(t)) \, dt \notag \\
		& = \int_{t_1}^{t_1+\ttau} \left[ r_0(\boldsymbol{\zeta}(t)) \ell\left( \frac{1+\sum_{k=1}^\infty \zeta'_k(t)}{r_0(\boldsymbol{\zeta}(t))} \right) + \sum_{k=1}^\infty r_k(\boldsymbol{\zeta}(t)) \ell\left(\frac{-\zeta'_k(t)}{r_k(\boldsymbol{\zeta}(t))}\right) \right] dt  \notag \\
		& = \int_{t_1}^{t_1+\ttau} L(\boldsymbol{\zeta}(t), \boldsymbol{\zeta}'(t)) \, dt, 
		\label{eq:I1I2L_second_1_2}
	\end{align}
	where the last equality uses the definition of $\ell$ in \eqref{eq:ell} and $L$ in \eqref{eq:L} and we use the convention that $0 \ell(x/0)=0$ for $x\ge 0$.	
	Therefore 
	\begin{equation}
		\label{eq:IL_1}
		\int_{t_1}^{t_1+\ttau} L(\boldsymbol{\zeta}(t), \boldsymbol{\zeta}'(t)) \, dt = \sum_{k=0}^\infty \int_{[t_1,t_1+\ttau] \times [0,1]} \ell(\varphi_k(s,y)) \,ds\,dy \le I_{t_1,t_1+\ttau}(\boldsymbol{\zeta},\psi) + \varepsilon.
	\end{equation}
	Since $\varepsilon \in (0,1)$ is arbitrary, we have \eqref{eq:I1I2L_second_1}.

	Next we show that 
	\begin{equation}
		\label{eq:I1I2L_second_2}
		I_{t_1,t_1+\ttau}(\boldsymbol{\zeta},\psi) \le \int_{t_1}^{t_1+\ttau} L(\boldsymbol{\zeta}(s), \boldsymbol{\zeta}'(s))\,ds.
	\end{equation}
	Assume without loss of generality that $\int_{t_1}^{t_1+\ttau} L(\boldsymbol{\zeta}(s), \boldsymbol{\zeta}'(s))\,ds < \infty$.
	Since there exists some $(\boldsymbol{\zeta}^*,\psi^*) \in \mathcal{J}^0_{0,t_1}(\boldsymbol{x}^{(0)}, \boldsymbol{x}^{(1)})$ such that $I_{0,t_1}(\boldsymbol{\zeta}^*,\psi^*) < \infty$, we can further assume without loss of generality that $I_{0,t_1}(\boldsymbol{\zeta},\psi) < \infty$.
	Then there exists some $\boldsymbol{\varphi}^* \in \mathcal{S}_{t_1}(\boldsymbol{\zeta},\psi)$.
	Let $\boldsymbol{\varphi}(t,y) \doteq \boldsymbol{\varphi}^*(t,y)$ for $t \in [0,t_1)$, and for $t \in [t_1,t_1+\ttau]$ define
	\begin{align*}
		\rho_k(t) & \doteq - \frac{\zeta'_k(t)}{r_k(\boldsymbol{\zeta}(t))} \one_{\{r_k(\boldsymbol{\zeta}(t)) \ne 0\}}, k \in \mathbb{N}, \\
		\rho_0(t) & \doteq \frac{\sum_{k=1}^\infty (k-2)\rho_k(t)r_k(\boldsymbol{\zeta}(t)) - \psi'(t)}{2r_0(\boldsymbol{\zeta}(t))} \one_{\{r_0(\boldsymbol{\zeta}(t)) \ne 0\}}, \\
		\varphi_k(t,y) & \doteq \rho_k(t) 1_{[0,r_k(\boldsymbol{\zeta}(t)))}(y) + 1_{[r_k(\boldsymbol{\zeta}(t)),1]}(y), y \in [0,1], k \in \mathbb{N}_0.
	\end{align*}
	Clearly \eqref{eq:I1I2L_phi_k} and \eqref{eq:I1I2L_psi} hold for $t \in [t_1,t_1+\ttau]$ and hence $\boldsymbol{\varphi} \in \mathcal{S}_{t_1+\ttau}(\boldsymbol{\zeta},\psi)$.
	Also one can check that \eqref{eq:I1I2L_second_1_2} still holds.
	Therefore
	\begin{equation*}
		I_{t_1,t_1+\ttau}(\boldsymbol{\zeta},\psi) \le \sum_{k=0}^\infty \int_{[t_1,t_1+\ttau] \times [0,1]} \ell(\varphi_k(s,y)) \,ds\,dy = \int_{t_1}^{t_1+\ttau} L(\boldsymbol{\zeta}(t), \boldsymbol{\zeta}'(t)) \, dt.
	\end{equation*}
	This gives \eqref{eq:I1I2L_second_2} and completes the proof of  \eqref{eq:IL}.
	
	 Finally, \eqref{eq:I1I2L} follows on combining \eqref{eq:I1I2},  \eqref{eq:def-I-j} and \eqref{eq:IL}.
	This completes the proof of the lemma.
\hfill \qed

\subsection{Proofs of Lemmas \ref{lem:uniqbeta} and \ref{lem:lemctybeta}.}
\label{subsec:6.2}

{\em Proof of Lemma \ref{lem:uniqbeta}.}
Consider for $(\boldsymbol{x}^{(1)}, \boldsymbol{x}^{(2)}) \in \Xi$, the function $\alpha \mapsto B(\alpha)$ on $(0,1)$, defined by
\begin{align}
	B(\alpha) \equiv B(\alpha;\boldsymbol{x}^{(1)}, \boldsymbol{x}^{(2)}) & \doteq \frac{1}{\alpha} \left( (1-\alpha^2)\sum_{k=1}^\infty \frac{k \new_k}{1-\alpha^k} - \sum_{k=1}^\infty k \new_k + x_0^{(2)} - \alpha^2 x_0^{(1)} \right) \nonumber\\
	& = \new_1  - \sum_{k=3}^\infty k \new_k \frac{\alpha - \alpha^{k-1}}{1-\alpha^k} + \frac{x_0^{(2)}}{\alpha} - \alpha x_0^{(1)} \nonumber\\
	& = \new_1  - \sum_{k=3}^\infty k \new_k B_k(\alpha) + \frac{x_0^{(2)}}{\alpha} - \alpha x_0^{(1)}, \label{eq:longbalph}
\end{align}
where 
$B_k(\alpha) \doteq (\alpha- \alpha^{k-1})/(1 - \alpha^k)$ for $k \ge 3$ and 
$\boldsymbol{z} = \boldsymbol{x}^{(1)} -\boldsymbol{x}^{(2)}$ as before.
For each $k \ge 3$ and $\alpha \in (0,1)$, using the inequality of arithmetic and geometric means one can verify that
\begin{equation}
	B_k'(\alpha) = \frac{(1-\alpha^2)(k-1)}{(1-\alpha^k)^2} \left(\frac{1+\alpha^2+\alpha^4+\dotsb+\alpha^{2k-4}}{k-1} - \alpha^{k-2}\right) > 0, \label{eq:B_k_monotone}
\end{equation}
and
\begin{equation}
	0 = B_k(0+) \le B_k(\alpha) \le B_k(1-) = (k-2)/k. \label{eq:B_k_bd}
\end{equation}
So $B(1-) = \new_1 - \sum_{k=3}^\infty (k-2)\new_k + x_0^{(2)} - x_0^{(1)} = - \left( \sum_{k=1}^\infty (k-2) \new_k + \new_0 \right) < 0$ by assumption and $B(\alpha)$ is decreasing in $\alpha \in (0,1)$.
Also note that the assumption $\sum_{k=1}^\infty k \new_k + \new_0> 2 \sum_{k=1}^\infty \new_k$ can be rewritten as
\begin{equation*}
	\sum_{k=3}^\infty (k-2) \new_k + x_0^{(1)} > \new_1 + x_0^{(2)},
\end{equation*}
which implies
\begin{equation}
	\label{eq:cor_assump}
	\mbox{ either } x_0^{(1)}>0 \mbox{ or } \new_k > 0 \mbox{ for some } k \ge 3.
\end{equation} 
From this and \eqref{eq:B_k_monotone} we see that $B(\alpha)$ is actually strictly decreasing in $\alpha \in (0,1)$.
Since each $B_k(\alpha)$ is continuous on $(0,1)$, $B(\alpha)$ is also continuous by \eqref{eq:B_k_bd} and the dominated convergence theorem.
Finally, since $B(0+) = \new_1 + \infty \cdot 1_{\{x_0^{(2)} > 0\}} > 0$ and $B(1-) <0$, there must exist a unique $\beta \in(0,1)$ such that $B(\beta)=0$.
This completes the proof of the lemma.
\hfill \qed\\

\noindent {\em Proof of Lemma  \ref{lem:lemctybeta}.}
Suppose $(\boldsymbol{x}^{(1),n}, \boldsymbol{x}^{(2),n}) \to (\boldsymbol{x}^{(1)}, \boldsymbol{x}^{(2)})$ as $n \to \infty$, where $(\boldsymbol{x}^{(1),n}, \boldsymbol{x}^{(2),n}), (\boldsymbol{x}^{(1)}, \boldsymbol{x}^{(2)}) \in \Xi$.
Recall the function $B(\cdot)$ defined above \eqref{eq:B_k_monotone} and the definition of $\beta(\cdot)$ from Section \ref{sec:degree_distribution}.
We consider two possible cases for the values of $x_0^{(2)}$ and $\new_1$. 

{\bf Case 1:} $x_0^{(2)}=0$ and $\new_1=0$. 
In this case $\beta = \beta(\boldsymbol{x}^{(1)}, \boldsymbol{x}^{(2)})=0$ by definition and $x_0^{(2)} \log \beta=0$ by our convention. 
Since $B(0+) = \new_1 + \infty \cdot 1_{\{x_0^{(2)} > 0\}} = 0$ and $B(\alpha)$ is strictly decreasing in $\alpha \in (0,1)$, we have $B(\alpha)<0$ for every $\alpha \in (0,1)$.
Fixing $\alpha \in (0,1)$, from \eqref{eq:B_k_bd} and the dominated convergence theorem one has 
\begin{equation}
	\label{eq:B_n_cvg}
	B^{(n)}(\alpha) \doteq B(\alpha;\boldsymbol{x}^{(1),n}, \boldsymbol{x}^{(2),n}) \to B(\alpha)
	\doteq B(\alpha;\boldsymbol{x}^{(1)}, \boldsymbol{x}^{(2)})
\end{equation} 
as $n \to \infty$.
Therefore $B^{(n)}(\alpha)<0$ for sufficiently large $n$.
Since $B^{(n)}$ is decreasing, we must have $\beta^{(n)} \doteq \beta(\boldsymbol{x}^{(1),n}, \boldsymbol{x}^{(2),n}) \le \alpha$ for all such $n$.
Since $\alpha \in (0,1)$ is arbitrary, this implies that as $n \to \infty$, $\beta^{(n)} \to 0 = \beta$.
Next note that the convergence of $x_0^{(2),n} \log \beta^{(n)} \to x_0^{(2)} \log \beta =0$ holds trivially if 
$x_0^{(2),n} =0$ for all sufficiently large $n$. Suppose now that $x_0^{(2),n} >0$ for every $n$. Also  take $n$ to be sufficiently large, so that $x_0^{(2),n}<1$.
From \eqref{eq:longbalph} and since $kB^{(n)}_k(\alpha) \le (k-2)$ from \eqref{eq:B_k_bd} [applied with $(\boldsymbol{x}^{(1)}, \boldsymbol{x}^{(2)})$ replaced by $(\boldsymbol{x}^{(1),n}, \boldsymbol{x}^{(2),n})$] we have
\begin{equation*}
	B^{(n)}((x_0^{(2),n})^2) \ge 0 - \sum_{k=3}^\infty (k-2)\new_k^n + \frac{1}{x_0^{(2),n}} - (x_0^{(2),n})^2x_0^{(1),n} > 0
\end{equation*}
for $x_0^{(2),n}$ sufficiently small.
So $\beta^{(n)} \ge (x_0^{(2),n})^2$ for  such $n$ and so $x_0^{(2),n} \log \beta^{(n)} \to 0 = x_0^{(2)} \log \beta$.
	
{\bf Case 2:} $x_0^{(2)} > 0$ or $\new_1 > 0$.
In this case, for $n$ sufficiently large,  we must have $x_0^{(2),n} > 0$ or $\new_1^n > 0$.	
So $\beta^{(n)}$ satisfies $B^{(n)}(\beta^{(n)})=0$ for all such $n$.
Since $\beta>0$, and by proof of Lemma \ref{lem:uniqbeta} $B(\beta)=0$, $B(0+) >0$ and $B(\cdot)$ is strictly decreasing, we have $B(\beta/2)>0$.
As in the proof of \eqref{eq:B_n_cvg}, we see that $B^{(n)}(\beta/2) \to B(\beta/2)$ as $n\to \infty$, and so
 $B^{(n)}(\beta/2) > 0$ for all sufficiently large $n$.
Since $B^{(n)}$ is decreasing, we must have $\beta^{(n)} \ge \beta/2>0$ for all sufficiently large $n$.
From this, \eqref{eq:B_k_bd} and the dominated convergence theorem one can show that along any convergent subsequence of $\beta^{(n)}$, $B^{(n)}(\beta^{(n)}) \to B(\lim \beta^{(n)})$.
So any  limit point of $\beta^{(n)}$ is a solution to $B(\alpha)=0$ defined on $(0,1)$.
But $\beta$ is the unique solution to this equation.
So  $\beta^{(n)} \to \beta$ and also $x_0^{(2),n} \log \beta^{(n)} \to x_0^{(2)} \log \beta$.
This completes the proof of the lemma. \hfill \qed

\subsection{Proof of Lemma \ref{lem:minimizer-def}}
\label{subsec:6.3}

(a) 
Recall the definition of $\ttau$ in \eqref{eq:tau} and $\tilde{\ttau},\tilde{z}_k$ from Construction \ref{cons:cont}.
From \eqref{eq:def-beta} we have
\begin{equation}\label{eq:taub}
	\tilde \ttau = \frac{\ttau}{1-\beta^2} = \frac{x^{(1)}_0-x^{(2)}_0 + \sum_{k=1}^\infty k \new_k}{2(1-\beta^2)} = \frac{1}{2} \left( x^{(1)}_0 + \sum_{k=1}^\infty k \tilde{\new}_k \right).
\end{equation}
Since $\beta \in [0,1)$, we have $\ttau \le \ttau/(1-\beta^2) = \tilde{\ttau}$.
This proves part (a).


	(b) We first show that $({\boldsymbol{\tilde\zeta}},\tilde{\psi}) \in \mathcal{J}^1_{t_1,t_1+\ttau}(\boldsymbol{x}^{(1)}, \boldsymbol{x}^{(2)})$.
	For this, it suffices to check
	\begin{align}
		\tilde{\zeta}_k(t_1+\ttau) & = p_k^{(2)} \mbox{ for } k \in \mathbb{N},  \;
				\tilde{\zeta}_0(t_1+\ttau)  = x^{(2)}_0, \label{eq:minimizer-condition-0} \\
		\tilde{\psi}(t) - \tilde{\psi}(t_1) & \ge -x_0^{(1)} 
\mbox{ for } t \in [t_1,t_1+\ttau]. \label{eq:minimizer-condition-2}
	\end{align}
	From \eqref{eq:def-minimizer} we have 
	$$\tilde{\zeta}_k(t_1+\ttau) = p_k^{(1)} - \tilde \new_k( 1 - (1- \ttau/\tilde \ttau)^{k/2})=  p_k^{(1)} - \tilde \new_k( 1 - \beta^k) = p_k^{(1)} - \new_k= p_k^{(2)},$$ 
	which gives the first statement in \eqref{eq:minimizer-condition-0}. 
	From this, \eqref{eq:def-minimizer-0} and \eqref{eq:tau} it follows that
	\begin{equation*}
		\tilde{\zeta}_0(t_1+\ttau) = x_0^{(1)} + \sum_{k=1}^\infty k(p_k^{(1)}-p_k^{(2)}) - 2\ttau = x_0^{(2)}, 
	\end{equation*}
	which gives the second statement in \eqref{eq:minimizer-condition-0}.
	
	For \eqref{eq:minimizer-condition-2}, applying the change of variable $t - t_1 = \tilde \ttau (1- \alpha_t^2)$, namely $\alpha_t = \left( 1-\frac{t-t_1}{\tilde \ttau}\right)^{1/2}$ for $t \in [t_1,t_1+\ttau]$, we have
	\begin{align*}
		\sum_{k=1}^\infty k(p_k^{(1)}-\tilde{\zeta}_k(t)) - 2(t-t_1) & = \sum_{k=1}^\infty k \tilde \new_k \left[ 1- \left( 1-\frac{t-t_1}{\tilde \ttau}\right)^{k/2} \right] -2(t-t_1) \\
		& = \sum_{k=1}^\infty k \tilde \new_k ( 1- \alpha_t^k) - 2\tilde \ttau (1- \alpha_t^2) \\
		& = \sum_{k=1}^\infty k \tilde \new_k ( 1- \alpha_t^k) - \sum_{k=1}^\infty k \tilde{\new}_k (1-\alpha_t^2) - x^{(1)}_0 (1-\alpha_t^2) \doteq F(\alpha_t),
	\end{align*}
	where the third equality follows from part (a).
	Using this we can write, for $t \in [t_1,t_1+\ttau]$,
	\begin{equation}
		\label{eq:minimizer-condition-alpha}
		\tilde{\zeta}_0(t) = x_0^{(1)} + F(\alpha_t) = \tilde{\psi}(t) - \tilde{\psi}(t_1) + x_0^{(1)}, 
	\end{equation}
	Note that, for $t \in [t_1,t_1+\ttau]$,
	\begin{align*}
		x_0^{(1)} + F(\alpha_t) & = \alpha_t^2 x_0^{(1)} + \sum_{k=1}^\infty k \tilde \new_k(\alpha_t^2 -\alpha_t^k) = \alpha_t(1 - \alpha_t) \left( \frac{\alpha_t x_0^{(1)}}{1-\alpha_t} -\tilde \new_1  + \sum_{k=3}^\infty k \tilde \new_k \frac{\alpha_t- \alpha_t^{k-1}}{1-\alpha_t}  \right) \\
		& = \alpha_t(1 - \alpha_t) \left( \frac{\alpha_t x_0^{(1)}}{1-\alpha_t} -\tilde \new_1  + \sum_{k=3}^\infty k \tilde \new_k \tilde{B}_k(\alpha_t) \right) = \alpha_t (1-\alpha_t) \tilde{B}(\alpha_t),
	\end{align*}
	where $\tilde{B}_k(\alpha) \doteq (\alpha- \alpha^{k-1})/(1-\alpha)$ for $k \ge 3$ and $\Btil(\alpha) \doteq \frac{\alpha x_0^{(1)}}{1-\alpha} -\tilde \new_1  + \sum_{k=3}^\infty k \tilde \new_k \tilde{B}_k(\alpha)$.
	One can verify (e.g.\ using Young's inequality) that
	\begin{equation*}
		\tilde{B}_k'(\alpha) = \frac{(k-2)\alpha^{k-1}-(k-1)\alpha^{k-2}+1}{(\alpha-1)^2} > 0, \quad k \ge 3, \: \alpha\in[0,1).
	\end{equation*}
	So $\tilde{B}(\alpha)$ is increasing.
	Using \eqref{eq:def-beta} one can verify that $\tilde{B}(\beta)=\frac{x_0^{(2)}}{\beta(1-\beta)} 1_{\{\beta>0\}} \ge 0$.
	Since for $t \in (t_1, t_1+\ttau]$, $\alpha_t \in [\beta, 1)$,  for all such $t$
	$x_0^{(1)} + F(\alpha_t) \ge 0$.
	This along with \eqref{eq:minimizer-condition-alpha} gives \eqref{eq:minimizer-condition-2}.
	
	So far we have verified that $({\boldsymbol{\tilde\zeta}},\tilde{\psi}) \in \mathcal{J}^1_{t_1,t_2}(\boldsymbol{x}^{(1)}, \boldsymbol{x}^{(2)})$.
	From \eqref{eq:def-minimizer-0} we also have $\frac{d}{dt} r({\boldsymbol{\tilde\zeta}}(t)) = -2$ for $t \in [t_1,t_1+\ttau]$.
	Thus actually $({\boldsymbol{\tilde\zeta}},\tilde{\psi}) \in \mathcal{J}^2_{t_1,t_2}(\boldsymbol{x}^{(1)}, \boldsymbol{x}^{(2)})$, completing the proof of (b).	
	
	(c) Since for $t \in (t_1, t_1+\ttau)$, $\tilde{\zeta}_0(t) = x_0^{(1)} + F(\alpha_t) = \alpha_t (1-\alpha_t) \tilde{B}(\alpha_t)$ and $\tilde{B}(\beta) \ge 0$, it suffices to show that $\tilde{B}(\alpha)$ is strictly increasing in $\alpha \in [\beta,1)$.
	But thanks to \eqref{eq:cor_assump}, this is immediate from the fact that $\tilde{B}_k'(\alpha)>0$ and $\frac{\alpha x_0^{(1)}}{1-\alpha}$ is strictly increasing when $x_0^{(1)}>0$.
	This gives part (c) and completes the proof of the lemma.
\hfill \qed

\subsection{Proof of Lemma \ref{lem:minimizer-cost}}
\label{subsec:6.4}

	Let $\mu_0 \doteq \frac{1}{2} \left( x_0^{(1)} + \sum_{k=1}^\infty k p_k^{(1)} \right)$.
	Recall $\tilde \new_k = \frac{\new_k}{1-\beta^k}$. 
	It then follows from \eqref{eq:def-minimizer}, \eqref{eq:def-minimizer-0} and Lemma \ref{lem:minimizer-def}(a) that for $t \in [t_1,t_1+\ttau]$,
	\begin{align}
		\tilde \zeta'_k(t) & = -\frac{k \tilde \new_k}{2 \tilde \ttau-2(t-t_1)} \left( 1- \frac{t-t_1}{\tilde \ttau} \right)^{k/2} = -\frac{k}{2\tilde \ttau-2(t-t_1)} [\tilde \zeta_k(t)-p_k^{(1)} + \tilde \new_k], \label{eq:Euler-Lagrange-key-1} \\
		1 + \sum_{k=1}^\infty \tilde \zeta'_k(t) & = \frac{x_0^{(1)} - 2(t-t_1) + \sum_{k=1}^\infty k [p_k^{(1)} - \tilde \zeta_k(t)]}{2\tilde \ttau -2(t-t_1)} = \frac{\tilde \zeta_0(t)}{2\tilde \ttau - 2(t-t_1)}, \label{eq:Euler-Lagrange-key-2} \\
		r(\boldsymbol{\tilde \zeta}(t)) & = 2\mu_0 -2(t-t_1) \notag.
	\end{align}	
	From these we have
	\begin{align}
		& \int_{t_1}^{t_1+\ttau} L(\boldsymbol{\tilde\zeta}(s), \boldsymbol{\tilde\zeta}'(s))\,ds \notag \\
	 	& = \int_{t_1}^{t_1+\ttau} \left[ \left(1 + \sum_{k=1}^\infty \tilde\zeta'_k(t)\right) \log \left(\frac{1 + \sum_{k=1}^\infty\tilde\zeta'_k(t) }{\tilde\zeta_0(t)/r(\boldsymbol{\tilde\zeta}(t))}\right) + \sum_{k=1}^\infty(- \tilde\zeta'_k(t)) \log \left(\frac{-\tilde\zeta'_k(t)}{k\tilde\zeta_k(t)/r(\boldsymbol{\tilde\zeta}(t))}\right) \right] dt \notag \\
		& = \int_{t_1}^{t_1+\ttau} \left[  \log(2 \mu_0-2(t-t_1)) -  \log (2 \tilde \ttau - 2(t-t_1))	- \sum_{k=1}^\infty  \tilde\zeta'_k(t) \log \left(\frac{\tilde\zeta_k(t)-p_k^{(1)}+ \tilde \new_k}{\tilde\zeta_k(t)}\right) \right] dt. \label{eq:H-K-1}
	\end{align}
	We claim that we can interchange the integration and summation in the last line.
	To see this, first note that there exists some $M \in \mathbb{N}$ such that $\new_k \le \tilde{\new}_k \le 2 \new_k \le 2 p_k^{(1)} < 1$ for $k \ge M$.
	Since $\tilde\zeta_k(t)$ is non-increasing, we have
	
	\begin{align*}
		& \sum_{k=M}^\infty \int_{t_1}^{t_1+\ttau} \left| \tilde\zeta'_k(t) \log \left( \frac{\tilde\zeta_k(t)-p_k^{(1)}+ \tilde \new_k}{\tilde\zeta_k(t)} \right) \right| dt \\
		& = \sum_{k=M}^\infty \int_{t_1}^{t_1+\ttau} \left| \log \left( \frac{\tilde\zeta_k(t)-p_k^{(1)}+ \tilde \new_k}{\tilde\zeta_k(t)} \right) \right| d (p_k^{(1)}-\tilde\zeta_k(t)) 
		= \sum_{k=M}^\infty  \int_0^{\new_k}  \left| \log \left(\frac{ \tilde \new_k - u}{p_k^{(1)} - u}\right)\right| du \\
		& \le \sum_{k=M}^\infty  \int_0^{\new_k} \left( -\log (\tilde \new_k - u) - \log (p_k^{(1)} - u) \right) du
		\le - 2\sum_{k=M}^\infty \int_0^{2p_k^{(1)}} \log u \,du.
	\end{align*}
	Using $\tilde{\ell}(x) \doteq x \log x - x = \int \log x \, dx$, the last expression equals
	\begin{equation*}
		- 2\sum_{k=M}^\infty (\tilde{\ell}(2p_k^{(1)}) - \tilde{\ell}(0)) = 4\sum_{k=M}^\infty p_k^{(1)} \log \left( \frac{1}{p_k^{(1)}} \right) -4(\log 2 -1)\sum_{k=M}^\infty p_k^{(1)}.
	\end{equation*}
	Here the last term is clearly finite.
	Letting $\tilde{M} \doteq \sum_{k=M}^\infty p_k^{(1)} \in (0,1]$, we have
	\begin{align*}
		\sum_{k=M}^\infty p_k^{(1)} \log \left( \frac{1}{p_k^{(1)}} \right) & = \tilde{M} \sum_{k=M}^\infty \frac{p_k^{(1)}}{\tilde{M}} \log \left( \frac{1}{k^2 p_k^{(1)}} \right) + \sum_{k=M}^\infty p_k^{(1)} \log k^2 \\
		& \le \tilde{M} \log \left( \sum_{k=M}^\infty \frac{1}{\tilde{M}k^2} \right) + 2 \sum_{k=M}^\infty kp_k^{(1)} < \infty,
	\end{align*}
	where the  inequality holds since $\log x$ is concave and $\log x \le x$.
	Therefore
	\begin{equation*}
		\sum_{k=M}^\infty \int_{t_1}^{t_1+\ttau} \left| \tilde\zeta'_k(t) \log \left( \frac{\tilde\zeta_k(t)-p_k^{(1)}+ \tilde \new_k}{\tilde\zeta_k(t)} \right) \right| dt < \infty.
	\end{equation*}
	One can easily verify that for $1 \le k \le M$,
	\begin{equation*}
		\int_{t_1}^{t_1+\ttau} \left| \tilde\zeta'_k(t) \log \left( \frac{\tilde\zeta_k(t)-p_k^{(1)}+ \tilde \new_k}{\tilde\zeta_k(t)} \right) \right| dt = \int_0^{\new_k}  \left| \log \left(\frac{ \tilde \new_k - u}{p_k^{(1)} - u}\right)\right| du < \infty.
	\end{equation*}
	So the claim holds.
	Actually we have also shown that $\int_{t_1}^{t_1+\ttau} L({\boldsymbol{\tilde\zeta}}(s), {\boldsymbol{\tilde\zeta}}'(s))\,ds < \infty$.
	
	From \eqref{eq:H-K-1} it then follows that
	\begin{align}
		& \int_{t_1}^{t_1+\ttau} L(\boldsymbol{\tilde\zeta}(s), \boldsymbol{\tilde\zeta}'(s))\,ds \notag \\
		& = \int_{t_1}^{t_1+\ttau} \left[  \log(\mu_0 - (t-t_1)) -  \log (\tilde \ttau - (t-t_1)) \right] dt  - \sum_{k=1}^\infty \int_{t_1}^{t_1+\ttau} \log \left( \frac{\tilde\zeta_k(t)-p_k^{(1)}+ \tilde \new_k}{\tilde\zeta_k(t)} \right) d\tilde\zeta_k(t) \notag \\
		& = \left. \left[ -\tilde{\ell}(\mu_0 - (t-t_1)) + \tilde{\ell} (\tilde \ttau - (t-t_1)) - \sum_{k=1}^\infty  \tilde{\ell}(\tilde\zeta_k(t)-p_k^{(1)} + \tilde \new_k) + \sum_{k=1}^\infty  \tilde{\ell}(\tilde\zeta_k(t)) \right] \right|_{t=t_1}^{t_1+\ttau} \notag \\
		& = -(\mu_0 -\ttau)\log(\mu_0 - \ttau) + ( \tilde \ttau-\ttau) \log (\tilde \ttau- \ttau)  + \mu_0 \log \mu_0 - \tilde \ttau \log \tilde \ttau \notag \\
		& \quad + \sum_{k=1}^\infty \left[ -(\tilde \new_k - \new_k) \log (\tilde \new_k - \new_k) + p_k^{(2)}\log p_k^{(2)} + \tilde \new_k \log \tilde \new_k - p_k^{(1)} \log p_k^{(1)} \right], \label{eq:H-K-2}
	\end{align}
	where the last line follows from $\boldsymbol{\tilde\zeta}(t_1) = \boldsymbol{x}^{(1)}$ and $\boldsymbol{\tilde\zeta}(t_1+\ttau) = \boldsymbol{x}^{(2)}$.
	Using  $\tilde \ttau = \ttau/(1-\beta^2)$, 
	\begin{align*}
		( \tilde \ttau-\ttau) \log (\tilde \ttau- \ttau) - \tilde \ttau \log \tilde \ttau 
		& = \frac{\beta^2 \ttau}{1- \beta^2}  \log \frac{\beta^2 \ttau}{1- \beta^2} - \frac{\ttau}{1-\beta^2} \log \frac{\ttau}{1-\beta^2}\\
		& = - \ttau \log \ttau + \ttau \log (1-\beta^2) + \frac{2\beta^2 \ttau}{1- \beta^2}  \log \beta.
	\end{align*}
	Since $\tilde \new_k = \new_k / (1 - \beta^k)$, we have
	\begin{align*}
		 &\sum_{k=1}^\infty \left[ -(\tilde \new_k - \new_k) \log (\tilde \new_k - \new_k)  + \tilde \new_k \log \tilde \new_k \right] \\
		& \quad= \sum_{k=1}^\infty \left[ \new_k \log \new_k - \new_k \log(1-\beta^k) - \frac{k\beta^k \new_k}{1-\beta^k} \log \beta \right]\\
		& \quad= \sum_{k=1}^\infty \left[ \new_k \log \new_k - \new_k \log(1-\beta^k) \right] + \sum_{k=1}^\infty \left( k\new_k- \frac{k\new_k}{1-\beta^k} \right) \log \beta\\
		& \quad= \sum_{k=1}^\infty \left[ \new_k \log \new_k - \new_k \log(1-\beta^k) \right] + \left( x_0^{(2)} -\frac{2\beta^2 \ttau}{1- \beta^2} \right) \log \beta,
	\end{align*}
	where the last line is from \eqref{eq:def-beta} and \eqref{eq:taub}. 
	The last two displays along with \eqref{eq:H-K-2} give
	\begin{align}
		& \int_{t_1}^{t_1+\ttau} L(\boldsymbol{\tilde\zeta}(s), \boldsymbol{\tilde\zeta}'(s))\,ds \notag \\
		& \quad= -(\mu_0 -\ttau)\log(\mu_0 - \ttau) - \ttau \log \ttau  + \mu_0 \log \mu_0  + \ttau \log (1-\beta^2) \notag \\
		& \quad\quad + \sum_{k=1}^\infty \left[ \new_k \log \new_k + p_k^{(2)} \log p_k^{(2)}  - p_k^{(1)} \log p_k^{(1)} \right] - \sum_{k=1}^\infty \new_k \log(1-\beta^k) + x_0^{(2)} \log \beta \label{eq:semi-cont}\\
		& \quad= \tilde{H}(\boldsymbol{\new}) + \tilde{H}(\boldsymbol{x}^{(2)}) - \tilde{H}(\boldsymbol{x}^{(1)}) + \tilde{K}(\boldsymbol{x}^{(1)}, \boldsymbol{x}^{(2)}). \notag
	\end{align}
	Finiteness of the above  follows as in Remark \ref{rem:finhk}.
	This gives the first statement in the lemma.
	
	For the lower semicontinuity, first note that $-(\mu_0 -\ttau)\log(\mu_0 - \ttau) - \ttau \log \ttau  + \mu_0 \log \mu_0  + \ttau \log (1-\beta^2) - \sum_{k=1}^\infty \new_k \log(1-\beta^k) + x_0^{(2)} \log \beta$ is continuous from Lemma \ref{lem:lemctybeta} and Assumption \ref{asp:exponential-boundN}.
	The remaining terms in \eqref{eq:semi-cont} can be written as
	\begin{align*}
		&\sum_{k=1}^\infty \left[ \new_k \log \new_k + p_k^{(2)} \log p_k^{(2)}  - p_k^{(1)} \log p_k^{(1)} \right]\\
		 &=
		 \sum_{k=0}^\infty \left[ \new_k \log \new_k + p_k^{(2)} \log p_k^{(2)}  - p_k^{(1)} \log p_k^{(1)} \right]
		- \left[ \new_0 \log \new_0 + p_0^{(2)} \log p_0^{(2)}  - p_0^{(1)} \log p_0^{(1)} \right]\\
		&=\sum_{k=0}^\infty \left[ \new_k \log \frac{\new_k}{p_k^{(1)}} \right] + \sum_{k=0}^\infty \left[ p_k^{(2)} \log \frac{p_k^{(2)}}{p_k^{(1)}} \right] - \left[ \new_0 \log \new_0 + p_0^{(2)} \log p_0^{(2)}  - p_0^{(1)} \log p_0^{(1)} \right],
	\end{align*}
	where $\new_0 \doteq 1- \sum_{k=1}^\infty \new_k$, and $p_0^{(i)} \doteq 1- \sum_{k=1}^\infty p_k^{(i)}$ for $i=1,2$.
	The last term in the above display is clearly a lower semicontinuous function of $(\boldsymbol{x}^{(1)}, \boldsymbol{x}^{(2)}) \in \Xi$. 
	The lemma follows.
\hfill \qed

\subsection{Proof of Lemma \ref{lem:minimizer-verify-general}}
\label{subsec:6.5}

We begin with a lemma that gives the statement in Lemma \ref{lem:minimizer-verify-general} under a stronger assumption.

\begin{Lemma}
	\label{lem:minimizer-verify-special}
	Suppose the same setting as in Lemma \ref{lem:minimizer-verify-general}. 
	Suppose in addition that: (i) $x^{(1)}_0,x^{(2)}_0>0$, and (ii) for every $k \in \mathbb{N}$, if $p^{(1)}_k > 0$ then $p^{(2)}_k > 0$.
	Then \eqref{eq:i2ttau} is satisfied.
\end{Lemma}

\begin{proof}
	The first equality in \eqref{eq:i2ttau} is proved in Lemma \ref{lem:I1I2L}.
	For the second equality, we need to show that $({\boldsymbol{\tilde\zeta}},\tilde{\psi})$ is the minimizer of the function
	\begin{equation*}
		\tilde{G}(\boldsymbol{\zeta},\psi) \doteq \int_{t_1}^{t_1+\ttau} L(\boldsymbol{\zeta}(s), \boldsymbol{\zeta}'(s))\,ds, \quad (\boldsymbol{\zeta},\psi) \in \mathcal{J}^2_{t_1,t_1+\ttau}(\boldsymbol{x}^{(1)}, \boldsymbol{x}^{(2)}).
	\end{equation*}
	We will prove this via contradiction.
	First note that $\mathcal{J}^2_{t_1,t_1+\ttau}(\boldsymbol{x}^{(1)}, \boldsymbol{x}^{(2)})$ is a convex set.
	Also using the definition of $L$, one can verify that $\tilde{G}(\boldsymbol{\zeta},\psi)$ is a convex function in $(\boldsymbol{\zeta},\psi) \in \mathcal{J}^2_{t_1,t_1+\ttau}(\boldsymbol{x}^{(1)}, \boldsymbol{x}^{(2)})$.
	Now suppose there exists some $(\boldsymbol{\zeta},\psi) \in \mathcal{J}^2_{t_1,t_1+\ttau}(\boldsymbol{x}^{(1)}, \boldsymbol{x}^{(2)})$ such that $\tilde{G}(\boldsymbol{\zeta},\psi) < \tilde{G}({\boldsymbol{\tilde\zeta}},\tilde{\psi})$.
	From Lemma \ref{lem:minimizer-cost} 
	we have $\tilde{G}({\boldsymbol{\tilde\zeta}},\tilde{\psi}) < \infty$.
	For $\varepsilon \in [0,1]$, construct the family of paths $(\boldsymbol{\zeta}^\varepsilon,\psi^\varepsilon) \doteq (1-\varepsilon) ({\boldsymbol{\tilde\zeta}},\tilde{\psi}) + \varepsilon (\boldsymbol{\zeta},\psi)$. 
	Letting $g(\varepsilon) \doteq \tilde{G}(\boldsymbol{\zeta}^\varepsilon,\psi^\varepsilon)$, we have $g(1) = \tilde{G}(\boldsymbol{\zeta},\psi) < \tilde{G}({\boldsymbol{\tilde\zeta}},\tilde{\psi}) = g(0)$.
	It follows from the convexity that $g$ is left and right differentiable wherever it is finite.
	We will show that $g'_+(0)=0$, where $g'_+(\cdot)$ is the right derivative of $g$.
	The convexity of $g$ will then give the desired contradiction.	
	
	By convexity of $g$, we have $g(\varepsilon) < g(0)$ for every $\varepsilon \in (0,1]$.
	From Lemma \ref{lem:minimizer-def}(c), assumption (i) and continuity of $\tilde{\zeta}_0$ we have
	\begin{equation}
		\label{eq:deltadefn}
		\delta \doteq \inf_{t \in [t_1,t_1+\ttau]} \tilde{\zeta}_0(t) > 0.
	\end{equation}
	From \eqref{eq:Euler-Lagrange-key-2} we see 
	\begin{equation}
		\label{eq:minimizer-verify-special-temp}
		1 + \sum_{k=1}^\infty \tilde\zeta'_k(t) = \frac{\tilde\zeta_0(t)}{2\tilde \ttau - 2(t-t_1)} \ge \frac{\delta}{2\tilde{\ttau}} > 0, \quad t \in [t_1,t_1+\ttau].
	\end{equation}
	Now fix $0 < \varepsilon < \frac{1}{4} \wedge \delta \wedge \frac{\delta}{2\tilde{\ttau}}$.
	Then $\zeta^\varepsilon_0(t) > \frac{\delta}{2}$ for all $t \in [t_1,t_1+\ttau]$.
	
	We next argue that one can assume without loss of generality that	
	\begin{equation}
			\label{eq:minimizer-finite-dim}
			\zeta_k(t) = \tilde{\zeta}_k(t) \mbox{ for all } t \in [t_1,t_1+\ttau] \mbox{ and } k \ge n_0
	\end{equation}
	for some large enough $n_0 \in \Nmb$.
	To show this, we define $(\boldsymbol{\zeta}^n,\psi^n)$ for $n \in \mathbb{N}$ as follows: For $t \in [0,t_1)$, $(\boldsymbol{\zeta}^n(t),\psi^n(t)) \doteq (\boldsymbol{\zeta}^\varepsilon(t),\psi^\varepsilon(t))$, and for $t \in [t_1,t_1+\ttau]$, 
	\begin{align*}
		\zeta^n_k(t) & \doteq \tilde{\zeta}_k(t), \quad k \ge n, \\ 
		\zeta^n_k(t) & \doteq \zeta^\varepsilon_k(t), \quad 1 \le k < n, \\
		\zeta^n_0(t) & \doteq x_0^{(1)} + \sum_{k=1}^\infty k(p_k^{(1)}-\zeta^n_k(t)) - 2(t-t_1), \\
		\psi^n(t) & \doteq \psi^\varepsilon(t_1) + \sum_{k=1}^\infty k(p_k^{(1)}-\zeta^n_k(t)) - 2(t-t_1).
	\end{align*}
	From this definition we have $(\zeta_k^n)_{k \in \mathbb{N}} \to (\zeta^\varepsilon_k)_{k \in \mathbb{N}}$ in $\mathbb{C}([0,t_1+\ttau]:\mathbb{R}_+^\infty)$ as $n \to \infty$.
	So $(\zeta_0^n,\psi^n) \to (\zeta^\varepsilon_0,\psi^\varepsilon)$ in $\mathbb{C}([0,t_1+\ttau]:\mathbb{R}^2)$ as $n \to \infty$.
	From this we see $\psi^n(t) - \psi^n(t_1) + x_0^{(1)} = \zeta_0^n(t)$ is uniformly bounded away from $0$ in $t \in[t_1,t_1+\ttau]$ for sufficiently large $n$.
	So $\boldsymbol{\zeta}^n \in \mathcal{J}^2_{t_1,t_1+\ttau}(\boldsymbol{x}^{(1)}, \boldsymbol{x}^{(2)})$ for all such $n$.
	Recall $L_k$ and $L$ defined in \eqref{eq:L}.
	Using the definition of $\zeta_k^n$ for $1 \le k < n$ 
	\begin{align}
		\tilde{G}(\boldsymbol{\zeta}^n,\psi^n) - \tilde{G}(\boldsymbol{\zeta}^\varepsilon,\psi^\varepsilon)
		& = \int_{t_1}^{t_1+\ttau} [L(\boldsymbol{\zeta}^n(s),(\boldsymbol{\zeta}^n(s))') - L(\boldsymbol{\zeta}^\varepsilon(s),(\boldsymbol{\zeta}^\varepsilon(s))')] \,ds \nonumber\\
		& = \int_{t_1}^{t_1+\ttau} [L_0(\boldsymbol{\zeta}^n(s),(\boldsymbol{\zeta}^n(s))') - L_0(\boldsymbol{\zeta}^\varepsilon(s),(\boldsymbol{\zeta}^\varepsilon(s))')] \,ds \nonumber\\
		& \quad + \int_{t_1}^{t_1+\ttau} \sum_{k=n}^\infty [L_k(\boldsymbol{\zeta}^n(s),(\boldsymbol{\zeta}^n(s))') - L_k(\boldsymbol{\zeta}^\varepsilon(s),(\boldsymbol{\zeta}^\varepsilon(s))')] \,ds. \label{eq:eqsun11}
	\end{align}
	We claim that both terms on the right side converge to $0$ as $n \to \infty$.
	To see this, note that
	$$\Scale[0.92]{L_0(\boldsymbol{\zeta}^n(s),(\boldsymbol{\zeta}^n(s))') = \left(1+\sum_{k=1}^\infty (\zeta^n_k)'(s)\right) \log \left[ \left(1+\sum_{k=1}^\infty (\zeta^n_k)'(s)\right) \Big/ \left(\frac{\zeta^n_0(s)}{r(\boldsymbol{\zeta}^n(s))}\right) \right] \to L_0(\boldsymbol{\zeta}^\varepsilon(s),(\boldsymbol{\zeta}^\varepsilon(s))')}$$
	as $n \to \infty$, for each $s \in [t_1,t_1+\ttau]$.
	From \eqref{eq:minimizer-verify-special-temp} and the choice of $\varepsilon$ we have that
	$$1 \ge 1+\sum_{k=1}^\infty (\zeta^n_k)'(s) \ge \frac{\delta}{2\tilde\ttau} - \varepsilon > 0.$$
	Since $\zeta^\varepsilon_0(s)$ and $\zetatil_0(s)$ are both bounded from above and away from $0$ for all $s \in [t_1,t_1+\ttau]$, 
	$$\sup_{n \in \Nmb} \sup_{s \in [t_1,t_1+\ttau]} |L_0(\boldsymbol{\zeta}^n(s),(\boldsymbol{\zeta}^n(s))')| < \infty.$$
	The first term on the right side of \eqref{eq:eqsun11} then converges to $0$ as $n \to \infty$ by the dominated convergence theorem.
	For the second term, note that
	$$L_k(\boldsymbol{\zeta}^n(s),(\boldsymbol{\zeta}^n(s))') = -(\zeta^n_k)'(s) \log \left[ - (\zeta^n_k)'(s) \Big/ \left(\frac{k\zeta_k^n(s)}{r(\boldsymbol{\zeta}^n(s))}\right) \right] \to L_k(\boldsymbol{\zeta}^\varepsilon(s),(\boldsymbol{\zeta}^\varepsilon(s))')$$
	as $n \to \infty$, for each $s \in [t_1,t_1+\ttau]$.
	Since $\boldsymbol{\zeta}^n \in \mathcal{J}^2_{t_1,t_1+\ttau}(\boldsymbol{x}^{(1)}, \boldsymbol{x}^{(2)})$, we have $r(\boldsymbol{\zeta}^n(s)) = r(\boldsymbol{\zetatil}(s)) = r(\boldsymbol{\zeta}^\varepsilon(s))$ for each $s \in [t_1,t_1+\ttau]$, and hence
	$$|L_k(\boldsymbol{\zeta}^n(s),(\boldsymbol{\zeta}^n(s))')| \le |L_k(\boldsymbol{\zetatil}(s),\boldsymbol{\zetatil}'(s))| + |L_k(\boldsymbol{\zeta}^\varepsilon(s),(\boldsymbol{\zeta}^\varepsilon(s))')|.$$
	Since $\tilde{G}(\boldsymbol{\zeta}^\varepsilon,\psi^\varepsilon)<\infty$ and $\tilde{G}({\boldsymbol{\tilde\zeta}},\tilde{\psi})<\infty$, we see that the last expression is summable over $k \in \Nmb$ and integrable over $s \in [t_1,t_1+\ttau]$.
	Therefore the second term in the claim converges to $0$ as $n \to \infty$ by the dominated convergence theorem.	
	From the above claim we then have that $\tilde{G}(\boldsymbol{\zeta}^{n_0},\psi^{n_0}) < \tilde{G}({\boldsymbol{\tilde\zeta}},\tilde{\psi})$ for sufficiently large $n_0$.
	We now fix such a $n_0$ and, abusing notation, denote 
	$(\boldsymbol{\zeta},\psi) = (\boldsymbol{\zeta}^{n_0},\psi^{n_0})$
	and define $(\boldsymbol{\zeta}^\varepsilon,\psi^\varepsilon)$ as before, by using the new definition of 
	$(\boldsymbol{\zeta},\psi)$, so that \eqref{eq:minimizer-finite-dim} holds.

	Since $(\boldsymbol{\zeta},\psi) \in \mathcal{J}^2_{t_1,t_1+\ttau}(\boldsymbol{x}^{(1)}, \boldsymbol{x}^{(2)})$, we have $r(\boldsymbol{\zeta}(t)) = x_0^{(1)} + \sum_{k=1}^\infty kp_k^{(1)} - 2(t-t_1)$ and $\zeta_0(t) = r(\boldsymbol{\zeta}(t)) - \sum_{k=1}^\infty k\zeta_k(t) = x_0^{(1)} + \sum_{k=1}^\infty k(p_k^{(1)}-\zeta_k(t)) - 2(t-t_1)$ for $t \in [t_1,t_1+\ttau]$.
	Using the definition of $L$, one can write
	\begin{align*}
		\tilde{G}(\boldsymbol{\zeta},\psi)
		& = \int_{t_1}^{t_1+\ttau} \left\{ \left(1+\sum_{k=1}^\infty \zeta'_k(t)\right) \log \left[ \left(1+\sum_{k=1}^\infty \zeta'_k(t)\right) \Big/ \left(\frac{\zeta_0(t)}{r(\boldsymbol{\zeta}(t))}\right) \right] \right. \\
		& \qquad \left. - \sum_{k=1}^\infty \zeta'_k(t) \log \left[ \left( -\zeta'_k(t) \right) \Big/ \left(\frac{k\zeta_k(t)}{r(\boldsymbol{\zeta}(t))}\right) \right] \right\} dt \\
		& = \int_{t_1}^{t_1+\ttau} \left\{ \left(1+\sum_{k=1}^\infty \zeta'_k(t)\right) \log \left(\frac{1+\sum_{k=1}^\infty \zeta'_k(t)}{x_0^{(1)} + \sum_{k=1}^\infty k(p_k^{(1)}-\zeta_k(t)) - 2(t-t_1)}\right) \right. \\
		& \qquad \left. - \sum_{k=1}^\infty \zeta'_k(t) \log \left(\frac{-\zeta'_k(t)}{k\zeta_k(t)}\right) \right\} dt + \int_{t_1}^{t_1+\ttau} \log \left( x_0^{(1)} + \sum_{k=1}^\infty kp_k^{(1)} - 2(t-t_1) \right) dt,
	\end{align*}	
	and the analogous expression holds for $\tilde{G}(\boldsymbol{\zeta}^\varepsilon,\psi^\varepsilon)$.
	Let $\boldsymbol{\theta} \doteq \boldsymbol{\zeta} - {\boldsymbol{\tilde\zeta}}$.
	From \eqref{eq:minimizer-finite-dim} we have $\theta_k=0$ for $k > n_0$ and hence
	\begin{align*}
		g(\varepsilon) & = \int_{t_1}^{t_1+\ttau} \left\{ \left(1+\sum_{k=1}^\infty (\zeta^\varepsilon_k)'(t)\right) \log \left(\frac{1+\sum_{k=1}^\infty (\zeta^\varepsilon_k)'(t)}{x_0^{(1)} + \sum_{k=1}^\infty k(p_k^{(1)}-\zeta_k^\varepsilon(t)) - 2(t-t_1)}\right) \right. \\
		& \qquad \left. - \sum_{k=1}^{n_0} (\zeta^\varepsilon_k)'(t) \log \left(\frac{-(\zeta^\varepsilon_k)'(t)}{k\zeta_k^\varepsilon(t)}\right) \right\} dt + C_0 \\
		& = \int_{t_1}^{t_1+\ttau} \eta(t,(\tilde{\zeta}_k(t) + \varepsilon \theta_k(t),\tilde{\zeta}_k'(t) + \varepsilon \theta_k'(t))_{k=1}^{n_0}) \, dt + C_0, \\
		& \doteq \int_{t_1}^{t_1+\ttau} \tilde{\eta}(t,\varepsilon) \, dt + C_0
	\end{align*}
	for some constant $C_0$, where
	\begin{align*}
		&\eta(t,(u_k,v_k)_{k=1}^{n_0})\\
		 &\quad = \left(1+\sum_{k=1}^{n_0} v_k +\alpha_t\right) \log \left(\frac{1+\sum_{k=1}^{n_0} v_k + \alpha_t}{x_0^{(1)} + \sum_{k=1}^{n_0} k(p_k^{(1)}-u_k) + \gamma_t - 2(t-t_1)}\right) - \sum_{k=1}^{n_0} v_k \log \left(\frac{-v_k}{ku_k}\right),
	\end{align*}
	with $\alpha_t \doteq \sum_{k=n_0+1}^{\infty} \tilde{\zeta}_k'(t)$ and
	$\gamma_t \doteq \sum_{k=n_0+1}^{\infty} k(p_k^{(1)}-\tilde{\zeta}_k(t))$.
	We wish to show that differentiation under the integral over $t$ with respect to $\varepsilon$ is valid in a neighborhood of $0$.
	For this, we now establish an integrable bound on the partial derivative of $\tilde{\eta}$ with respect to $\varepsilon$.
	To obtain such a bound, note that we only need to consider the contribution from $\varepsilon \theta_k(t)$ for $1 \le k \le n_0$ such that $p_k^{(2)}>0$, since when $p_k^{(2)} = 0$, one has that $p_k^{(1)} = 0$ by assumption (ii), which implies $\theta_k(t) \equiv 0$.
	Therefore assume without loss of generality that $p_k^{(2)}>0$ for every $1 \le k \le n_0$.
	Further note that we can assume $p_k^{(1)} > p_k^{(2)}$, since otherwise, once more,  $\theta_k(t) \equiv 0$.
	Therefore we assume without loss of generality that
	\begin{equation}
		\label{eq:p1p2strict}
		p_k^{(1)} > p_k^{(2)} > 0, \quad 1 \le k \le n_0.
	\end{equation} 
	Denote by $\frac{\partial \eta}{\partial u_k}$ and $\frac{\partial \eta}{\partial v_k}$ the corresponding partial derivatives for the function $\eta(t,(u_k,v_k)_{k=1}^{n_0})$.
	Then one can verify that
	\begin{equation*}
		\frac{\partial \tilde{\eta}(t,\varepsilon)}{\partial \varepsilon} = \sum_{k=1}^{n_0} \frac{\partial \eta}{\partial u_k} \vert_{(t,(\zeta_k^\varepsilon(t), (\zeta_k^\varepsilon)'(t))_{k=1}^{n_0})} \theta_k(t) + \sum_{k=1}^{n_0} \frac{\partial \eta}{\partial v_k} \vert_{(t,(\zeta_k^\varepsilon(t), (\zeta_k^\varepsilon)'(t))_{k=1}^{n_0})} \theta_k'(t).
	\end{equation*}
	The partial derivatives of $\eta$ are
	\small{\begin{align}
		& \frac{\partial \eta(t,(u_k,v_k)_{k=1}^{n_0})}{\partial u_k} = \frac{k(1+\sum_{j=1}^{n_0} v_j + \sum_{j=n_0+1}^\infty \tilde{\zeta}_j'(t))}{x_0^{(1)} + \sum_{j=1}^{n_0} j(p_j^{(1)}-u_j) + \sum_{j=n_0+1}^\infty j(p_j^{(1)}-\tilde{\zeta}_j(t)) - 2(t-t_1)} + \frac{v_k}{u_k},\label{eq:37.a} \\
		& \frac{\partial \eta(t,(u_k,v_k)_{k=1}^{n_0})}{\partial v_k} = \log \left(\frac{1+\sum_{j=1}^{n_0} v_j + \sum_{j=n_0+1}^\infty \tilde{\zeta}_j'(t)}{x_0^{(1)} + \sum_{j=1}^{n_0} j(p_j^{(1)}-u_j) + \sum_{j=n_0+1}^\infty j(p_j^{(1)}-\tilde{\zeta}_j(t)) - 2(t-t_1)}\right) - \log \frac{-v_k}{ku_k},		\label{eq:37.b}
	\end{align}}
	for $1 \le k \le n_0$.
	For all $0 \le \varepsilon < \quarter \wedge \delta \wedge \frac{\delta}{2 \tilde{\ttau}}$ and $t \in [t_1,t_1+\ttau]$, from \eqref{eq:p1p2strict} and \eqref{eq:deltadefn},
	\begin{align*}
		& 0 < \frac{\delta}{2} \le (1-\varepsilon) \zetatil_0(t) \le \zeta_0^\varepsilon(t) \le \sum_{k=1}^\infty kp_k < \infty, \\
		& \zeta_0^\varepsilon(t) = x_0^{(1)} + \sum_{j=1}^{n_0} j(p_j^{(1)}-\zeta^\varepsilon_j(t)) + \sum_{j=n_0+1}^\infty j(p_j^{(1)}-\tilde{\zeta}_j(t)) - 2(t-t_1), \\
		& 0 < p_k^{(2)} \le \zeta_k^\varepsilon(t) \le p_k^{(1)} < \infty, \:\: -1 \le (\zeta_k^\varepsilon)'(t) \le 0, \:\: |\theta_k(t)| \le p_k^{(1)}, \:\: |\theta_k'(t)| \le 2, \quad 1 \le k \le n_0, \\
		& 0 < \frac{\delta}{4\tilde{\ttau}} \le (1-\varepsilon)\left(1+\sum_{k=1}^\infty \zetatil_k'(t)\right) \le 1+\sum_{k=1}^\infty (\zeta^\varepsilon_k)'(t) \le 1,
	\end{align*}
	where the last line uses \eqref{eq:minimizer-verify-special-temp} and Lemma \ref{lem:I1I2L}.  Furthermore, using \eqref{eq:def-minimizer} we get
	\begin{align*}
		(\zeta_k^\varepsilon)'(t) \le (1-\varepsilon) \zetatil_k'(t) \le -\frac{3k\ztil_k}{8\tilde \ttau} \left( 1-\frac{t-t_1}{\tilde \ttau}\right)^{k/2-1} = -\frac{3k\ztil_k}{8\tilde \ttau^{k/2}} \left( \tilde\ttau-(t-t_1)\right)^{k/2-1}.
	\end{align*}
	Combining these bounds we have
	\begin{align*}
		\left| \frac{\partial \eta}{\partial u_k} \vert_{(t,(\zeta_k^\varepsilon(t), (\zeta_k^\varepsilon)'(t))_{k=1}^{n_0})} \right| & \le \frac{k}{\delta/2} + \frac{1}{p_k^{(2)}}, \\
		\left| \frac{\partial \eta}{\partial v_k} \vert_{(t,(\zeta_k^\varepsilon(t), (\zeta_k^\varepsilon)'(t))_{k=1}^{n_0})} \right| & \le \max\left\{\left|\log \frac{1}{\delta/2}\right|, \left|\log \frac{\delta/4\tilde\ttau}{x_0^{(1)}+\sum_{j=1}^\infty jp_j}\right| \right\} \\
		& \quad + \max\left\{\left|\log \frac{1}{kp_k^{(2)}}\right|, \left|\log \frac{\frac{3k\ztil_k}{8\tilde \ttau^{k/2}} \left( \tilde \ttau-(t-t_1)\right)^{k/2-1}}{kp_k^{(1)}}\right| \right\}
	\end{align*}
	for all $\varepsilon \in [0,1/4]$, $t \in [t_1,t_1+\ttau]$, and $k=1,\dotsc,n_0$.
	Therefore 
	one can find some $\tilde C_0 \in (0,\infty)$ such that
	\begin{equation*}
		\left| \frac{\partial \tilde{\eta}(t,\varepsilon)}{\partial \varepsilon} \right| \le \Ctil_0 + \Ctil_0 |\log \left( \tilde\ttau-(t-t_1)\right)|, \quad \varepsilon \in [0,1/4], \quad t \in [t_1,t_1+\ttau].
	\end{equation*}
	Since $|\log(\tilde\ttau - (t-t_1))|$ is integrable in $t \in [t_1,t_1+\ttau]$, we have obtained an integrable bound on $\left| \frac{\partial \tilde{\eta}(t,\varepsilon)}{\partial \varepsilon} \right|$ that is uniform in $\varepsilon \in [0, 1/4]$. Thus
	we can differentiate under the integral sign to get
	\begin{equation*}
		g'(\varepsilon) = \int_{t_1}^{t_1+\ttau} \frac{\partial \tilde{\eta}(t,\varepsilon)}{\partial \varepsilon} \, dt
	\end{equation*}
	for all $0 \le \varepsilon < \frac{1}{4} \wedge \delta \wedge \frac{\delta}{2 \tilde{\ttau}}$.
	Next we claim that the following Euler-Lagrange equations are satisfied.
	\begin{equation}
		\label{eq:Euler-Lagrange}
		\frac{\partial \eta}{\partial u_n} (t,(\tilde{\zeta}_k(t),\tilde{\zeta}'_k(t))_{k =1}^{n_0}) = \frac{d}{dt} \frac{\partial \eta}{\partial v_n} (t,(\tilde{\zeta}_k(t),\tilde{\zeta}'_k(t))_{k =1}^{n_0}) \mbox{ for } 1 \le n \le n_0, t \in [t_1,t_1+\ttau].
	\end{equation}	
	Once this claim is verified, we have
	\begin{align*}
		g'_+(0) & = \sum_{k=1}^{n_0} \int_{t_1}^{t_1+\ttau} \left[ \frac{\partial \eta}{\partial u_k} \vert_{(t,(\tilde{\zeta}_k(t), \tilde{\zeta}_k'(t))_{k=1}^{n_0})} \theta_k(t) + \frac{\partial \eta}{\partial v_k} \vert_{(t,(\tilde{\zeta}_k(t), \tilde{\zeta}_k'(t))_{k=1}^{n_0})} \theta_k'(t) \right] dt \\
		& = \sum_{k=1}^{n_0} \int_{t_1}^{t_1+\ttau} \theta_k'(t) \left[ - \int_{t_1}^t \frac{\partial \eta}{\partial u_k} \vert_{(s,(\tilde{\zeta}_k(s), \tilde{\zeta}_k'(s))_{k=1}^{n_0})} \, ds + \frac{\partial \eta}{\partial v_k} \vert_{(t,(\tilde{\zeta}_k(t), \tilde{\zeta}_k'(t))_{k=1}^{n_0})} \right] dt \\
		& = \sum_{k=1}^{n_0} \int_{t_1}^{t_1+\ttau} \tilde{c}_k \theta_k'(t) \, dt 
		 = \sum_{k=1}^{n_0} \tilde{c}_k (\theta_k(t_1+\ttau) - \theta_k(t_1)) = 0,
	\end{align*}
	where the second equality follows from integration by parts, the third is a consequence of \eqref{eq:Euler-Lagrange} with some suitable constants $\tilde c_k$ and the last equality holds since $\theta_k(t_1)=0=\theta_k(t_1+\ttau)$.
	This gives the desired contradiction and shows that $({\boldsymbol{\tilde\zeta}},\tilde{\psi})$ is the minimizer.
	
	Finally we prove the claim \eqref{eq:Euler-Lagrange}. Fix $1 \le n \le n_0$.
	Using \eqref{eq:37.a} and  \eqref{eq:Euler-Lagrange-key-2} one can verify that
	\begin{align*}
		\frac{\partial \eta}{\partial u_n} (t,(\tilde{\zeta}_k(t),\tilde{\zeta}'_k(t))_{k =1}^{n_0}) & = \frac{n(1 + \sum_{k=1}^\infty \tilde{\zeta}'_k(t))}{\tilde{\zeta}_0(t)} + \frac{\tilde{\zeta}_n'(t)}{\tilde{\zeta}_n(t)} \\
		& = \frac{n}{2\tilde{\ttau} -2(t-t_1)} + \frac{\tilde{\zeta}_n'(t)}{\tilde{\zeta}_n(t)} \\
		& = \frac{d}{dt} \left( -\frac{n}{2} \log(\tilde{\ttau}-(t-t_1)) + \log(\tilde{\zeta}_n(t)) \right).
	\end{align*}
	Therefore it suffices to show
	\begin{equation}
		\label{eq:Euler-Lagrange-simple}
		-\frac{n}{2} \log(\tilde{\ttau}-(t-t_1)) + \log(\tilde{\zeta}_n(t)) = \frac{\partial \eta}{\partial v_n} (t,(\tilde{\zeta}_k(t),\tilde{\zeta}'_k(t))_{k =1}^{n_0}) + \bar{c}_n
	\end{equation}
	for some constant $\bar{c}_n$.
	From \eqref{eq:37.b} one has that
	\begin{align*}
		\frac{\partial \eta}{\partial v_n} (t,(\tilde{\zeta}_k(t),\tilde{\zeta}'_k(t))_{k =1}^{n_0}) & = \log(n\tilde{\zeta}_n(t)) - \log(-\tilde{\zeta}_n'(t)) + \log\left(\frac{1 + \sum_{k=1}^\infty \tilde{\zeta}'_k(t)}{\tilde{\zeta}_0(t)}\right) \\
		& = \log(n\tilde{\zeta}_n(t)) - \log(-\tilde{\zeta}_n'(t)) - \log(2\tilde{\ttau} -2(t-t_1))
	\end{align*}
	where the last line follows from \eqref{eq:Euler-Lagrange-key-2}.
	From this we have
	\begin{align*}
		& -\frac{n}{2} \log(\tilde{\ttau}-(t-t_1)) + \log(\tilde{\zeta}_n(t)) - \frac{\partial \eta}{\partial v_n} (t,(\tilde{\zeta}_k(t),\tilde{\zeta}'_k(t))_{k =1}^{n_0}) \\
		& = - \left(\frac{n}{2}-1\right) \log(\tilde{\ttau}-(t-t_1)) - \log\left(\frac{n}{2}\right) + \log(-\tilde{\zeta}_n'(t)) \\
		& = \log \tilde{z}_n - \frac{n}{2} \log \tilde{\ttau}.
	\end{align*}
	where the last line follows from \eqref{eq:Euler-Lagrange-key-1} and \eqref{eq:def-minimizer}.
	Therefore \eqref{eq:Euler-Lagrange-simple} holds with $\bar{c}_n=\log \tilde{z}_n - \frac{n}{2} \log \tilde{\ttau}$
	which proves \eqref{eq:Euler-Lagrange}.
	This completes the proof.		
\end{proof}

\begin{proof}[Proof of Lemma \ref{lem:minimizer-verify-general}]
	The first equality in \eqref{eq:i2ttau} follows as before from  Lemma \ref{lem:I1I2L}.
	Lemma \ref{lem:minimizer-verify-special} shows that the second equality holds if additional two assumptions in Lemma \ref{lem:minimizer-verify-special} are satisfied.
	Let $(\boldsymbol{\zeta},\psi) \in \mathcal{J}^2_{t_1,t_1+\ttau}(\boldsymbol{x}^{(1)}, \boldsymbol{x}^{(2)})$ be a  trajectory such that $\int_{t_1}^{t_1+\ttau} L(\boldsymbol{\zeta}(s), \boldsymbol{\zeta}'(s))\,ds \le \int_{t_1}^{t_1+\ttau} L({\boldsymbol{\tilde\zeta}}(s), {\boldsymbol{\tilde\zeta}}'(s))\,ds$.
	It suffices to show 
	\begin{equation}\int_{t_1}^{t_1+\ttau} L(\boldsymbol{\zeta}(s), \boldsymbol{\zeta}'(s))\,ds \ge \int_{t_1}^{t_1+\ttau} L({\boldsymbol{\tilde\zeta}}(s), {\boldsymbol{\tilde\zeta}}'(s))\,ds.\label{eq:showineq}
	\end{equation}
	We claim that we can assume 
	\begin{itemize}
	\item
		$\zeta_0(t) > 0$ for all $t \in (t_1,t_1+\ttau)$,
	\item
		if $\new_k>0$ for some $k \in \mathbb{N}$, then $\zeta_k(t) > 0$ for all $t \in (t_1,t_1+\ttau)$.
	\end{itemize}
	For this, note that ${\boldsymbol{\tilde\zeta}}$ satisfies these two properties.
	Letting $(\boldsymbol{\zeta}^\varepsilon,\psi^\varepsilon) \doteq \varepsilon (\boldsymbol{\zeta},\psi) + (1-\varepsilon) ({\boldsymbol{\tilde\zeta}},\tilde{\psi})$ for $\varepsilon \in (0,1)$ we have that $(\boldsymbol{\zeta}^\varepsilon,\psi^\varepsilon) \in \mathcal{J}^2_{t_1,t_1+\ttau}(\boldsymbol{x}^{(1)}, \boldsymbol{x}^{(2)})$ and it satisfies the two claimed properties.
	Also, from the convexity of $L$ we see that, it suffices to prove \eqref{eq:showineq} with $(\boldsymbol{\zeta},\psi)$ 
	replaced with $(\boldsymbol{\zeta}^\varepsilon,\psi^\varepsilon)$. 
	Therefore the claim holds.
	
	Fix two sequences of time instants $t_1^{(n)} \doteq t_1 + \frac{1}{n}$ and $t_2^{(n)} \doteq t_1+\ttau - \frac{1}{n}$.
	Note that $t_2^{(n)} = t_1^{(n)} + \ttau^{(n)}$ where $\ttau^{(n)}$ is defined by \eqref{eq:tau} by replacing
	$(\boldsymbol{x}^{(1)},\boldsymbol{x}^{(2)})$ with $(\boldsymbol{x}^{(1),n},\boldsymbol{x}^{(2),n}) =
	(\boldsymbol{\zeta}(t_1^{(n)}), \boldsymbol{\zeta}(t_2^{(n)}))$.
	Consider now the optimization problem in \eqref{eq:def-I-j} associated with
	$I^2_{t_1^{(n)},t_1^{(n)} + \ttau^{(n)}}(\boldsymbol{x}^{(1),n},\boldsymbol{x}^{(2),n})$.
	Note that for this problem the  two additional assumptions in Lemma \ref{lem:minimizer-verify-special} are satisfied.
	Furthermore, the assumption $\sum_{k=1}^\infty k \new_k + \new_0 > 2 \sum_{k=1}^\infty \new_k$ in Lemma \ref{lem:minimizer-verify-general} also holds with $\boldsymbol{z}$ replaced by $\boldsymbol{z}^{(n)} =
\boldsymbol{x}^{(1),n}-\boldsymbol{x}^{(2),n}$, for sufficiently large $n$.
	Therefore Lemma \ref{lem:minimizer-verify-special} can be applied with $(\boldsymbol{x}^{(1)},\boldsymbol{x}^{(2)})$ replaced with $(\boldsymbol{x}^{(1),n},\boldsymbol{x}^{(2),n})$.
	Let $({\boldsymbol{\tilde\zeta}}^{(n)},\tilde{\psi}^{(n)}) \in \mathcal{J}^2_{t_1^{(n)},t_1^{(n)} + \ttau^{(n)}}(\boldsymbol{x}^{(1),n},\boldsymbol{x}^{(2),n})$ 
	be the corresponding minimizer and $\beta^{(n)} \doteq \beta(\boldsymbol{x}^{(1),n},\boldsymbol{x}^{(2),n})$.
	Then
	\begin{align*}
		 \int_{t_1}^{t_1+\ttau} L(\boldsymbol{\zeta}(s), \boldsymbol{\zeta}'(s))\,ds 
		& = \lim_{n \to \infty} \int_{t_1^{(n)}}^{t_2^{(n)}} L(\boldsymbol{\zeta}(s), \boldsymbol{\zeta}'(s))\,ds \\
		& \ge \liminf_{n \to \infty} \int_{t_1^{(n)}}^{t_2^{(n)}} L({\boldsymbol{\tilde\zeta}}^{(n)}(s), ({\boldsymbol{\tilde\zeta}}^{(n)})'(s))\,ds \\
		& = \liminf_{n \to \infty} [\tilde{H}(\boldsymbol{z}^{n}) + \tilde{H}(\boldsymbol{x}^{(2),n}) - \tilde{H}(\boldsymbol{x}^{(1),n}) + \tilde{K}(\boldsymbol{x}^{(1),n}, \boldsymbol{x}^{(2),n})] \\
		& \ge \tilde{H}(\boldsymbol{\new}) + \tilde{H}(\boldsymbol{x}^{(2)}) - \tilde{H}(\boldsymbol{x}^{(1)}) + \tilde{K}(\boldsymbol{x}^{(1)}, \boldsymbol{x}^{(2)}) \\
		& = \int_{t_1}^{t_1+\ttau} L({\boldsymbol{\tilde\zeta}}(s), {\boldsymbol{\tilde\zeta}}'(s))\,ds.
	\end{align*}
	Here the first inequality follows from Lemma \ref{lem:minimizer-verify-special} and the last three lines use Lemma \ref{lem:minimizer-cost}.
\end{proof}

\section{Proof of LLN}

\label{sec:examples} In this section we give the proofs of Theorem \ref{thm:LLN} and Proposition \ref{prop:uniqueness_LLN}.

\noindent {\em Proof of Theorem \ref{thm:LLN} }

	(1)
	Assume without loss of generality that $T\ge 1$. Since $f_1(t) \le 1$, we see from Assumption \ref{asp:exponential-boundN-S} that $r(\boldsymbol{\zeta}(\cdot))$ with $r$ from (\ref{eq:r_k}) and $\psi$ are well-defined.	
	Let $\varphi_k(s,y) = 1$ for all $k \in \mathbb{N}_0$ and $(s,y) \in [0,T] \times [0,1]$.
	It suffices to show $\boldsymbol{\varphi} \in \mathcal{S}_T(\boldsymbol{\zeta} ,\psi )$ and $(\boldsymbol{\zeta} ,\psi ) \in \mathcal{C}_T$.
	Since $f_1(t)=F^{-1}_1(t) {{1}}_{[0,1]}(t)$, we have $\tau_{\boldsymbol{\zeta}}=1$, where $\tau_{\boldsymbol{\zeta}}$ was defined in \eqref{eq:defntauphi}.
	Since $F_1(f_1(t))=t$ for $t\in [0,1]$,  
	\begin{equation*}
		f'_1(t) = - \frac{1}{\sum_{k=1}^\infty k p_k (f_1(t))^{k-1}} \mbox{ for } 0 < t < \tau_{\boldsymbol{\zeta}} \mbox{ and } f'_1(t)=0 \mbox{ for } \tau_{\boldsymbol{\zeta}} < t <T.
	\end{equation*} 
	Using this it follows that for $k \in \mathbb{N}$,
	\begin{equation*}
		\zeta'_k(t) = - \frac{k \zeta_k(t)}{\sum_{j=1}^\infty j\zeta_j(t)} = - r_k(\boldsymbol{\zeta}(t)) \mbox{ for } 0 < t < \tau_{\boldsymbol{\zeta}} \mbox{ and } \zeta'_k(t)=0 \mbox{ for } \tau_{\boldsymbol{\zeta}} < t <T.
	\end{equation*}
	From this we see that \eqref{eq:phi_k} holds and we can write
	\begin{equation*}
		\psi(t) = \sum_{k=0}^\infty (k-2) \int_0^t r_k(\boldsymbol{\zeta}(s))\,ds.
	\end{equation*}
	This gives \eqref{eq:psi} and verifies that $\boldsymbol{\varphi} \in \mathcal{S}_T(\boldsymbol{\zeta} ,\psi )$.
	
	Next we argue that $(\boldsymbol{\zeta} ,\psi ) \in \mathcal{C}_T$.
	From Assumption \ref{asp:exponential-boundN-S},  for $t < \tau_{\boldsymbol{\zeta}}$, as $K \to \infty$,
	\begin{equation*}
		\sum_{k=K}^\infty |k-2| |\zeta'_k(t)| \le \sum_{k=K}^\infty k r_k(\boldsymbol{\zeta}(t)) \le \frac{\sum_{k=K}^\infty k^2 p_k}{r(\boldsymbol{\zeta}(t))} \to 0. 
	\end{equation*}
	In particular, $\psi$ is absolutely continuous and thus property (a) of $\mathcal{C}_T$ holds.
	Also, for $t < \tau_{\boldsymbol{\zeta}}$,
	\begin{equation*}
		\psi'(t) = \sum_{k=1}^\infty (k-2) r_k(\boldsymbol{\zeta}(t)) = \frac{\sum_{k=1}^\infty k(k-2)p_k (f_1(t))^k}{r(\boldsymbol{\zeta}(t))} \le \frac{f_1(t) \sum_{k=1}^\infty k(k-2)p_k}{r(\boldsymbol{\zeta}(t))} \le 0.
	\end{equation*}	
	Therefore $\Gamma(\psi)(t)=0=\zeta_0(t)$ for $t < \tau_{\boldsymbol{\zeta}}$.
	For $\tau_{\boldsymbol{\zeta}} \le t \le T$, clearly $\Gamma(\psi)(t)=0=\zeta_0(t)$.
	So we have checked property (b) of $\mathcal{C}_T$.
	Property (c) of $\mathcal{C}_T$ follows from the definition of $\zeta_k$, $k \in \mathbb{N}$.
	Therefore $(\boldsymbol{\zeta} ,\psi ) \in \mathcal{C}_T$ and part (1) follows.

	(2) The fact that when $p_1>0$ there is a unique $\rho \in (0,1)$ such that $G_1(\rho)=\rho$ is proved in \cite{molloy1995critical}.
	Since $f_\rho(t) \le 1$, we see from Assumption \ref{asp:exponential-boundN-S} that $r(\boldsymbol{\zeta}(\cdot))$ and $\psi$ are well-defined.
	Let $\varphi_k(s,y) = 1$ for all $k \in \mathbb{N}_0$ and $(s,y) \in [0,T] \times [0,1]$.
	It suffices to show $\boldsymbol{\varphi} \in \mathcal{S}_T(\boldsymbol{\zeta} ,\psi )$ and $(\boldsymbol{\zeta} ,\psi ) \in \mathcal{C}_T$.
	First consider times $t < \tau$.
	Using the definitions of $r$, $G_1$ and $\tau$, for $t < \tau$
	\begin{equation*}
		r(\boldsymbol{\zeta}(t)) = \mu-2t -\mu \sqrt{1-2t/\mu}G_1 ( \sqrt{1-2t/\mu}) + \sum_{k=1}^\infty k p_k (1 - 2t/\mu)^{k/2} = \mu-2t > \mu \rho^2 \ge 0.
	\end{equation*}
	From this one can verify that for $t < \tau$, 
	\begin{align*}
		\zeta'_k(t) & = - \frac{k \zeta_k(t)}{\mu-2t} = -r_k(\boldsymbol{\zeta}(t)).
	\end{align*}
	Using this we see that \eqref{eq:phi_k} holds for $t < \tau$ and hence as before \eqref{eq:psi} holds as well.
	To show that $(\boldsymbol{\zeta} ,\psi ) \in \mathcal{C}_t$ for $t<\tau$, it suffices to show that $\psi(t)$ is absolutely continuous and $\zeta_0(t) = \psi(t)$ for $t \in [0,\tau)$.
	Note that for $t < \tau$, $\sum_{k=1}^\infty |k-2| |r_k(\boldsymbol{\zeta}(t))| \le \frac{\sum_{k=1}^\infty k^2 p_k}{\mu-2t}$.
	So from Assumption \ref{asp:exponential-boundN-S}, $\psi$ is absolutely continuous over $[0,\tau]$.
	Also, one can verify that for $t < \tau$,
	\begin{equation*}
		\zeta'_0(t) = \frac{d}{dt} r(\boldsymbol{\zeta}(t)) - \sum_{k=1}^\infty k \zeta'_k(t) = - 2 + \sum_{k=1}^\infty k r_k(\boldsymbol{\zeta}(t)) = \psi'(t).
	\end{equation*}
	So $\zeta_0(t) = \psi(t)$ for $t < \tau$.
	Thus we have  that $\boldsymbol{\varphi} \in \mathcal{S}_t(\boldsymbol{\zeta} ,\psi )$ and $(\boldsymbol{\zeta} ,\psi ) \in \mathcal{C}_t$ for each $t<\tau$.
	
 
	We now consider $t \in [\tau, \tau_{\boldsymbol{\zeta}}]$. 
	Since $\rho \in [0,1)$ and $G_1(\rho)=\rho$, we have
	\begin{align*}
		0 & = \frac{\mu(G_1(\rho)-\rho)}{\rho-1} = \frac{1}{\rho-1} \sum_{k=1}^\infty kp_k (\rho^{k-1}-\rho) = -p_1 + \rho \sum_{k=3}^\infty kp_k \frac{\rho^{k-2}-1}{\rho-1} \\
		& = -p_1 + \rho \sum_{k=3}^\infty kp_k (\rho^{k-3}+\rho^{k-4}+\dotsb+1) \\
		& \ge -p_1 + \rho \sum_{k=3}^\infty kp_k (k-2)\rho^{k-3} 
		 \ge \sum_{k=1}^\infty k(k-2)p_k\rho^{k-1}
	\end{align*}
	and therefore $0 \ge \sum_{k=1}^\infty k(k-2)p_k\rho^k = \sum_{k=1}^\infty k(k-2) \zeta_k(\tau)$.
	Namely, the assumption in part (1) is satisfied with ${\boldsymbol{p}}$ replaced by $\boldsymbol{\zeta}(\tau)$.
	Thus the proof for the case $t \in [\tau, \tau_{\boldsymbol{\zeta}}]$ is very similar to that in part (1), with $f_1(t)$ replaced by $f_{\rho}(t-\tau)$ and $p_k$ replaced with $\zeta_k(\tau)$, and we would like to omit the detail.
	This completes the proof of (2). \hfill \qed \\

\noindent {\em Proof of Proposition \ref{prop:uniqueness_LLN}.}
	Suppose for $i=1,2$, $(\boldsymbol{\zeta}^{(i)},\psi^{(i)})$ are two pairs such that $I_T(\boldsymbol{\zeta}^{(i)},\psi^{(i)}) = 0$.
	By the definition of $I_T(\cdot)$, $(\boldsymbol{\zeta}^{(i)},\psi^{(i)}) \in \mathcal{C}_T$.
	\blue{From Remark \ref{rmk:unique_varphi} we see that there exists some $\boldsymbol{\varphi}^{(i)} \in \mathcal{S}_T(\boldsymbol{\zeta}^{(i)},\psi^{(i)})$ whose cost equals $I_T(\boldsymbol{\zeta}^{(i)},\psi^{(i)})$, namely}
	\begin{equation*}
		\sum_{k=0}^\infty \int_{[0,T] \times [0,1]} \ell(\varphi_k^{(i)}(s,y)) \, ds\,dy =I_T(\boldsymbol{\zeta}^{(i)},\psi^{(i)})= 0.
	\end{equation*}	
	Since $\ell(x) = 0$ if and only if $x=1$, we must have $\varphi_k^{(i)}(s,y) = 1$ for a.e.\ $(s,y) \in [0,T]\times[0,1]$ and $k \in \mathbb{N}_0$. 
	Using such $\varphi^{(i)}$ with \eqref{eq:psi} and \eqref{eq:phi_k}, we see that
	\begin{align}
		\zeta_k^{(i)}(t) & = p_k - \int_0^t  r_k(\boldsymbol{\zeta}^{(i)}(s))\,ds, k \in \mathbb{N}, \label{eq:phi_unique_LLN} \\
		\psi^{(i)}(t) & = \sum_{k=0}^\infty (k-2) \int_0^t  r_k(\boldsymbol{\zeta}^{(i)}(s))\,ds. \label{eq:psi_unique_LLN}
	\end{align}
	Since $\zeta_0^{(i)} = \Gamma(\psi^{(i)})$,  for a.e.\ $t$,  $(\zeta_0^{(i)})'(t) \ge (\psi^{(i)})'(t) = \sum_{k=0}^\infty (k-2)  r_k(\boldsymbol{\zeta}^{(i)}(t))$, and by (\ref{eq:r_k})
	\begin{align*}
		\frac{d}{dt} r(\boldsymbol{\zeta}^{(i)}(t)) & = (\zeta_0^{(i)})'(t) + \sum_{k=1}^\infty k (\zeta_k^{(i)})'(t) \ge \sum_{k=0}^\infty (k-2)  r_k(\boldsymbol{\zeta}^{(i)}(t)) - \sum_{k=1}^\infty k  r_k(\boldsymbol{\zeta}^{(i)}(t)) \\
		& = -2\cdot{{1}}_{\{r(\boldsymbol{\zeta}^{(i)}(t))>0\}} \ge -2.
	\end{align*}
	Consider the strictly increasing function $g^{(i)}(t)$ defined by 
	\begin{equation}
		\label{eq:g_unique_LLN}
		g^{(i)}(0) = 0, \: (g^{(i)})'(t) = r(\boldsymbol{\zeta}^{(i)}(g^{(i)}(t))) {{1}}_{\{g^{(i)}(t) < \tau_{\boldsymbol{\zeta}^{(i)}}\}} + {{1}}_{\{g^{(i)}(t) \ge \tau_{\boldsymbol{\zeta}^{(i)}}\}},
	\end{equation}
	where $\tau_{\boldsymbol{\zeta}^{(i)}}$ is as in \eqref{eq:defntauphi}. 
	Since $\frac{d}{dt} r(\boldsymbol{\zeta}^{(i)}(t)) \in [-2,0]$ and $0 \le r(\boldsymbol{\zeta}^{(i)}(\cdot)) \le r(\boldsymbol{\zeta}^{(i)}(0)) = \sum_{k=1}^\infty kp_k < \infty$, we see that $r(\boldsymbol{\zeta}^{(i)}(\cdot))$ is bounded and Lipschitz.
	Also $r(\boldsymbol{\zeta}^{(i)}(t)) > 0$ for $t < \tau_{\boldsymbol{\zeta}^{(i)}}$.
	So we have existence and uniqueness of the strictly increasing function $g^{(i)}(t)$ before it reaches $\tau_{\boldsymbol{\zeta}^{(i)}}$.
	The existence, uniqueness and monotonicity of $g^{(i)}(t)$ after  $\tau_{\boldsymbol{\zeta}^{(i)}}$ is straightforward.
	
	Define $({\tilde{\boldsymbol{\zeta}}}^{(i)}(t),{\tilde{\psi}}^{(i)}(t)) \doteq (\boldsymbol{\zeta}^{(i)}(g^{(i)}(t)),\psi^{(i)}(g^{(i)}(t)))$.
	From \eqref{eq:phi_unique_LLN} and \eqref{eq:psi_unique_LLN} it follows that
	\[
	\Scale[0.9]{\begin{aligned}
		{\tilde{\zeta}}_k^{(i)}(t) & = p_k - \int_0^t  k{\tilde{\zeta}}_k^{(i)}(s)\,ds, k \in \mathbb{N}, \\
		{\tilde{\psi}}^{(i)}(t) & = \sum_{k=1}^\infty (k-2) \int_0^t  k{\tilde{\zeta}}_k^{(i)}(s)\,ds - 2 \int_0^t  {\tilde{\zeta}}_0^{(i)}(s)\,ds 
		 = \sum_{k=1}^\infty (k-2) \int_0^t  k{\tilde{\zeta}}_k^{(i)}(s)\,ds - 2 \int_0^t  \Gamma({\tilde{\psi}}^{(i)})(s)\,ds.	
	\end{aligned}}
	\]
	Clearly ${\tilde{\zeta}}_k^{(1)} = {\tilde{\zeta}}_k^{(2)}$ for each $k \in \mathbb{N}$.
	Also, since $\Gamma$ is Lipschitz on path space, Gronwall's inequality implies ${\tilde{\psi}}^{(1)} = {\tilde{\psi}}^{(2)}$, and hence ${\tilde{\zeta}}_0^{(1)} = {\tilde{\zeta}}_0^{(2)}$.
	Noting that \eqref{eq:g_unique_LLN} can be written as
	\begin{equation*}
		g^{(i)}(0) = 0, \: (g^{(i)})'(t) = r(\tilde{\boldsymbol{\zeta}}^{(i)}(t)) {{1}}_{\{r(\tilde{\boldsymbol{\zeta}}^{(i)}(t)) > 0\}} + {{1}}_{\{r(\tilde{\boldsymbol{\zeta}}^{(i)}(t)) = 0\}},
	\end{equation*}
	we have $g^{(1)} = g^{(2)}$.
	Since $g^{(i)}$ is strictly increasing, its inverse function is well-defined and we must have that $(\boldsymbol{\zeta}^{(1)},\psi^{(1)})=(\boldsymbol{\zeta}^{(2)},\psi^{(2)})$.
	This completes the proof. \hfill \qed

\vspace{\baselineskip}\noindent \textbf{Acknowledgement:} 
We would like to thank the two referees for a careful review of our work and for the many helpful suggestions.
The research of SB was supported in part by the NSF (DMS-1606839, DMS-1613072)  and the Army Research Office (W911NF-17-1-0010).
The research of AB was supported in part by the NSF (DMS-1305120, DMS-1814894, DMS-1853968).
The research of PD was supported in part by the NSF (DMS-1904992) and DARPA (W911NF-15-2-0122). 
The research of RW was supported in part by DARPA (W911NF-15-2-0122).

\bibliographystyle{plain}

\vspace{\baselineskip}

\scriptsize{\textsc{\noindent S. Bhamidi and A. Budhiraja\newline
Department of Statistics and Operations Research\newline
University of North Carolina\newline
Chapel Hill, NC 27599, USA\newline
email: bhamidi@email.unc.edu, budhiraj@email.unc.edu \vspace{\baselineskip} }

\textsc{\noindent P. Dupuis \newline
Division of Applied Mathematics\newline
Brown University\newline
Providence, RI 02912, USA\newline
email: paul\_dupuis@brown.edu \vspace{\baselineskip} }

\textsc{\noindent R. Wu\newline
Department of Mathematics\newline
University of Michigan\newline
Ann Arbor, MI 48109, USA\newline
email: ruoyu@umich.edu }}

\end{document}